\documentclass[reqno]{amsart}
\usepackage[a4paper,portrait,margin=2.5cm]{geometry}

\usepackage[T1]{fontenc}
\usepackage[utf8]{inputenc}

\usepackage{color}
\usepackage[colorlinks=true]{hyperref}
\hypersetup{urlcolor=blue, citecolor=blue}

\usepackage{amsmath,amssymb,amsfonts,amsthm} 
\usepackage{graphicx}
\usepackage{pgfplots}
\usepackage{placeins}
\usepackage{ulem}
\usepackage{cleveref}
\usepackage{thmtools}
\usepackage{todonotes}
\usepackage{stmaryrd}
\usepackage{mathtools}
\usepackage{mathrsfs}  
\usepackage{enumerate}

\numberwithin{equation}{section}

\makeatletter
\renewcommand{\paragraph}{\@startsection{paragraph}{4}{0ex}%
    {-3.25ex plus -1ex minus -0.2ex}%
    {1.5ex plus 0.2ex}%
    {\normalfont\normalsize}}
\makeatother

\usepackage{etoolbox}
\appto\appendix{\addtocontents{toc}{\protect\setcounter{tocdepth}{1}} \newpage}

\definecolor{centeredcolor}{RGB}{191,0,191}
\definecolor{activitycolor}{RGB}{0,128,0}
\definecolor{BCHcolor}{RGB}{255,0,0}
\definecolor{sedancolor}{RGB}{0,0,255}

\def\ICI{\begin{center}
{\color{magenta}
==================================\\
ICI\\
==================================}
\end{center}}

\newcounter{cst}

\def\wt#1{{\widetilde{#1}}}

\newcommand{\RR}{\mathbb R}

\newcommand{\xVV}[1]{\mathbf #1}
\newcommand{\xCC}[1]{\ifmmode \mathcal{#1} \else $\mathcal{#1}$ \fi}

\newcommand{\cC}{\xCC C}
\newcommand{\cc}{\mathscr{C}}

\newcommand{\cE}{\xCC E}

\newcommand{\cH}{\xCC H}

\newcommand{\cR}{\xCC R}

\newcommand{\cV}{\xCC V}

\newcommand{\bJ}{\xVV J}

\DeclareMathOperator{\dive}{div}

\def\ds{\displaystyle}

\def\bvarphi{\boldsymbol{\varphi}}
\def\bpsi{\boldsymbol{\psi}}

\def\p{\partial}
\def\R{\mathbb{R}}
\def\O{\Omega}
\def\be{\begin{equation}}
\def\ee{\end{equation}}
\def\n{\boldsymbol{n}}
\def\x{\boldsymbol{x}}
\def\d{{\rm d}}
\def\grad{\nabla}
\def\div{\dive}

\def\ov#1{\overline{#1}}
\def\wh#1{\widehat{#1}}

\def\Gg{\mathcal{G}}

\def\bgamma{{\boldsymbol \gamma}}
\def\eps{\epsilon}
\def\Dd{\mathcal{D}}
\def\Cc{\mathcal{C}}
\def\Ff{\mathcal{F}}
\def\Hh{\mathcal{H}}
\def\Tt{\mathcal{T}}

\def\Ee{\mathcal{E}}

\def\sig{\sigma}

\def\dt{{\Delta t}}
\def\bc{{\boldsymbol c}}
\def\bu{{\boldsymbol u}}
\def\bU{{\boldsymbol U}}
\def\bV{{\boldsymbol V}}
\def\bv{{\boldsymbol v}}
\def\bw{{\boldsymbol w}}
\def\bPhi{{\boldsymbol \Phi}}
\def\bdt{{\boldsymbol{\Delta t}}}

\def\bbL{\mathbb{L}}

\newcommand\tendsto[2]{{\xrightarrow[#2]{}#1}}

\theoremstyle{plain}
\newtheorem{theorem}{Theorem}[section]
\newtheorem{lemma}{Lemma}[section]
\newtheorem{prop}{Proposition}[section]
\newtheorem{definition}{Definition}

\theoremstyle{remark}
\newtheorem{remark}{Remark}[section]

\declaretheoremstyle[
headfont=\normalfont\bfseries, bodyfont=\normalfont, postheadspace=0.3em, headpunct=: ]{hypstyle}

\crefname{hyp}{\bf{H}}{\bf{H}}
\Crefname{hyp}{Hypothesis}{Hypotheses}



\begin{document}
\title[Finite Volume schemes for ionic liquid]{A numerical analysis focused comparison of several Finite Volume schemes for an
Unipolar Degenerated Drift-Diffusion Model}

\author[C. Canc\`es]{Cl\'ement Canc\`es}\address{Cl\'ement Canc\`es (\href{mailto:clement.cances@inria.fr}{\tt clement.cances@inria.fr}):  Inria, Univ. Lille, CNRS, UMR 8524 - Laboratoire Paul Painlev\'e, F-59000 Lille} 
\author[C. Chainais-Hillairet]{Claire Chainais-Hillairet}\address{Claire Chainais-Hillairet (\href{mailto:claire.chanais@univ-lille.fr}{\tt claire.chanais@univ-lille.fr}): 
Univ. Lille,  CNRS, UMR 8524, Inria - Laboratoire Paul Painlev\'e, F-59000 Lille} 
\author[J. Fuhrmann]{J\"urgen Fuhrmann}\address{J\"urgen Fuhrmann (\href{mailto:juergen.fuhrmann@wias-berlin.de}{\tt juergen.fuhrmann@wias-berlin.de}): Weierstrass Institute (WIAS), Mohrenstr. 39, 10117 Berlin, Germany}
\author[B. Gaudeul]{Beno\^{\i}t Gaudeul}\address{Beno\^\i t Gaudeul (\href{mailto:benoit.gaudeul@univ-lille.fr}{\tt benoit.gaudeul@univ-lille.fr}): 
Univ. Lille,  CNRS, UMR 8524, Inria - Laboratoire Paul Painlev\'e, F-59000 Lille} 
\maketitle


\begin{abstract}
In this paper, we consider an  unipolar degenerated drift-diffusion system where the relation between the concentration of the charged species $c$ and the chemical potential $h$ is $h(c)=\log \frac{c}{1-c}$. 
We design four different finite volume schemes based on four different formulations of the fluxes. We provide a stability analysis and existence results for the four schemes. 
The convergence proof with respect to the discretization parameters is established for two of them. Numerical experiments illustrate the behaviour of the different schemes. 

\end{abstract}

\section{Introduction}

\subsection{Motivation}

Unipolar drift-diffusion  models describe  the transport of  a charged
species in the presence of  a fixed or moving countercharge. They consist of the coupling of a drift-diffusion equation 
on the density of the charged species $c$ with a Poisson equation on the electric potential $\Phi$. They can be written under a general form 
as 
$$
\left\{
\begin{aligned}
&\partial_t c +\dive\left(\bJ\right)=0, \quad \bJ= -\eta(c) \nabla (h(c)+\Phi),\\
& -\lambda^2 \Delta\Phi= c +c^{dp},
\end{aligned}
\right.
$$
where $h$ is the chemical potential, $\eta$ the mobility coefficient, $\lambda$ the scaled Debye length coming from the nondimensionalization of the physical model and 
$c^{dp}$ describes the doping profile of the media.

Such models occur in a number of interesting application cases.
Charge carriers  in most classical  semiconductors exhibit a
relationship  $c=\mathcal  F(h)$,  where  $\mathcal F$  is  the  Fermi
integral  of  index  $\frac12$  which   can be approximated in the range
$-\infty < h \lessapprox 1.3$ by  the 
function  $\mathcal F(h)= \frac1{\gamma+ \exp(-h)}$ with $\gamma=0.27$
\cite{blakemore1982approximations}.
For $\gamma=1$, this relationship  corresponds to the Fermi integral of index
-1 and implies $h=\log \frac{c}{1-c}$. It is  the limit for vanishing disorder
of               the                Gauss-Fermi               integral
\cite{koprucki2013discretization,paasch2010charge}   which  is   used  to
describe organic semiconductors \cite{coehoorn2005charge}.
A similar relationship is valid for the oxygen ion concentration in a solid oxide electrolyte
\cite{VagnerEtAl2019} and in a simple model of an ionic liquid \cite{Fuhrmann16}.

While the  relationship between chemical potential and concentration
is sufficient to  describe the thermodynamic equilibrium, the description
of charge transport driven by the sum  of the  gradients of the chemical potential and the electrostatic
potential $\Phi$ needs an additional specification of the mobility coefficient $\eta$. Setting this
coefficient proportional to the concentration $c$ is common in the case of semiconductors \cite{selberherr2012analysis}.
A similar ansatz describes  the limit of large lattice mass density in solid oxide electrolytes.
It also follows from a formal reduction of an generalized Nernst-Planck model \cite{dreyer2013overcoming,dreyer2014mixture}
to the case  of a mixture of two charged species  including an infinitely mobile and charged solvent -- ionic liquids -- as performed in \cite{Fuhrmann16}.
We hint that more general  and fully consistent models  for both
solid  oxide   electrolytes  and   ionic  liquids   consider  mobility
coefficients          of           the          type          $c(1-c)$ \cite{VagnerEtAl2019,brochard1983polymer,gavish2016theory}.

In this paper, we consider that the mobility coefficient is $\eta(c)=c$ and the chemical potential $h(c)=\log  \frac{c}{1-c}$ (corresponding to 
 $\mathcal F(h)= \frac1{1+ \exp(-h)}$). Strong degeneration
described  by a  bounded dependency  of the  concentration $c$  on the
chemical potential  $h$ leads to  a number of  structural mathematical
challenges in the corresponding drift-diffusion models.  These need to
be addressed properly in numerical schemes. The consideration of this simplified model is a starting point for the study
of generalized Nernst-Planck models  for multiple ionic species in electroneutral solvents \cite{dreyer2013overcoming,dreyer2014mixture,Fuhrmann15,Fuhrmann16}. 
Moreover, the design of discretization methods
for the case where $\eta(c)=c(1-c)$  is also possible topic of further investigation following the present
paper.

\subsection{A simplified unipolar degenerate drift-diffusion model}
Let us now define the framework of the study. 
We consider the evolution of a the concentration $c$  of a charged species in a connected bounded open domain $\O$ of $\R^d$ ($d\leq 3$) with polyhedral  and Lipschitz continuous boundary $\p\O$ during a finite but arbitrary time $T>0$.
After nondimensionalization with appropriate scaling, we regard the following system of partial differential equations (PDEs). The concentration $c$ satisfies the conservation law
\be\label{eq:cons_loc} 
\partial_t c +\dive\left(\bJ\right)=0 \quad \text{in}\; (0,T) \times \O.
\ee
The flux $\bJ$ is negatively proportional to the gradient of the electrochemical potential as expressed by the expression 
\be\label{cances flux}
\bJ=-c\nabla\left(h(c)+\Phi \right) \quad \text{in}\; (0,T) \times \O, 
\ee	
where $h(c) = \log\left(\frac{c}{1-c}\right)$ is the chemical potential. In what follows, we consider that the electrostatic potential $\Phi$ is related to
space charge density
thanks to the Poisson equation
\be\label{Poisson equation}
-\Delta \Phi=c+c^{dp}  \quad \text{in}\; (0,T) \times \O,
\ee
which means that the Debye length is set to 1. Extension to general Debye length is straightforward.
The doping profile $c^{dp}$ is assumed to be constant w.r.t. time and to be bounded, i.e., $c^{dp} \in L^\infty(\O)$.

One interpretation of $c$ is the concentration of majority carriers (holes)  in a p-type organic semiconductor with constant in time doping. 
Another interpretation of $c$ is the cation concentration in an ionic liquid following the formal approach introduced in \cite{Fuhrmann16}.

The system is supplemented with the prescription of the initial concentration
\be\label{eq:c^0}
c_{|_{t=0}} = c^0 \in L^\infty(\O) \quad \text{with}\quad 0 \leq c^0 \leq 1\quad \text{and}\quad 0 < \ov c = \oint_\O c^0\d\x < 1, 
\ee
and of boundary conditions. The choice of the boundary conditions may depend on the targeted application: organic semiconductor or ionic liquid. 
For the analysis purpose, we consider boundary conditions which are well adapted to the ionic liquid model. 
Other boundary conditions will also be considered in the numerical simulations in Section~\ref{sec:numerics}.
There are no-flux boundary conditions for the concentration:
\be\label{eq:no-flux}
\bJ \cdot \n = 0 \quad \text{on}\; (0,T) \times \p\O.
\ee
And concerning the Poisson equation~\eqref{Poisson equation}, it is supplemented with inhomogeneous Dirichlet boundary 
conditions on a part $\Gamma_D$ of $\p\O$, and by homogeneous Neumann boundary condition on the remaining part 
$\Gamma_N = \p\O \setminus \Gamma_D$ of the boundary: 
\be\label{eq:BC_Phi}
\Phi = \Phi^D \quad \text{on}\; (0,T) \times \Gamma_D, \qquad 
\grad \Phi \cdot \n = 0 \quad \text{on}\; (0,T) \times \Gamma_N. 
\ee
Throughout the paper, we assume that $\Phi^D$ is defined on the whole domain $\O$ and does not depend on time, with $\Phi^D \in H^{1}(\O) \cap L^\infty(\O)$.

The goal of this paper is to study and compare several different Finite Volume schemes for the system~\eqref{eq:cons_loc}--\eqref{eq:BC_Phi}.
They are based on various reformulations of the flux $\bJ$. Indeed, we may introduce either  
the so-called excess chemical potential $\nu(c) = h(c) - \log(c) = - \log(1-c)$, or the activity and the inverse activity 
coefficient respectively defined by $a(c)=e^{h(c)} = \frac{c}{1-c}$, and $\beta(c)=\frac{c}{a(c)}=1-c$, or the diffusion enhancement $r(c)= - \log(1-c)$ satisfying $r'(c) = ch'(c)$.  Then the flux $\bJ$, initially defined by \eqref{cances flux}, rewrites
\begin{align}
&{\bJ}=-\nabla c -c\nabla \left(\Phi+\nu(c)\right),\label{Sedan flux}\\
&\phantom{\bJ}=-\beta(c)(\nabla a(c) +a(c)\nabla \Phi), \label{Jurgen flux}\\
&\phantom{\bJ}=-r'(c)\nabla c-c\nabla \Phi,\label{Marianne flux}
\end{align}
These formulations \eqref{cances flux}, \eqref{Sedan flux}, \eqref{Jurgen flux} and \eqref{Marianne flux} lead to different schemes that we aim to compare from a numerical analysis point of view. We may notice that the flux ${\bJ}$ also rewrites 
\begin{equation}\label{Weak solution flux}
    \bJ=-\nabla r(c)-c\nabla \Phi.
\end{equation}
This last formulation will be used to define the weak solution to  \eqref{eq:cons_loc}--\eqref{eq:BC_Phi}.

Before going to the discretization of the problem, let us highlight the entropy structure of system~\eqref{eq:cons_loc}--\eqref{eq:BC_Phi}, 
which plays a central role in what follows. 

\subsection{Entropy structure and weak solutions}
The goal of this section is to shortly depict the gradient flow structure of the system~\eqref{eq:cons_loc}--\eqref{eq:BC_Phi}. 
We stay here at a formal level, and remain sloppy about regularity issues. The solutions $(c, \Phi)$ to \eqref{eq:cons_loc}--\eqref{eq:BC_Phi} 
are supposed to be regular enough so that the following calculations are justified. 
Define the mixing entropy density 
\[H(c) = c\log(c) + (1-c) \log(1-c),\] 
which is an antiderivative of $h$, then the electrochemical energy is given by 
\be\label{eq:E}
E(c,\Phi) = \int_\O \left\{ H(c) + \frac12 |\grad \Phi|^2\right\}\d\x - \int_{\Gamma_D} \Phi^D \grad \Phi \cdot \n \d {\boldsymbol \gamma}.
\ee
The next proposition shows that the electrochemical energy is a Lyapunov functional. Moreover, the dissipation rate for the 
energy is explicitly given. 
\begin{prop}\label{prop:E}
Let $(c,\Phi)$ be a smooth solution to~\eqref{eq:cons_loc}--\eqref{eq:BC_Phi}, with $c$ bounded away from $0$ and $1$, then 
\[
\frac{\d}{\d t}E(c,\Phi) + \int_\O c \left|\grad (h(c) + \Phi) \right|^2 \d\x = 0.
\]
\end{prop}
\begin{proof}
We notice first that, since $\Phi^D$ does not depend on time,
$$
\frac{\d}{\d t}E(c,\Phi)=\int_\O (h(c) \p_t c + \grad \Phi\cdot \p_t\grad \Phi) \d\x -
\int_{\Gamma_D} \Phi^D \p_t \grad \Phi \cdot \n \d\boldsymbol{\gamma}.
$$
Then we apply the Gauss theorem and we use the Poisson equation \eqref{Poisson equation} with a constant doping profile, in order to get 
$$
\frac{\d}{\d t}E(c,\Phi)=\int_\O (h(c)+\Phi)\p_t c. 
$$
Multiplying the conservation law \eqref{eq:cons_loc} by $h(c)+\Phi$ and 
integrating over the domain $\Omega$ yields 
$$
\int_\O \p_t c (h(c)+\Phi)=-\int_\O  c \left|\grad (h(c) + \Phi) \right|^2 \d\x,
$$
thanks to the no-flux boundary condition~\eqref{eq:no-flux}. It concludes the proof of 
Proposition~\ref{prop:E}.
\end{proof}

Let $c\in L^\infty(\O;[0,1])$, we denote by $\Phi[c]$ the 
unique solution to~\eqref{Poisson equation}. One can easily check that the energy functional $c \mapsto E(c, \Phi[c])$ is bounded on $L^\infty(\O;[0,1])$. Indeed, $H$ takes values in $[-\log 2, 0]$ and the bounds on the electrical energy can be obtained by multiplying the Poisson equation by $\Phi-\Phi^D$ and $\Phi$ and integrating over $\O$.    
Therefore, $E(c(t), \Phi(t))$ is finite for all $t >0$, whence a $L^\infty((0,T);H^1(\O))$ estimate on $\Phi$.
We also deduce from Proposition~\ref{prop:E} that the total energy dissipation is bounded, i.e. 
\be\label{eq:dissip}
\int_0^T \int_\O c  \left|\grad (h(c) + \Phi) \right|^2 \d\x \d t \leq C
\ee
for some $C$ uniform with respect to the final time horizon $T$.
Using again that $0 \leq c \leq 1$, we deduce from \eqref{eq:dissip} that 
\be\label{eq:L2H1_r}
\int_0^T \int_\O |\grad r(c)|^2 \d\x\d t \leq \int_0^T \int_\O c  \left|\grad h(c) \right|^2 \d\x \d t \leq C. 
\ee

The aforementioned $L^\infty((0,T);H^1(\O))$ estimate on the potential $\Phi$ and Estimate \eqref{eq:L2H1_r} on $r(c)$ suggest a notion 
of weak solution which is based on the expression~\eqref{Weak solution flux} of the flux $\bJ$. 
In what follows, we denote the vector space $\cH_{\Gamma^D}=\lbrace f\in H^1(\Omega), f_{|_{\Gamma_D}}=0\rbrace$ 
and $Q_T= (0,T)\times\O$.

\begin{definition}\label{Def:weaksol}
A couple $(c,\Phi)$ is a \emph{weak solution of \eqref{eq:cons_loc}--\eqref{eq:BC_Phi}} if 
\begin{itemize}
\item[$\bullet$]  $c \in L^\infty((0,T);[0,1])$ with $r(c) \in L^2((0,T);H^1(\O))$, and $\Phi - \Phi^D \in L^\infty((0,T),\cH_{\Gamma^D})$;
\item[$\bullet$] For all $\varphi \in C^\infty_c([0,T)\times \ov \O)$, there holds 
\be\label{eq:weak_c}
\iint_{Q_T} c \p_t \varphi \d\x\d t + \int_\O c^0 \varphi(0,\cdot) \d\x - \iint_{Q_T}\left( r(c) + c \grad \Phi \right) \cdot \grad \varphi \d\x\d t = 0;
\ee
\item[$\bullet$] For all $\psi \in \cH$ and almost all $t\in (0,T)$, there holds 
\be\label{eq:weak_Phi}
\int_\O \grad \Phi(t,\x) \cdot \grad \psi(\x) \d\x = \int_\O (c(t,\x)+c^{dp}(\x)) \psi(\x) \d\x.
\ee
\end{itemize}
\end{definition}

The goal of this paper is to compare from a numerical analysis point of view several different 
numerical schemes to approximate the solutions to~\eqref{eq:cons_loc}--\eqref{eq:BC_Phi}. 
We pay a particular attention to the preservation at the discrete level of the key properties of the continuous model, 
in particular concerning the preservation of the physical bounds $0 \le c \leq 1$ and the energy/energy dissipation 
relation highlighted in Proposition~\ref{prop:E}. The definition of the Finite Volume approximation is detailed 
in the next section.

\section{Finite Volume approximations}

This section is organized as follows. First in Section~\ref{ssec:mesh}, 
we state the requirements on the mesh and fix some notations. Then in Section~\ref{ssec:scheme}, 
we describe the common basis to the different schemes to be studied in this paper. 
All the methods presented in this paper rely on so-called two-point flux approximations, but 
four different schemes are introduced in Section~\ref{ssec:fluxes} based on the formulations
\eqref{cances flux} to \eqref{Marianne flux} of the flux $\bJ$. Then in Section~\ref{ssec:main}, 
we state our two main results. The first one, namely Theorem~\ref{thm:main1}, focuses on 
the case of a fixed mesh. We are interested 
in the existence of a solution to the nonlinear system corresponding to the schemes, and to 
the dissipation of the energy at the discrete level. More precisely, one establishes 
that the all the studied schemes satisfy a discrete counterpart to Proposition~\ref{prop:E}.
Our second main result, namely Theorem~\ref{thm:main2}, is devoted to the convergence of the 
scheme as the time step and the mesh size tend to $0$. 

\subsection{Discretization of $(0,T) \times \O$}\label{ssec:mesh}

In this paper, we perform a parallel study of four numerical schemes based on 
two-point flux approximation (TPFA) finite volume schemes. As explained 
in~\cite{Droniou-review,Tipi}, this approach appears to be very efficient as soon 
as the continuous problem to be solved numerically are isotropic and one has the 
freedom to choose a suitable mesh fulfilling the so-called orthogonality condition~\cite{Herbin95, EGH00}.
We recall here the definition of such a mesh, which is illustrated by Figure~\ref{fig:mesh}.

\begin{definition}
\label{def:mesh}
An \emph{admissible mesh of $\O$} is a triplet $\left(\Tt, \Ee, {(\x_K)}_{K\in\Tt}\right)$ such that the following conditions are fulfilled. 
\begin{enumerate}[(i)]
\item Each control volume (or cell) $K\in\Tt$ is non-empty, open, polyhedral and convex. We assume that 
\[
K \cap L = \emptyset \quad \text{if}\; K, L \in \Tt \; \text{with}\; K \neq L, 
\qquad \text{while}\quad \bigcup_{K\in\Tt}\ov K = \ov \O. 
\]
\item Each face $\sig \in \Ee$ is closed and is contained in a hyperplane of $\R^d$, with positive 
$(d-1)$-dimensional Hausdorff (or Lebesgue) measure denoted by $m_\sig = \cH^{d-1}(\sig) >0$.
We assume that $\cH^{d-1}(\sig \cap \sig') = 0$ for $\sig, \sig' \in \Ee$ unless $\sig' = \sig$.
For all $K \in \Tt$, we assume that 
there exists a subset $\Ee_K$ of $\Ee$ such that $\p K =  \bigcup_{\sig \in \Ee_K} \sig$. 
Moreover, we suppose that $\bigcup_{K\in\Tt} \Ee_K = \Ee$.
Given two distinct control volumes $K,L\in\Tt$, the intersection $\ov K \cap \ov L$ either reduces to a single face
$\sig  \in \Ee$ denoted by $K|L$, or its $(d-1)$-dimensional Hausdorff measure is $0$. 
\item The cell-centers $(\x_K)_{K\in\Tt}$ are pairwise distinct with $\x_K \in K$, and are such that, if $K, L \in \Tt$ 
share a face $K|L$, then the vector $\x_L-\x_K$ is orthogonal to $K|L$.
\item For the boundary faces $\sig \subset \p\O$, we assume that either $\sig \subset \Gamma_D$ or $\sig \subset \ov \Gamma_N$.
For $\sig \subset \p\O$ with $\sig \in \Ee_K$ for some $K\in \Tt$, we assume additionally that 
there exists $\x_\sig \in \sig$ such that $\x_\sig - \x_K$ is orthogonal to $\sig$.

\end{enumerate} 
\end{definition}

\begin{figure}[htb]
\centering
\resizebox{7cm}{!}{\input{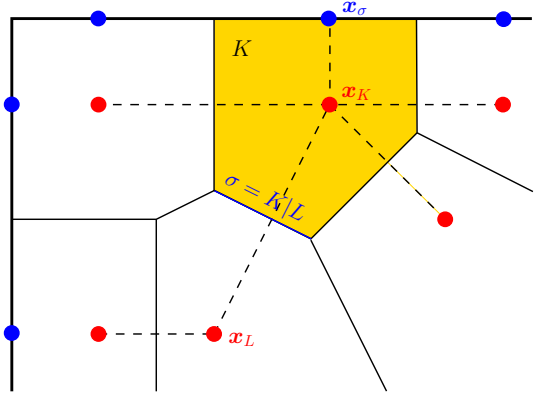}}
\caption{Illustration of an admissible mesh as in Definition~\ref{def:mesh}.}
\label{fig:mesh}
\end{figure}

We denote by $m_K$ the $d$-dimensional Lebesgue measure of the control volume $K$.
The set of the faces is partitioned into two subsets: the set $\Ee_{\rm int}$ of the interior faces defined by 
$
\Ee_{\rm int}= \left\{ \sig \in \Ee\; \middle| \; \sig = K|L\; \text{for some}\; K,L \in \Tt\right\}, 
$
and the set $\Ee_{\rm ext}$ of the exterior faces defined by 
$
\Ee_{\rm ext}= \left\{ \sig \in \Ee\; \middle| \; \sig \subset \p\O\right\},
$
which can also be partitioned into 
$\Ee^D = \{ \sig \subset \Gamma_D\}$ and $\Ee^N = \{\sig \subset \ov \Gamma_N\}$. For a given control volume $K\in\Tt$, we also define $\Ee_{K,{\rm int}}$ the set of its faces which belong to $\Ee_{\rm int}$. For such a face $\sigma\in \Ee_{K,{\rm int}}$, we may write $\sigma=K|L$, meaning that $\sigma=\ov K \cap \ov L$ .

Given $\sig \in \Ee$, we denote by 
\[
d_\sig = \begin{cases}
|\x_K - \x_L| & \text{if}\; \sig = K|L \in \Ee_{\rm int}, \\
|\x_K - \x_\sig| & \text{if}\; \sig \in \Ee_{\rm ext}, 
\end{cases}
\qquad 
\text{and by}\quad \tau_\sig = 
\frac{m_\sig}{d_\sig} 
\]
We finally introduce the size $h_\Tt$ and the regularity $\zeta_\Tt$ (which is assumed to be positive) of a discretization $(\Tt, \Ee, (\x_K)_{K\in\Tt})$ of $\O$ by setting 
\[
h_\Tt = \max_{K \in \Tt}\;{\rm diam}(K), \qquad \zeta_\Tt = \min_{K\in\Tt}\min_{\sig\in\Ee_K} \frac{d(x_K,\sigma)}{d_\sig}.
\]

Concerning the time discretization of $(0,T)$, we consider an increasing finite family of times $0 = t_0 < t_1 < \dots, < t_N = T$. 
We denote by $\dt_n = t_{n}-t_{n-1}$ for $1\leq n\leq N$, by $\boldsymbol{\Delta t} = \left(\dt_n\right)_{1\leq n \leq N}$, 
and by $\ov {\boldsymbol{\Delta t}} = \max_{1 \leq n \leq N} \dt_n$.

\subsection{A common basis for the Finite Volume schemes}\label{ssec:scheme}

All the numerical schemes studied in this paper are based on TPFA Finite Volumes. 
The initial data $c_0$ is discretized into $\left(c^0_K\right)_{K\in\Tt} \in \RR^\Tt$ by setting
\be\label{eq:cK0}
c_K^0 = \frac1{m_K} \int_K c^0(\x) \d\x, \qquad \forall K \in \Tt, 
\ee
while the doping profile $c^{dp}$ is discretized into $\left(c^{dp}_K\right)_{K\in\Tt} \in \RR^\Tt$ by
\be\label{eq:cKdp}
c_K^{dp} = \frac1{m_K} \int_K c^{dp}(\x) \d\x, \qquad \forall K \in \Tt.
\ee
Assume that $\bc^{n-1}=\left(c_{K}^{n-1}\right)_{K\in\Tt}$ is given for some $n > 0$, then we have to define how to compute 
$(\bc^n,\bPhi^n)=\left(c_{K}^{n}, \Phi_K^n\right)_{K\in\Tt}$. 

First, we introduce some notations. For all $K\in \Tt$ and all $\sig \in \Ee_K$, we define the mirror values $c_{K\sig}^n$ and $\Phi_{K\sig}^n$ 
of $c_K^n$ and $\Phi_K^n$ respectively across $\sig$ by setting 
\be\label{eq:mirror}
c_{K\sig}^n = \begin{cases}
c_L^n  &\text{if}\; \sig = K|L \in \Ee_{\rm int},\\
c_K^n & \text{if}\; \sig \in \Ee_{\rm ext}, 
\end{cases}
\quad 
\Phi_{K\sig}^n = \begin{cases}
\Phi_L^n & \text{if}\; \sig = K|L \in \Ee_{\rm int}, \\
\Phi_K^n & \text{if}\; \sig \in \Ee^N,\\
\Phi_\sig^n = \frac{1}{m_\sig}\int_\sig \Phi^D \d\boldsymbol \gamma & \text{if}\; \sig \in \Ee^D.  
\end{cases}
\ee
Given $\bu =\left(u_K\right)_{K\in\Tt} \in \R^\Tt$, we define the oriented and absolute jumps of $\bu$ across any edge by 
\[
D_{K\sig}\bu = u_{K\sig} - u_K, \quad D_\sig\bu = |D_{K\sig} \bu|, \qquad \forall K \in \Tt, \; \forall \sig \in \Ee_K.
\] 

We consider a backward Euler scheme in time and a TPFA finite volume scheme in space. 
It is written as follows:
\begin{subequations}\label{scheme}
\be\label{eq:scheme_Phi}
-\sum_{\sig \in \Ee_K}\tau_\sig D_{K\sig}\bPhi^n
= m_K \left(c_{K}^n + c^{dp}_K\right), \qquad \forall K \in \Tt. 
\ee
\be\label{eq:scheme_c}
m_K\frac{c_K^n - c_K^{n-1}}{\dt_n}  + \sum_{\sig\in\Ee_K}F_{K\sig}^n = 0, \qquad \forall K \in \Tt, 
\ee
\end{subequations}
where $F_{K\sig}^n$ should be a conservative and consistent approximation of $\frac1{\dt_n}\int_{t_{n-1}}^{t_n}\int_\sig \bJ \cdot \n_{K\sig}$ ($\n_{K\sig}$ denotes 
the normal to $\sig$ outward $K$).
The explicit formulas relating the numerical fluxes $F_{K\sig}^n$ to the primary unknowns are now the only remaining degree of freedom. Four possible choice are given in the next section.

\subsection{Numerical fluxes for the conservation of the chemical species}\label{ssec:fluxes}

In order to close the system~\eqref{eq:scheme_Phi}--\eqref{eq:scheme_c}, it remains to define the numerical fluxes $F_{K\sig}^n$. 

Due to the no-flux boundary condition \eqref{eq:no-flux}, we impose, in all the cases,
\be\label{eq:noflux_bc}
F_{K\sig}^n = 0, \quad \forall K \in \Tt, \forall \sig \in \Ee_K\cap \Ee_{\rm ext}.
\ee
The inner fluxes are defined with a function ${\mathcal F}$ of the primary unknowns $(c_K^n, c_L^n, \Phi_K^n,\Phi_L^n)$:
\be\label{eq:fluxint}
F_{K\sig}^n = \tau_\sig{\mathcal F}(c_K^n, c_L^n, \Phi_K^n,\Phi_L^n),\quad \forall K\in\Tt, \forall \sigma=K|L.
\ee
 We discuss now four strategies that are based on the four expressions~\eqref{cances flux}, 
\eqref{Sedan flux}, \eqref{Jurgen flux}, and \eqref{Marianne flux}. They lead to different formulas for ${\mathcal F}$.
Three of the discrete fluxes are extensions of the Scharfetter-Gummel scheme~\cite{SG69} and let the Bernoulli function 
$B(u)= \frac{u}{e^u-1}$, with $B(0)=1$, appear in their definition. 

All the functions ${\mathcal F}$ defined below verify 
$$
{\mathcal F}(c_K,c_L,\Phi_K,\Phi_L)=-{\mathcal F}(c_L,c_K,\Phi_L,\Phi_K) \quad \forall (c_K,c_L,\Phi_K,\Phi_L)\in (0,1)\times (0,1)\times \R\times \R,
$$
so that the numerical fluxes are locally conservative, which means 
\be\label{eq:cons_F}
F_{K\sig}^n +F_{L\sig}^n=0\quad \forall \sig=K|L \in \Ee_{\rm int}.
\ee

\subsubsection{The centered flux}\label{sssec:centered}
The so-called~{centered flux} is derived from formula~\eqref{cances flux}, which suggests the following definition of ${\mathcal F}$: 
\be\label{eq:centered_flux}
{\mathcal F}(c_K,c_L,\Phi_K,\Phi_L) = - \frac{c_K + c_{L}}2 D_{K\sig} \left(h(\bc) + \bPhi \right).
\ee
The associate  flux can be seen as a particular case in the TPFA context of the fluxes introduced in~\cite{CG_VAGNL, CCHK_CMAM, Cances_OGST,CNV_HAL} 
in various multipoint flux approximations (MPFA) or finite element contexts. In opposition to the three next schemes, the centered scheme is not based on the 
Scharfetter-Gummel scheme. We can notice that, even if the relation~\eqref{Weak solution flux} between the flux and the concentration would be linear (i.e., if 
$h(c)= \log(c)$ so that $r(c) = c$), ${\mathcal F}$ is nonlinear with respect to $c_K$ and $c_L$ and also singular near 0.

\subsubsection{The Sedan flux}\label{sssec:Sedan}
The second flux we introduce is named {Sedan} after the eponymous code SEDAN III~\cite{SEDAN}. Formula~\eqref{Sedan flux} for the flux $\bJ$ 
suggests to use a classical Scharfetter-Gummel scheme, but for a modified potential $\Phi+\nu(c)$ instead of only $\Phi$, leading to the following definition of ${\mathcal F}$:
\be\label{eq:SEDAN_flux}
{\mathcal F}(c_K,c_L,\Phi_K,\Phi_L)=  \Bigl\{
B\left(D_{K\sigma}(\bPhi + \nu(\bc)) \right) c_K - B\left(-D_{K\sigma}(\bPhi + \nu(\bc))\right)c_{L}\Bigl\} .
\ee
\begin{remark}\label{rem:Sedan}
We notice that the Sedan flux defined by \eqref{eq:SEDAN_flux} satisfies 
$$
{\mathcal F}(c_K,c_L,\Phi,\Phi)=r(c_K)-r(c_L), \quad \forall (c_K,c_L)\in (0,1)\times (0,1),\  \forall \Phi\in \R.
$$
It means that, when ${\mathbf J}=-\nabla r(c)$, we recover the classical two-point flux approximation:
$$
F_{K\sig}^n =\tau_\sig (r(c_K^n)-r(c_L^n)),\quad  \forall K\in\Tt, \forall \sigma=K|L.
$$
\end{remark}

\subsubsection{The activity based flux}\label{sssec:activity}
The {activity based flux} we discuss now is a restriction to our simplified model of the flux introduced in~\cite{Fuhrmann15, FG_FVCA8}.
It relies on the expression~\eqref{Jurgen flux} of the flux $\bJ$. Assume that $a(c)$ and $\beta(c)$ are independent one from another 
(even though this is of course not true), then the flux $\bJ$ is linear w.r.t. $a(c)$, while $\beta(c)$ is  a multiplicative factor.
This suggests to choose a particular average for $\beta(c)$ ---here the arithmetic mean--- and to apply the Scharfetter-Gummel scheme 
in order to approximate $\grad a(c) +a(c)\grad \Phi$. This yields 
\be\label{eq:activity_flux}
{\mathcal F}(c_K,c_L,\Phi_K,\Phi_L)= \frac{\beta(c_K)+\beta(c_{L})}{2}
\Bigl\{ B(D_{K\sigma}\bPhi)a(c_K)-B(-D_{K\sigma}\bPhi)a(c_{L})
\Bigl\}.
\ee

\subsubsection{The Bessemoulin-Chatard flux}\label{sssec:Marianne}
The last numerical flux we consider here is named {Bessemoulin-Chatard flux} after the author's name of~\cite{BC12}. 
Formula~\eqref{Marianne flux} for the flux $\bJ$ suggests that, up to the introduction of a variable diffusion coefficient approximating 
the quantity $r'(c)$ per face, one can use the Scharfetter-Gummel scheme. Following~\cite{BC12}, 
the approximation $dr(c_K,c_L)$ of $r'(c)$ is defined as 
\[
dr(c_K,c_L) = \begin{cases}
\ds\frac{h(c_K)-h(c_{L})}{\log(c_K)-\log(c_{L})} &\text{if}\; c_K \neq  c_L, \\
r'(c_K) & \text{if}\; c_K = c_{L}.
\end{cases}
\]
This leads to the following definition of ${\mathcal F}$:
\be\label{eq:Marianne_flux}
{\mathcal F}(c_K,c_L,\Phi_K,\Phi_L)=  dr(c_K,c_L)
\left\{B\left(\frac{D_{K\sigma}\bPhi}{dr(c_K,c_L)}\right)c_K-B\left(-\frac{D_{K\sigma}\bPhi}{dr(c_K,c_L)}\right)c_{L}\right\}. 
\ee

\subsection{Main results and organisation of the paper}\label{ssec:main}

We have introduced four schemes defined by \eqref{eq:cK0}--\eqref{eq:fluxint}, supplemented with one of the four definitions of ${\mathcal F}$: \eqref{eq:centered_flux}, \eqref{eq:SEDAN_flux}, \eqref{eq:activity_flux}, or \eqref{eq:Marianne_flux}.
Besides numerical comparisons between the different approaches ---this will be the purpose of Section~\ref{sec:numerics}---, 
we aim at proposing shared pieces of numerical analysis for all the schemes. 

All the four schemes proposed above yield a nonlinear system to be solved at each time step. 
The first theorem proven is this paper concerns the existence of discrete solutions for a given mesh, 
and the preservation of the physical bounds: boundedness of the concentration between $0$ and $1$, 
decay of the energy. The discrete energy functional $E_\Tt$ is the discrete counterpart of the continuous energy 
functional $E$, namely
\be\label{eq:E_T}
E_\Tt(\bc^n,\bPhi^n) = \sum_{K\in\Tt}m_K H(c_K^n) + \frac12 \sum_{\sig\in\Ee}
\tau_\sig \left(D_{\sig} \bPhi^n\right)^2 - \sum_{K\in\Tt}\sum_{\sig\in\Ee^D\cap \Ee_K} \tau_\sig \Phi_{\sig}^D D_{K\sig} \bPhi^n.
\ee
As stated in Theorem~\ref{thm:main1} below, the nonlinear systems corresponding to all the four schemes 
admit solutions which preserve the physical bounds on the concentrations and the decay of the energy. 
The proof of Theorem~\ref{thm:main1} will be the purpose of Section~\ref{sec:existence}.

\begin{theorem}\label{thm:main1}
Let $(\Tt,\Ee,\left(\x_K\right)_{K\in\Tt})$ be an admissible mesh and let $\bc^0$ be defined by~\eqref{eq:cK0}. Then, for all $1\leq n\leq N$, the nonlinear system of equations \eqref{eq:mirror}--\eqref{eq:fluxint}, supplemented either with \eqref{eq:centered_flux}, \eqref{eq:SEDAN_flux}, \eqref{eq:activity_flux}, or \eqref{eq:Marianne_flux}, has a solution $(\bc^{n}, \bPhi^{n})\in [0,1]^{\Tt} \times \R^{\Tt}$.
Moreover, the solution to the scheme satisfies, for all $1\leq n\leq N$,
$$
E_\Tt(\bc^{n,},\bPhi^{n}) \leq E_\Tt(\bc^{n-1},\bPhi^{n-1})\mbox{ and }
0<c_K^{n}<1,\quad \forall K\in\Tt.
$$
\end{theorem}

Knowing a discrete solution to the scheme, $(\bc^{n}, \bPhi^{n})_{1\leq n\leq N}$, we can define an approximate solution $(c_{\Tt,\bdt}, \Phi_{\Tt,\bdt})$. It is the piecewise constant function defined almost everywhere by 
\[
c_{\Tt, \boldsymbol{\Delta t}}(t,\x) = c_K^{n}, \quad 
\Phi_{\Tt, \boldsymbol{\Delta t}}(t,\x) = \Phi_K^{n} \quad \text{if}\;(t,\x) \in (t_{n-1},t_n]\times K.
\]
This definition will be developed in Section \ref{sec:convergence} and supplemented by other reconstruction operators.

Let $\left(\Tt_m, \Ee_m,(\x_K)_{K\in\Tt_m}\right)_{m\geq 1}$ be a sequence of admissible meshes in the sense of Definition~\ref{def:mesh} 
such that $h_{\Tt_m}, \ov{\boldsymbol{\Delta t}}_m \underset{m\to\infty}\longrightarrow 0$ while the mesh regularity remains bounded, i.e., 
$\zeta_{\Tt_m} \geq \zeta^\star$ for some $\zeta^\star>0$ not depending on $m$, a natural question is the convergence of the associated sequence of approximate solution $(c_{\Tt_m,\bdt_m}, \Phi_{\Tt_m,\bdt_m})_{m\geq 1}$ towards a weak solution to the continuous problem. The convergence result is stated in Theorem \ref{thm:main2}, only for the centered scheme and the Sedan scheme.

\begin{theorem}\label{thm:main2}
For the centered scheme (inner fluxes defined by \eqref{eq:fluxint} {and} \eqref{eq:centered_flux}) and the Sedan scheme (inner fluxes defined by \eqref{eq:fluxint} 
{and} \eqref{eq:SEDAN_flux}), a sequence of approximate solutions $(c_{\Tt_m,\bdt_m}, \Phi_{\Tt_m,\bdt_m})_{m\geq 1}$ satisfies, up to a subsequence,
\be\label{eq:conv}
c_{\Tt_m, \bdt_m} \underset{m\to\infty}\longrightarrow c \quad \text{a.e. in}\; Q_T, \qquad
\Phi_{\Tt_m, \bdt_m} \underset{m\to\infty}\longrightarrow \Phi \quad \text{in}\; L^2(Q_T) 
\ee
where $(c,\Phi)$ is a weak solution in the sense of Definition~\ref{Def:weaksol}.
\end{theorem}

The above theorem deserves some comments. First, the convergence proof carried out in what follows 
does not encompass the activity based scheme and the Bessemoulin-Chatard scheme for reasons 
that will appear clearly in the proof later on. This does of course not mean that these schemes do not converge, 
but only that our analysis does not cover them. Second, the topologies for which the convergence is claimed 
in~\eqref{eq:conv} is suboptimal when compared to the results we prove in Section~\ref{sec:convergence}.
However, we choose to keep the statement as simple as possible. The interested reader can refer to 
Section~\ref{sec:convergence} to get finer results, including the convergence of approximate gradients 
to be defined later on. 

Section~\ref{sec:numerics} is then devoted to the comparison of the numerical results produced by the different schemes.

\section{Numerical analysis for fixed meshes}\label{sec:existence}

In this section, one aims to show that each scheme admits at least one solution and 
that the physical bounds are preserved by the schemes. 
Our approach is based on a topological degree argument~\cite{LS34,Dei85} to be detailed in Section~\ref{ssec:existence}. 
It relies on a priori estimates to be stated in Section~\ref{ssec:apriori}.
But let us start by some preliminary properties of the different functions ${\mathcal F}$, defined either by  \eqref{eq:centered_flux}, \eqref{eq:SEDAN_flux}, \eqref{eq:activity_flux}, or \eqref{eq:Marianne_flux}, and some consequences  for the inner numerical fluxes $F_{K\sig}^{n}$ .
\subsection{Face concentration and face dissipation}\label{ssec:fluxes2}
For each scheme, one can naturally define a face concentration functional 
${\mathcal C}:(0,1)\times(0,1)\times \R \times \R\to \R$ by
\begin{equation}\label{eq:EC}
\Cc: (c_K,c_{L}, \Phi_K,\Phi_{L})\mapsto \frac{{\mathcal F} (c_K,c_L,\Phi_K,\Phi_L)}{ h(c_K)+\Phi_K-h(c_{L})-\Phi_{L}}.
\end{equation}
It clearly satisfies 
$
\Cc\left(c_K,c_{L}, \Phi_K,\Phi_{L}\right) = \Cc\left(c_{L},c_{K}, \Phi_{L},\Phi_{K}\right).
$
It  means that we can define one unique face concentration by internal face and by choice of flux
\be\label{eq:Csig}
\Cc_\sig^{n} = \Cc\left(c_K^{n},c_{L}^{n}, \Phi_K^{n},\Phi_{L}^{n}\right) \quad\forall \sigma\in\Ee_{\rm int}, \sigma=K|L
\ee
and that each flux $F_{K\sig}^{n}$ can be rewritten as 
\be\label{eq:FKsig_Csig}
F_{K\sig}^{n}= -\tau_\sig \Cc_\sig^{n} D_{K\sig}(h(\bc^{n}) + \bPhi^{n}), \quad \forall K\in\Tt, \forall \sigma=K|L. 
\ee

We also introduce a  face dissipation functional $\Dd: (0,1)\times(0,1) \times \R \times \R\to\R$, defined by 
\be\label{eq:Dj}
\Dd(c_K,c_L,\Phi_K,\Phi_{L}) = \Cc\left(c_K,c_{L}, \Phi_K,\Phi_{L}\right)
\left| h(c_K) - h(c_{L}) + \Phi_K - \Phi_{L} \right|^2.
\ee
and we set, for each scheme, 
\be\label{eq:Dsig}
\Dd_\sig^{n} = \Dd(c_K^{n},c_{L}^{n},\Phi_K^{n},\Phi_{L}^{n}), \quad \forall \sigma\in\Ee_{\rm int}, \sigma=K|L.
\ee
For $\delta\in (0,1)$ and $M\in\R$, we finally define two functions associated to ${\mathcal D}$, ${\Psi}_{\delta,M}:(0,1) \to \R$ and  ${\Upsilon}_{\delta,M}:(0,1) \to \R$, by 
\be\label{eq:defPsiUpsilon}
\begin{aligned}
{\Psi}_{\delta,M}(c_L)&=\inf \{\Dd(c_K,c_L, \Phi_K, \Phi_L);  c_K\in [\delta,1),  (\Phi_K,\Phi_L) \in [-M,M]^2\}\\
{\Upsilon}_{\delta,M}(c_L)&=\inf \{\Dd(c_K,c_L, \Phi_K, \Phi_L);  c_K\in (0, 1-\delta], (\Phi_K,\Phi_L) \in [-M,M]^2\} 
\end{aligned}
\ee
{
Note that $\delta \mapsto \Psi_{\delta, M}(c_L)$  and $\delta \mapsto \Psi_{\delta, M}(c_L)$ are nondecreasing for all $M\in \R$ and all $c_L\in (0,1)$.
}

Our first lemma in this section focuses on the face concentration, which for three scheme over four can be shown to 
be an average value of the surrounding cell concentrations. 
\begin{lemma}\label{lem:avg}
The face concentration functional defined by \eqref{eq:EC} and either \eqref{eq:centered_flux}, \eqref{eq:SEDAN_flux} or \eqref{eq:Marianne_flux} verifies, for all $(c_K,c_{L},\Phi_K,\Phi_{L}) \in(0,1)\times(0,1)\times \R \times \R$, 
\be
\min(c_K,c_{L})\leq  {\mathcal C}(c_K,c_L,\Phi_K,\Phi_L)\leq\max(c_K,c_{L}). \label{eq:avg}
\ee
Property~\eqref{eq:avg} does not hold in general in the case where ${\mathcal F}$ is defined by \eqref{eq:activity_flux} (activity based flux), but one still has, for all $(c_K,c_{L},\Phi_K,\Phi_{L}) \in(0,1)\times(0,1)\times \R \times \R$, 
\be\label{eq:avgA}
\Cc(c_K,c_{L},\Phi_K,\Phi_{L}) \geq \frac{\min(c_K,c_{L})}{2}>0.
\ee
\end{lemma}
\begin{proof}
We first remark that Property~\eqref{eq:avg} trivially holds for the centered flux~\eqref{eq:centered_flux} since, in this case, 
\[\Cc(c_K,c_{L},\Phi_K,\Phi_{L}) = \frac{c_{K}+c_L}2.\] 

The proof is more intricate for the Sedan flux \eqref{eq:SEDAN_flux}  and the Bessemoulin-Chatard flux \eqref{eq:Marianne_flux}. It relies
on elementary properties of the Bernoulli function.

Let us start with the Bessemoulin-Chatard flux~\eqref{eq:Marianne_flux}, for which 
\begin{equation}\label{eq:bchD}
\Cc(c_K,c_{L},\Phi_K,\Phi_{L})=\frac{dr_\sig\left(B\left(\ds\frac{\Phi_{L}-\Phi_K}{dr_\sig}\right)c_K-
B\left(\ds\frac{\Phi_K-\Phi_{L}}{dr_\sig}\right)c_{L}\right)}{h(c_K)+\Phi_K-h(c_{L})-\Phi_{L}},
\end{equation}
where we have set 
\be\label{eq:dr_sig}
dr_\sig = dr(c_K,c_L).
\ee
Let us now recall the elementary property of the Bernoulli function:
\be\label{eq:logBernoulli}
B(\log(a) - \log(b)) a - B(\log(b) - \log(a))b =0, \qquad \forall (a,b) \in (0,1)^2.
\ee
Introducing the quantities $x = \log(c_K/c_{L})$ and 
$y=(\Phi_{L}-\Phi_K)/{dr_\sig}$, elementary calculations show that 
the relation~\eqref{eq:bchD} rewrites 
\be\label{eq:B_avg}
\Cc(c_K,c_{L},\Phi_K,\Phi_{L})=\frac{B(y)-B(x)}{x-y}c_K+\frac{B(-x)-B(-y)}{x-y}c_{L}.
\ee
But, the Bernoulli function is decreasing and satisfies 
$B(x) - B(-x) = -x$ for all $x\in\R$, which implies
\[\frac{B(y)-B(x)}{x-y} + \frac{B(-x)-B(-y)}{x-y} =1.\] 
Thus $\Cc(c_K,c_{L},\Phi_K,\Phi_{L})$ is a convex combination of $c_K$ and $c_{L}$, so that~\eqref{eq:avg}
holds for the Bessemoulin-Chatard flux.

The case of the Sedan flux~\eqref{eq:SEDAN_flux} can be treated similarly because \eqref{eq:B_avg} is still satisfied, 
but with $x=\log(c_K/c_L)$ and $y= \Phi_{L} + \nu(c_{L}) - \Phi_{K} + \nu(c_{K})$. Here again, $\Cc(c_K,c_{L},\Phi_K,\Phi_{L})$ is a convex 
combination of $c_K$ and $c_{L}$, so that~\eqref{eq:avg}
holds for the Sedan flux.

The fact that \eqref{eq:avg} does not hold for the activity based flux~\eqref{eq:activity_flux} is depicted on Figure~\ref{fig:EC}. 
Nevertheless, one can express the corresponding face concentration under the form 
\[
 \Cc(c_K,c_{L},\Phi_K,\Phi_{L})= \frac{\beta(c_K)+\beta(c_L)}2 
\times  \left(\frac{B(y)-B(x)}{x-y}a(c_K)+\frac{B(-x)-B(-y)}{x-y}a(c_{L})
  \right),
\]
with $x=\log(a(c_K)) - \log(a(c_{L}))$ and $y=\Phi_{L} - \Phi_K$.
Therefore,  $ \Cc(c_K,c_{L},\Phi_K,\Phi_{L})$ is the product of the arithmetic mean of the positive 
quantities $\beta(c_K)$ and $\beta(c_{L})$ with a convex combination of the positive quantities $a(c_K)$ and $a(c_{L})$. As $a$ is increasing, this convex combination is bounded by below by $a(\min(c_K,c_{L}))$. Using the identity $\beta(c)a(c)=c$, we get \eqref{eq:avgA}. 
\end{proof}

\begin{figure}[htb]
\begin{tikzpicture}
\begin{axis}[
 xlabel=$\Phi_{L}-\Phi_K$,
 ylabel=Face concentration $\Cc$,
 legend pos=outer north east,
 legend cell align={left}]
\addplot +[color=centeredcolor, dashdotted, mark=none, line width=1.5] table [y=trois, x=zero]{C_edge_1.dat};
\addlegendentry{$\Cc$ centered}
\addplot +[color=activitycolor, densely dotted, mark=none, line width=1.5] table [y=cinq, x=zero]{C_edge_1.dat};
\addlegendentry{$\Cc$ activity based}
\addplot +[color=BCHcolor, densely dashed, mark=none, line width=1.5] table [y=six, x=zero]{C_edge_1.dat};
\addlegendentry{$\Cc$ Bessemoulin-Chatard}
\addplot +[color=sedancolor, loosely dashed, mark=none, line width=1.5] table [y=quatre, x=zero]{C_edge_1.dat};
\addlegendentry{$\Cc$ Sedan}
\addplot [mark=none, style = solid, color = black, line width=1] table [y=deux, x=zero]{C_edge_1.dat};
\addplot [mark=none, style = solid, color = black, line width=1] table [y=un, x=zero]{C_edge_1.dat};
\addlegendentry{$c_{K}$, $c_L$}
\end{axis}
\end{tikzpicture}

\caption{
Evolution of the face concentration $\Cc(c_K,c_{L},\Phi_K,\Phi_{L})$ as a function of the jump 
of the potential $\Phi_{L}-\Phi_K$ for the choice $c_K = 0.3$ and $c_{L}=0.7$.}
\label{fig:EC}
\end{figure}
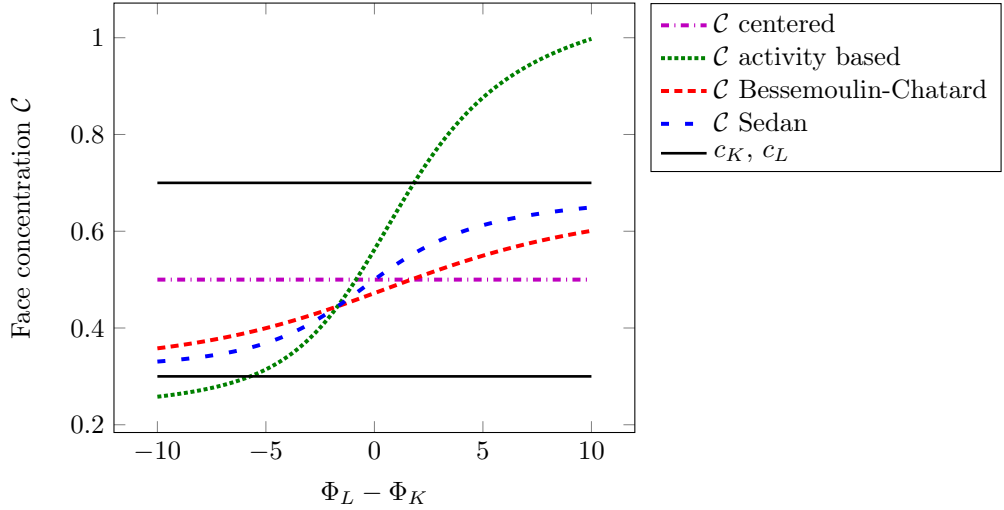

As a by-product of Lemma \ref{lem:avg}, we obtain that the face dissipation ${\mathcal D}$ is a nonnegative function as the product of nonnegative quantities. The second lemma of this section is about the coercivity of the face dissipation functional. As its proof is technical, it is given in Appendix~\ref{app:dissip}.
\begin{lemma}\label{lem:dissip}
The face dissipation functional defined by \eqref{eq:Dj} and either \eqref{eq:centered_flux}, \eqref{eq:SEDAN_flux},   \eqref{eq:activity_flux} or \eqref{eq:Marianne_flux}  satisfies the following dissipation property: given $\delta \in (0,1)$ and $M \in \R$, there holds, for $\Psi$ and $\Upsilon$ as defined in \eqref{eq:defPsiUpsilon}: 
\[
\begin{gathered}
\ds\lim_{c_L\to 0} {\Psi}_{\delta,M}(c_L)= + \infty,\\
\ds\lim_{c_L\to 1} {\Upsilon}_{\delta,M}(c_L)= + \infty.
\end{gathered} 
\]
\end{lemma}
\subsection{Uniform a priori estimates}\label{ssec:apriori}

In all this section, we assume that $(\bc^{n},\bPhi^{n})_{1\leq n\leq N}$ is a solution to the scheme 
\eqref{eq:mirror}--\eqref{eq:fluxint} with a  numerical flux defined among~\eqref{eq:centered_flux}--\eqref{eq:Marianne_flux}. We also assume that this solution verifies: $ 0 < c_K^n< 1$ for all $K\in\Tt$ and all $1\leq n\leq N$. Then the goal 
of this section is to derive enough a priori estimates on $(\bc^{n},\bPhi^{n})_{1\leq n\leq N}$
in order to show the existence of a weak solution to the nonlinear system induced by the 
scheme.

The first lemma is the discrete counterpart of the global conservation of mass. 
\begin{lemma}\label{lem:mass}
One has 
\[
\sum_{K\in\Tt} m_K c_K^{n} = \sum_{K\in\Tt} m_K c_K^{n-1} =  \int_\O c^0\d\x, \quad \forall 1\leq n\leq N.
\]
\end{lemma}
\begin{proof}
The first equality is obtained by summing~\eqref{eq:scheme_c} over $K\in\Tt$ and by using the no-flux boundary conditions \eqref{eq:noflux_bc} and the 
local conservativity of the scheme~\eqref{eq:cons_F}. 
A straightforward induction ensures the second equality thanks to~\eqref{eq:cK0} .
\end{proof}

The second a priori estimate is related to energy dissipation and can be seen as 
a discrete counterpart of Proposition~\ref{prop:E}. 

\begin{prop}\label{prop:E_disc}
There holds
\[
\frac{E_\Tt(\bc^{n},\bPhi^{n}) - E_\Tt(\bc^{n-1},\bPhi^{n-1})}{\dt_n} \leq
- \sum_{\sig \in \Ee_{\rm int}} \tau_\sig {\mathcal D}_{\sig}^n \leq 0, \quad \forall 1\leq n\leq N,
\]
with ${\mathcal D}_{\sig}^n$ defined by \eqref{eq:Dsig}, and $E_\Tt$ defined by \eqref{eq:E_T}.
\end{prop}
\begin{proof}
Due to the convexity of $H$ and of $x\mapsto x^2/2$, we have :
\begin{multline*}
E_\Tt(\bc^{n},\bPhi^{n}) - E_\Tt(\bc^{n-1},\bPhi^{n-1})\leq  \sum_{K\in\Tt} m_K (c_K^n - c_K^{n-1})h(c_K^n)+\\
\sum_{\sigma\in\Ee}\tau_\sig D_\sig \bPhi^nD_\sig (\bPhi^n-\bPhi^{n-1})-\sum_{K\in\Tt}\sum_{\sig\in\Ee^D\cap\Ee_K}\tau_\sig \Phi_\sig^D D_{K\sig}( \bPhi^n- \bPhi^{n-1}).
\end{multline*}
A discrete integration by parts permits to rewrite the sum of the two last terms, which, combined to the scheme \eqref{eq:scheme_Phi}, leads to 
\be\label{E:step1}
E_\Tt(\bc^{n},\bPhi^{n}) - E_\Tt(\bc^{n-1},\bPhi^{n-1})\leq  \sum_{K\in\Tt} m_K (c_K^n - c_K^{n-1})(h(c_K^n)+\Phi_K^n).
\ee
Multiplying the equation~\eqref{eq:scheme_c} by $h(c_K^n) + \Phi_K^n$ and summing over $K\in\Tt$, we obtain that 
\begin{multline}\label{E:step2}
\sum_{K\in\Tt} m_K \frac{c_K^n - c_K^{n-1}}{\Delta t_n}(h(c_K^n)+\Phi_K^n)=-\sum_{K\in\Tt}(h(c_K^n)+\Phi_K^n)\sum_{\sigma \in \Ee_K} F_{K\sigma}^n\\
=-\sum_{\sig \in \Ee_{\rm int}} \tau_\sig \Cc_\sig^n \left|D_\sig \left(h(\bc^n) + \bPhi^n\right)\right|^2
\end{multline}
Combining \eqref{E:step1} and \eqref{E:step2} provides the desired estimate.
\end{proof}

The third statement of this section is devoted to a uniform $L^\infty$ estimate of $(\bPhi^n)_{1\leq n\leq N}$.
It is a straightforward consequence of the slightly more general Proposition~\ref{prop:Linf_Poisson}
stated in appendix, together with the a priori bounds $0 < c_K^n < 1$ and $-\|c^{dp}\|_\infty \leq c_K^{dp} \leq \|c^{dp}\|_\infty$. 
\begin{lemma}\label{lem:LinfPhi}
There exists $M_\Phi$ depending only on $\Phi^D$, $c^{dp}$ and $\O$ such that 
\[
|\Phi_K^{n}| \leq M_\Phi, \qquad \forall K\in\Tt, \; \forall 1\leq n \leq N.
\]
\end{lemma}

Next lemma concerns the discrete $L^\infty((0,T);H^1(\O))$ estimate on the electric potential. 
\begin{lemma}\label{lem:LinfH1Phi}
There exists $C$ depending only on $\Phi^D$, $c^{dp}$, $\O$ and $\zeta_\Tt$ such that 
\[
\begin{gathered}
\sum_{\sig \in \Ee} \tau_\sig |D_\sig \bPhi^n|^2 \leq C, \qquad  \forall 1\leq n \leq N\\
\mbox{ and } \left\vert\sum_{K\in\Tt}\sum_{\sig\in\Ee^D\cap \Ee_K} \tau_\sig \Phi_{\sig}^D D_{K\sig} \bPhi^n\right\vert \leq C, \qquad  \forall 1\leq n \leq N.
\end{gathered}
\]
\end{lemma}

\begin{proof}
As $\Phi^D \in L^\infty \cap H^1(\O)$, it is  discretized into $\bPhi^D \in \R^{\Tt}$ by setting 
\[
\Phi^D_K = \frac1{m_K}\int_K \Phi^D\d\x, \quad \forall K \in \Tt,\ \mbox{ and } 
\Phi^D_\sig = \frac1{m_\sig}\int_\sig \Phi^D\d\bgamma,\qquad \forall \sig\in\Ee^D.
\]
It satisfies  $|\Phi_K^D| \leq \|\Phi^D\|_\infty$ for all $K\in\Tt$.
Multiplying~\eqref{eq:scheme_Phi} by $\Phi_K^n - \Phi_K^D$ and summing over $K\in\Tt$ provides
\be\label{eq:AB}
\sum_{\sig \in \Ee}\tau_\sig D_{K\sig}\bPhi^n D_{K\sig}(\bPhi^n - \bPhi^D) = \sum_{K\in\Tt}m_K (c_K^n + c_K^{dp})(\Phi_K^n - \Phi_K^D). 
\ee
Using the elementary inequality $a(a-b) \geq \frac{a^2-b^2}2 $, we get that 
\[
\sum_{\sig \in \Ee}\tau_\sig D_{K\sig}\bPhi^n D_{K\sig}(\bPhi^n - \bPhi^D) \geq 
\frac12 \sum_{\sig \in \Ee}\tau_\sig (D_\sig \bPhi^n)^2 - \frac12 \sum_{\sig \in \Ee}\tau_\sig (D_\sig \bPhi^D)^2.
\]
Using the boundedness of $c_K^n$, $c_K^{dp}$, $\Phi_K^D$, and of $\Phi_K^n$ (cf. Lemma~\ref{lem:LinfPhi}), we obtain that 
the right-hand side is bounded:
\[
\sum_{K\in\Tt}m_K (c_K^n + c_K^{dp})(\Phi_K^n - \Phi_K^D) \leq C.
\]
Following~\cite[Lemma 13.4]{EGH00}, there exists $C$ depending only on $\zeta_\Tt$ such that 
\be\label{eq:H1PhiD}
\sum_{\sig \in \Ee} \tau_\sig (D_\sig \bPhi^D)^2 \leq C \|\Phi^D\|^2_{H^1(\O)}, 
\ee
which allows to conclude the proof of the first inequality of Lemma~\ref{lem:LinfH1Phi}. Multiplying now the scheme \eqref{eq:scheme_Phi}  by $\Phi_K^n$ and summing over $K\in\Tt$ leads to the second inequality by following the same kind of computations.
\end{proof}

\begin{remark}\label{rem:D_T2} Using Proposition \ref{prop:E_disc}, we notice that Lemma~\ref{lem:LinfH1Phi} implies the existence of $C$ depending only on $c^0$, $\Phi^D$, $\Omega$ and $\zeta_\Tt$ such that:
\[
\sum_{n=1}^N \dt_n \sum_{\sig \in \Ee_{\rm int}} \tau_\sig {\mathcal D}_{\sig}^n \leq C, \qquad \forall N\in \mathbb{N}
\]
\end{remark}
As a last step before establishing the existence of a solution to the scheme, 
we show that the approximate concentrations $\bc^n$ are bounded away from $0$ and $1$. 
Note that contrary to Lemmas~\ref{lem:mass} and \ref{lem:LinfPhi} and to Proposition~\ref{prop:E_disc}, 
the estimate of the following Lemma is not uniform with respect to mesh size and time step. 
\begin{lemma}\label{lem:epsilon}
There exists $\eps>0$ depending on $\Tt, \bdt$, $\Phi^D$, $\bar{c}$, $c^{dp}$ and $\O$ such that 
\[
\eps < c_K^n <1-\eps, \qquad \forall K\in\Tt, \; \forall 1\leq n\leq N.
\]
\end{lemma}
\begin{proof}
The proof follows the idea of~\cite[Lemma 3.10]{CG16_MCOM} (see also~\cite[Lemma 3.7]{CG_VAGNL}).
Let us establish the lower bound, since the outline of the proof of the upper bound is similar. 
Here again, we drop the superscript $j$ for the ease of reading. 

Because of assumption~\eqref{eq:c^0} on the initial data and of the choice \eqref{eq:cK0} for its discretization, one knows 
that 
\[
\frac1{m_\O}\sum_{K\in\Tt}m_K c_K^0  = \ov c  \in (0,1). 
\]
Lemma~\ref{lem:mass} ensures the conservation of mass, so that 
\[
\frac1{m_\O}\sum_{K\in\Tt}m_K c_K^n  = \ov c  \in (0,1), \qquad \forall  n \geq 1.
\]
This implies that there exists $K_0\in\Tt$ such that $c_{K_0}^n \geq \ov c >0$. We set $\delta_0= {\ov c}$.

Denote by $\bPhi[\bc^n]$ the unique solution to the linear system~\eqref{eq:scheme_Phi}. Lemma \ref{lem:LinfH1Phi} ensures that the functional $\bc^n \mapsto E_\Tt(\bc^n, \bPhi[\bc^n])$ is bounded on $(0,1)^{\Tt}$. 
Therefore, it follows from Proposition~\ref{prop:E_disc} that there exists $C_\Dd$ depending (among others) on  $\dt_n$ 
such that 
\be\label{eq:D_T2}
\Dd_\Tt^n =  \sum_{\sig \in \Ee_{\rm int}}\tau_{\sig} \Dd_\sig^n \leq C_\Dd, 
\qquad \forall 1\leq n \leq N.
\ee
In particular, for all face $\sig \in \Ee_{K_0}$, one gets that $\tau_{\sig}\Dd_\sig^n \leq C_\Dd$.
Therefore, the concentration $c_{K_1}^n $ in any neighboring cell $K_1$ of $K_0$ 
is bounded away from $0$ by 
\begin{align*}
c_{K_1}^n\geq&  \inf\left\{c_{L}\in(0,1)\ ;\   {\Psi}_{{c_{K_0}^n},M_{\Phi}}(c_L)\leq {C_\Dd} / \tau_{\sig}\right\} \\
{\geq}&\;{\inf\left\{c_{L}\in(0,1)\ ;\   {\Psi}_{\delta_0,M_{\Phi}}(c_L)\leq {C_\Dd} / \min_{\sig\in\Ee_{\rm int}} \tau_\sig\right\} =:\delta_1>0}
\end{align*}
{thanks to the monotonicity of $\delta\mapsto \Psi_{\delta,M}(c_L)$.}
Owing to Lemma~\ref{lem:dissip}, the above right-hand side is bounded away from $0$ by some quantity that might also 
depends on $\Tt$ because of the presence of $\tau_\sig$. This lower bound can be set to $\delta_1$ and we can then iterate the procedure to the neighboring cells of $K_1$, and so on and so forth. Since the mesh is finite, only a finite number of iterations $ I_\Tt$ is needed to cover all the cells, 
whence a uniform lower bound on $c_K^n$: $\epsilon=\min_{1\leq i\leq I_\Tt} \delta_i$, {where}
\[
{\delta_{i+1}= \inf\left\{c_{L}\in(0,1)\ ;\   {\Psi}_{\delta_i,M_{\Phi}}(c_L)\leq {C_\Dd} / \min_{\sig\in\Ee_{\rm int}}\tau_\sig\right\}>0, 
\quad \delta_0=\ov c.}
\]
\end{proof}

\subsection{Existence of a solution to the schemes}\label{ssec:existence}

Based on the estimates derived in the previous section, we can establish the 
existence of at least one solution to each scheme. This completes the proof of Theorem~\ref{thm:main1}.

\begin{prop}\label{prop:existence}
Let $\bc^0$ be defined by~\eqref{eq:cK0}. Then, for all $1\leq n\leq N$, the nonlinear system of equations \eqref{eq:mirror}--\eqref{eq:fluxint}, supplemented either with \eqref{eq:centered_flux}, \eqref{eq:SEDAN_flux}, \eqref{eq:activity_flux}, or \eqref{eq:Marianne_flux}, has a solution $(\bc^{n}, \bPhi^{n})\in \R^{\Tt} \times \R^{\Tt}$.
\end{prop}
\begin{proof}
The proof is a proof by induction; it relies on a topological degree argument~\cite{LS34,Dei85} at each time step. The idea is to transform in a continuous 
way our complex nonlinear system into a linear system while guaranteeing that a priori estimates controlling the 
solution remain valid all along the homotopy. We sketch the main ideas of the proof, making the homotopy (parametrized 
by $\lambda \in [0,3]$) explicit. 

We denote by $\bc^\star=\bc^{n-1} \in (0,1)^{\Tt}$ 
the discrete concentration at the previous time step. We are interested in the existence of zeros for a functional 
\[
\Hh: \begin{cases}
[0,3] \times (0,1)^\Tt \times \R^\Tt \to \R^\Tt \times \R^\Tt\\
(\lambda, \bc, \bPhi) \mapsto \Hh(\lambda, \bc,\bPhi)
\end{cases}
\]
that boils down to the scheme \eqref{scheme} when $\lambda = 3$. For sake of simplicity, instead of giving the definition 
of $\Hh$ for the different values of $\lambda$, we give a sense to the fact that  $\bc^{(\lambda)},\bPhi^{(\lambda)}$ is solution to 
$\Hh(\lambda,\bc^{(\lambda)},\bPhi^{(\lambda)})=0$.

We start with $\lambda\in [0,1]$: $\bc^{(\lambda)}$ is defined as the solution to the nonlinear system of equation
\be\label{eq:lambda<1}
m_K \frac{c_K^{(\lambda)} - c_K^{\star}}{\dt_n} + (1-\lambda)\sum_{\sig \in \Ee_{K,{\rm int}}}\tau_\sig (c_K^{(\lambda)} - c_{L}^{(\lambda)}) + 
\lambda \sum_{\sig \in \Ee_{K,\rm int}}\tau_\sig \left(r(c_K^{(\lambda)}) - r(c_{L}^{(\lambda)})\right)= 0, 
\ee
while $\bPhi^{(\lambda)} = 0$. Let us remark that for $\lambda=0$, it boils down to an invertible linear system of equations. Moreover, adapting the proof of Proposition~\ref{prop:E_disc} and using the property $(r(a)-r(b))(h(a)-h(b))\geq (a-b)(h(a)-h(b))$ for all $(a,b)\in (0,1)^2$, we get:
$$
E_\Tt(\bc^{(\lambda)},{\mathbf 0})-E_\Tt(\bc^{\star},{\mathbf 0})\leq - \dt_n \sum_{\sigma\in\Ee_{\rm int}}\tau_\sig (c_K^{(\lambda)}-c_L^{(\lambda)})(h(c_K^{(\lambda)})-h(c_L^{(\lambda)})).
$$
As the associated dissipation function defined by $\Dd(c_K,c_L)=(c_K-c_L)(h(c_K)-h(c_L))$ is clearly coercive in the sense of Lemma~\ref{lem:dissip}, 
we can deduce as in Lemma~\ref{lem:epsilon} the existence of $\epsilon_1>0$ such that $\epsilon_1<c_K^{(\lambda)}<1-\epsilon_1$ for all $K\in \Tt$ and for all $\lambda\in [0,1]$. 

For $\lambda \in [1,2]$, one lets our system evolve from the monotone scheme corresponding to $\lambda =1$ (which, due to Remark~\ref{rem:Sedan} corresponds to the Sedan scheme for the case without electrical potential) to the scheme 
with the expected numerical  fluxes $F_{K\sig}$. But the electrical potential is still blocked to $\bPhi^{(\lambda)} = {\bf 0}$, i.e., 
\be\label{eq:lambda<2}
\begin{gathered}
m_K \frac{c_K^{(\lambda)} - c_K^{\star}}{\dt_n} + (2-\lambda)\sum_{\sig \in \Ee_{K,{\rm int}} }\tau_\sig \left(r(c_K^{(\lambda)}) - r(c_{L}^{(\lambda)})\right)
+ (\lambda-1) \sum_{\sig \in \Ee_{K,{\rm int}}}F_{K\sig}^{(\lambda)}= 0,\\
F_{K\sig}^{(\lambda)}=\tau_\sig {\mathcal F}(c_K^{(\lambda)},c_L^{(\lambda)},0,0).
\end{gathered}
\ee
with $\Ff$ defined either by \eqref{eq:centered_flux}, \eqref{eq:SEDAN_flux}, \eqref{eq:activity_flux}, or \eqref{eq:Marianne_flux}.
Thanks to Lemma~\ref{lem:epsilon}, there exists $\eps_2>0$ such that $\epsilon_2<c_K^{(\lambda)}<1-\epsilon_2$ for all $K\in \Tt$ and for all $\lambda\in [1,2]$. 

During the last step, $\lambda \in [2,3]$, we reactivate progressively the electrical potential while keeping equation~\eqref{eq:scheme_c}.
Defining 
\[
\Phi^{D,(\lambda)}_\sig = (\lambda-2) \Phi^{D}_\sig, \qquad \forall \sig \in \Ee^D, 
\]
$(\bc^{(\lambda)},\bPhi^{(\lambda)})$ is defined, for all $\lambda\in [2,3]$ as the solution to the nonlinear system: for all $K\in\Tt$, 
\[
\begin{gathered}
m_K \frac{c_K^{(\lambda)} - c_K^{\star}}{\dt_n} +\sum_{\sig \in \Ee_{K,{\rm int}}}F_{K\sig}^{(\lambda)}=0, \mbox{ with } F_{K\sig}^{(\lambda)}=\tau_\sigma\Ff(c_K^{(\lambda)},c_L^{(\lambda)}, (\lambda-2) \Phi_K^{(\lambda)},(\lambda-2) \Phi_L^{(\lambda)}),\\
-\sum_{\sig \in \Ee_K} \tau_\sig D_{K\sig} \bPhi^{(\lambda)} = (\lambda-2) m_K (c_K^{(\lambda)}+c_K^{dp}).
\end{gathered}
\]
Thanks to Proposition~\ref{prop:Linf_Poisson}, one has $\left|\Phi_K^{(\lambda)}\right| \leq M_\Phi$  for all $K \in \Tt$ and for all  $\lambda \in [2,3]$.
Moreover, as in Lemma~\ref{lem:epsilon}, we can establish the existence of $\eps_3>0$ such that $\epsilon_3<c_K^{(\lambda)}<1-\epsilon_3$ for all $K\in \Tt$ and for all $\lambda\in [2,3]$. 

Finally, all along the homotopy parametrized by $\lambda \in [0,3]$, the solution $\left(\bc^{(\lambda)},\bPhi^{(\lambda)}\right)$ remains 
inside the compact subset $[\eps,1-\eps]^\Tt\times [-M_\Phi-1,M_\phi+1]^\Tt$ with $\eps=\min(\eps_1,\eps_2,\eps_3)$. Thus the topological degree 
corresponding to $\Hh(\lambda,\bc, \bPhi) = {\bf 0}$ and to the set $[\eps,1-\eps]^\Tt\times [-M_\Phi-1,M_\phi+1]^\Tt$ is equal to one all 
along the homotopy and  in particular for $\lambda = 3$. This ensures the existence of (at least) one solution to the scheme
\eqref{scheme}.
\end{proof}

\section{About the convergence towards a weak solution}\label{sec:convergence}

The goal of this section is to prove Theorem~\ref{thm:main2}, which states the 
convergence of the centered scheme \eqref{scheme}, \eqref{eq:centered_flux}, 
and the Sedan scheme \eqref{scheme}, \eqref{eq:SEDAN_flux}, towards a weak solution 
to the continuous problem in the sense of Definition~\ref{Def:weaksol}. Unfortunately, 
the proof we propose here neither applies to the activity base scheme 
\eqref{scheme}, \eqref{eq:activity_flux}, nor to the Bessemoulin-Chatard scheme 
\eqref{scheme}, \eqref{eq:Marianne_flux}. This does 
not mean that these schemes do not converge. Indeed, numerical evidences 
provided in Section~\ref{sec:numerics} seem to show that all the four schemes converge. 

We consider here a sequence $\left(\Tt_m, \Ee_m, \left(\x_K\right)_{K\in\Tt_m}\right)_{m\geq 1}$ 
of admissible discretization with $h_{\Tt_m}, \ov \bdt_m$ tending to $0$ as $m$ tends to $+\infty$, 
while the regularity $\zeta_{\Tt_m}$ remains uniformly bounded from below by a positive constant $\zeta^\star$. 
Theorem~\ref{thm:main1} provides the existence of a family of discrete solutions 
$\left(\bc_m, \bPhi_m\right)_m= \left( \left(c_K^n\right)_{K\in\Tt_m, 1 \leq n \leq N_m}, \left(\Phi_K^n\right)_{K\in\Tt_m, 1 \leq n \leq N_m}\right)$.
To prove Theorem~\ref{thm:main2}, we first establish in Section~\ref{ssec:compact} some compactness properties on the family 
of piecewise constant approximate solutions $(c_{\Tt_m,\bdt_m}, \Phi_{\Tt_m,\bdt_m})$ satisfied by the centered scheme and the Sedan scheme. Then we identify 
the limit as a weak solution in Section~\ref{ssec:identify}. 

In order to enlighten the notations, 
we remove the subscript $m$ as soon as it is not necessary for understanding.

\subsection{Reconstruction operators}\label{ssec:reconstruct}

In order to carry out the analysis of convergence, we introduce some reconstruction operators 
following the methodology proposed in~\cite{kangourou_2018}. 

The operators $\pi_\Tt: \R^\Tt \to L^\infty(\O)$ and $\pi_{\Tt,\bdt}:\R^{\Tt\times N}\to L^\infty((0,T)\times \O)$ are defined respectively by 
\[
\pi_\Tt\bu(\x) = u_K \quad\text{if}\; \x \in K, \qquad \forall \bu = \left(u_K\right)_{K\in\Tt}, 
\]
and
\[
\pi_{\Tt,\bdt}\bu(t,\x) = u_K^n \quad\text{if}\; (t,\x) \in (t_{n-1},t_n] \times K, \qquad \forall \bu = \left(u_K^n\right)_{K\in\Tt, 1 \leq n \leq N}. 
\]
These operators allow to pass from the discrete solution ${(\bc^n, \bPhi^n)}_{1 \leq n \leq N}$ to the approximate solution since 
\[
c_{\Tt,\bdt} = \pi_{\Tt,\bdt} \left(\bc^n\right)_n, \qquad \Phi_{\Tt,\bdt} = \pi_{\Tt,\bdt} \left(\bPhi^n\right)_n.
\]

But in order to carry out the analysis, we further need to introduce approximate gradient reconstruction. 
Since the boundary conditions play a crucial role in the definition of the gradient, we need to enrich the discrete solution 
by face values $\left(c_\sig^n\right)_{\sig\in\Ee_{\rm ext}, 1\leq n \leq N}$ and $\left(\Phi_\sig^n\right)_{\sig\in\Ee_{\rm ext}, 1\leq n \leq N}$ 
defined by $c_\sig^n = c_{K\sig}^n$ and $\Phi_\sig^n = \Phi_{K\sig}^n$. With a slight abuse of notations, we still denote by 
$\bc^n = \left((c_K^n)_{K\in\Tt},(c_\sig^n)_{\sig\in\Ee_{\rm ext}}\right)$ and  $\bPhi^n = \left((\Phi_K^n)_{K\in\Tt},(\Phi_\sig^n)_{\sig\in\Ee_{\rm ext}}\right)$
the elements of $(0,1)^{\Tt+\Ee_{\rm ext}}$ and  $\R^{\Tt+\Ee_{\rm ext}}$ containing both the cell values and 
the exterior faces values of the concentration and the potential respectively. 

For $\sig = K|L \in \Ee_{\rm int}$, we denote by $\Delta_\sig$ the diamond cell corresponding to $\sig$, that is 
the interior of the convex hull of $\{\sig, \x_K, \x_L\}$. For $\sig \in \Ee_{\rm ext}$, the diamond cell $\Delta_\sig$ 
is defined as the interior of the convex hull of  $\{\sig, \x_K\}$. 
The approximate gradient $\grad_{\Tt}: \R^{\Tt+\Ee_{\rm ext}} \to L^2(\O)^d$ we use in the analysis is merely weakly consistent (unless $d=1$) and takes its source in~\cite{CHLP03, EG03}.
It is piecewise constant on the diamond cells $\Delta_\sig$, and it is defined as follows:
\[
\grad_{\Tt} \bu(\x) = - d \frac{D_{K\sig} \bu}{d_\sig} \n_{K\sig} \quad \text{if}\; \x \in \Delta_\sig, \qquad \forall \bu \in \R^{\Tt+\Ee_{\rm ext}}.
\]
We also define $\grad_{\Tt,\bdt}: \R^{(\Tt+\Ee_{\rm ext})\times N} \to L^2(Q_T)^d$ by setting 
\[
\grad_{\Tt,\bdt}\bu(t,\cdot) = \grad_\Tt \bu^n \quad \text{if}\; t\in(t_{n-1},t_n], \qquad \forall \bu = \left(\bu^n\right)_{1\leq n\leq N} \in \R^{(\Tt+\Ee_{\rm ext})\times N}.
\]

Let us recall now some key properties to be used in the analysis. 
First, for all $\bu, \bv \in \R^{\Tt+\Ee_{\rm ext}}$, there holds 
\[
\sum_{\sig \in \Ee}\tau_\sig D_{K\sig} \bu D_{K\sig} \bv = \frac1d \int_\O \grad_{\Tt} \bu\cdot \grad_{\Tt} \bv \d\x.
\]
This implies in particular that
\be\label{eq:norm_L2H1}
\sum_{\sig \in \Ee}\tau_\sig |D_\sig \bu|^2 = \frac1d\int_{\O} |\grad_{\Tt} \bu|^2 \d\x, \qquad \forall \bu \in  \R^{\Tt+\Ee_{\rm ext}}.
\ee

\subsection{Compactness properties for the approximate concentration}\label{ssec:compact}

The goal here is to take advantage of the a priori estimates established in Section~\ref{ssec:apriori} 
to recover enough compactness for the sequences of approximate solutions.
\begin{lemma}\label{lem:L2H1_xi}
Let $(\bc_m,\bPhi_m)$ be the family of discrete solutions defined either by the centered scheme or by the Sedan scheme.
There exists $C$ depending only on $\Phi^D$, $\O$, $\zeta^\star$, $c_0$, $c^{dp}$ and $T$, such that 
\[
\iint_{Q_T} |\grad_{\Tt_m, \bdt_m} r(\bc_m)|^2+ \left(\pi_{\Tt_m, \bdt_m} r(\bc_m)\right)^2\d\x\d t \leq C.
\]

\end{lemma}
\begin{proof}
We get rid of the subscript $m$ for the ease of reading.
We will split the proof in two parts, first we focus on the proof of:
\be\label{eq:normL2H10_r}
\iint_{Q_T} |\grad_{\Tt, \bdt} r(\bc)|^2\d\x\d t \leq C.
\ee
Thanks to \eqref{eq:norm_L2H1}, we have
$$
\begin{aligned}
\iint_{Q_T} |\grad_{\Tt, \bdt} r(\bc)|^2&=d\sum_{n=1}^N \Delta t_n \sum_{\sigma\in\Ee_{\rm int}} \tau_\sig \vert D_\sig(r(\bc^n))\vert^2,\\
&=d\sum_{n=1}^N \Delta t_n \sum_{\sigma\in\Ee_{\rm int}} \tau_\sig \left(\wt\Cc_\sig^{n}\right)^2
\vert D_\sig(h(\bc^n))\vert^2,
\end{aligned}
$$
where we have defined the mean face concentrations $\left(\wt \Cc_\sig^n\right)_{\sig \in \Ee_{\rm int}, 1 \leq n \leq N}$ by setting
\be\label{eq:wtCsig}
\wt \Cc_\sig^n =
\frac{D_\sig r(\bc^n)}{D_\sig h(\bc^n)} \; \text{if}\; c_K^n \neq c_L^n \quad \text{and}\quad\wt \Cc_\sig^n =c_K^n\; \text{otherwise}, \qquad \forall \sig = K|L.
\ee
As noticed in Appendix~\ref{app:alacon}, $\wt \Cc_\sig^n$ is a mean value of $c_K^n$ and $c_L^n$; so that   $\wt\Cc_\sig^n\in (0,1)$ for all $\sigma\in \Ee_{\rm int}$. Moreover Lemma~\ref{lem:cctilde} proved in Appendix~\ref{app:alacon} ensures that there exists $G>0$ such that 
\be\label{eq:alacon}
\frac{\wt\Cc_\sig^n}{\Cc_\sig^n} \leq G, \qquad \forall \sig \in \Ee_{\rm int}, \; \forall n \in \{1,\dots, N\}.
\ee
Then, thanks to Young inequality, we obtain 
\begin{multline*}
\iint_{Q_T} |\grad_{\Tt, \bdt} r(\bc)|^2\leq 2dG \sum_{n=1}^N \dt_n \sum_{\sig \in \Ee_{\rm int}} \tau_\sig \Cc_\sig^{n} |D_\sig (h(\bc^{n})+\bPhi^n)|^2\\
+ 2d\sum_{n=1}^N \dt_n \sum_{\sig \in \Ee_{\rm int}} \tau_\sig |D_\sig\bPhi^n|^2.
\end{multline*}
Therefore,  Remark~\ref{rem:D_T2} and Lemma~\ref{lem:LinfH1Phi} yield the desired bound \eqref{eq:normL2H10_r}.

We now focus on the proof of:
\be\label{eq:normL2L2_r}
\iint_{Q_T} \left(\pi_{\Tt, \bdt} r(\bc)\right)^2\d\x\d t \leq C.
\ee
Noticing that for $c^*=\frac{1+\bar{c}}{2}>\bar{c}$:
    \[
    r(c)\leq \left( r(c)-r(c^*)\right)^++r(c^*),
    \]
    we have, using $(a+b)^2\leq 2(a^2+b^2)$:
    \be\label{eq:decomposition}
\iint_{Q_T} |\pi_{\Tt, \bdt} r(\bc) |^2 \d\x\d t \leq 2 \iint_{Q_T} |(\pi_{\Tt, \bdt} r(\bc) -r(c^*))^+|^2 \d\x\d t + 2 r(c^*)^2 m(\Omega)T.
\ee
Let $t\in[0,T]$ and $u=(\pi_{\Tt, \bdt} r(\bc) -r(c^*))^+(t)$. We intend to show that we have a $L^2$ bound on $u$ following ideas of 
\cite[Appendix A.1]{Ahmed_M2AN}. As $u$ is nonnegative, we have:
\be\label{eq:firstbound}
\int_\Omega |u-\bar{u}|^2 = \int_{u=0} \bar{u}^2 +\int_{\Omega\setminus \lbrace u=0\rbrace} |u-\bar{u}|^2\geq m(\lbrace u=0\rbrace)\bar{u}^2,
\ee
where $\bar{u}=\oint_\Omega u$.
Using Poincaré-Wirtinger inequality (see \cite[Theorem 5]{BCCHF15} or \cite[Theorem 2.1]{GG10}), we have:
\be\label{eq:PoincareWirt}
\int_\Omega |u-\bar{u}|^2\leq \frac{C}{\zeta_\Tt}\sum_{\sigma\in\cE_{\textup{int}}}\tau_\sigma (D_\sigma u)^2.
\ee
If we had a lower bound on $m(\lbrace u=0\rbrace)$, the equations \eqref{eq:PoincareWirt} and \eqref{eq:firstbound} would yield an upper bound on $\bar{u}$. By definition of $u$ and monotonicity of $r$, $u$ is zero if and only if $c$ is smaller than $c^*$. Using the monotonicity of integration and Lemma \ref{lem:mass}, we have:
\[
c^*(m(\Omega)-m(\lbrace u=0\rbrace))=\int_{c> c^*} c^* \leq \int_\Omega \pi_{\Tt, \bdt} \bc(t) =m(\Omega)\bar{c}.
\]
Hence, as $c^*=(1+{\bar c})/2$,
\[
m(\Omega)\frac{1-\bar{c}}{1+\bar{c}}\leq m(\lbrace u=0\rbrace).
\]
Finally, we have:
\[
\int_\Omega u^2\leq 2\left(\int_\Omega |u-\bar{u}|^2+\bar{u}^2\right)\leq C\sum_{\sigma\in\cE_{\textup{int}}}\tau_\sigma (D_\sigma u)^2.
\]
Using the definition of $u$, we have:
\[
\sum_{\sigma\in\cE_{\textup{int}}}\tau_\sigma (D_\sigma u)^2\leq \int_\Omega |\nabla_{\Tt, \bdt}r(\bc)(t)|^2.
\]
 Hence, integrating in time, and using \ref{eq:decomposition}:
 \[
\iint_{Q_T} |\pi_{\Tt, \bdt} r(\bc) |^2 \d\x\d t \leq C \iint_{Q_T} |\nabla_{\Tt, \bdt}r(\bc)(t)|^2+C.
 \]
We then deduce \eqref{eq:normL2L2_r} from \eqref{eq:normL2H10_r}. It concludes the proof of Lemma \ref{lem:L2H1_xi}.
\end{proof}

\begin{prop}\label{prop:compact}
Let $(\bc_m,\bPhi_m)$ be the family of discrete solutions defined either by the centered scheme or by the Sedan scheme.
In both cases, there exists $c \in L^\infty(Q_T; [0,1])$ with $r(c) \in L^2((0,T);H^1(\O))$ such that, up to a subsequence, 
\be\label{eq:conv_ae}
\pi_{\Tt_m, \bdt_m} \bc_m  \underset{m\to\infty} \longrightarrow c \quad \text{a.e. in $Q_T$},
\ee
\be\label{eq:conv_grad_r}
\grad_{\Tt_m, \bdt_m} r(\bc_m) \underset{m\to\infty} \longrightarrow \grad r(c) \quad \text{weakly in $L^2(Q_T)$}.
\ee
\end{prop}

\begin{remark}
The limit $c$ obtained in Proposition \ref{prop:compact} could {\it a priori} depend on the choosen subsequence or be different for the centered scheme and the Sedan scheme. In Section \ref{ssec:identify}, we will identify each limit as a weak solution to the initial problem.
\end{remark}

\begin{proof} 
Since $0 < \pi_{\Tt_m, \bdt_m} \bc_m < 1$ for all $m \geq 1$, there exists $c\in L^\infty(Q_T;[0,1])$ such that $\pi_{\Tt_m, \bdt_m} \bc_m$ 
tends to $c$ in the $L^\infty(Q_T)$ weak-$\star$ sense. We still have to establish the almost everywhere convergence 
as well as the fact that $r(c)$ belongs to $L^2((0,T);H^1(\O))$. To this end, we make use of the blackbox \cite[Theorem 3.9]{ACM17} 
which provides both the almost everywhere convergence and the identification of the limit of $\pi_{\Tt_m,\bdt_m} r(\bc_m)$ as $r(c)$.
We already have Lemma~\ref{lem:L2H1_xi} at hand and $c$ is bounded in $L^\infty$, so that, owing to~\cite{ACM17}, it is sufficient to prove that there exists some $C$ not 
depending on $m$ such that, for all $\bvarphi_m = \left(\varphi_K, \varphi_\sig\right)_{K,\sig} \in \R^{(\Tt_m + \Ee_{{\rm ext},m})\times \bdt_m}$, 
there holds
\be\label{eq:ACM17}
\left| \sum_n\dt_n\sum_{K\in\Tt_m} m_K \frac{c_K^{n} - c_K^{n-1}}{\dt_n} \varphi_K^n \right|\leq C \| \grad_{\Tt_m,\bdt_m} \bvarphi_m \|_{L^\infty(Q_T)}, 
\ee
for having, among other things, the desired convergence~\eqref{eq:conv_ae}. Using \eqref{eq:scheme_c} and the writing~\eqref{eq:FKsig_Csig}
of the fluxes, we obtain that
\begin{align*}
\left|\sum_n\dt_n\sum_{K\in\Tt_m} m_K \frac{c_K^{n} - c_K^{n-1}}{\dt_n} \varphi_K^n\right| =&\;
 \left|\sum_n\dt_n\sum_{\sig \in \Ee_{\rm int}}\tau_\sig \Cc_\sig^n D_{K\sig} (h(\bc^n) + \bPhi^n) D_{K\sig} \bvarphi \right|\\
 &
 \begin{aligned}
 {} \leq \;\left( \sum_n\dt_n\sum_{\sig \in \Ee_{\rm int}}\tau_\sig \Cc_\sig^n \left|D_{\sig} (h(\bc^n) + \bPhi^n)\right|^2\right)^{1/2}\\
 \times\left(\sum_n\dt_n  \sum_{\sig \in \Ee_{\rm int}}\tau_\sig |D_\sig \bvarphi^n|^2\right)^{1/2}
 \end{aligned}\\
 \leq & C \|\grad_{\Tt_m,\bdt_m} \bvarphi_m\|_{L^2(Q_T)} \leq C \|\grad_{\Tt_m,\bdt_m} \bvarphi_m\|_{L^\infty(Q_T)},
\end{align*}
thanks to  the boundedness of the dissipation in Remark \ref{rem:D_T2} which is a consequence of Proposition~\ref{prop:E_disc}.

Since $\grad_{\Tt_m,\bdt_m} r(\bc_m)$ is bounded in $L^2(Q_T)^d$, it converges weakly in $L^2(Q_T)^d$ towards some $\bU$.
Let us show that $\bU= \grad r(c)$. Let $\bV \in C^\infty_c(Q_T;\R^d)$ be a smooth compactly supported vector field, we set
\[
\bV_\sig^n = \frac{1}{m_\sig}\int_\sig \bV(t_n,\bgamma)\d\bgamma, \qquad \forall \sig \in \Ee, \; \forall n \in \{1,\dots, N\}.
\]
It permits to define a  piecewise constant function on the diamond cells and the time intervals: 
$$\bV_{\Ee,\bdt}(t,\x) = \bV_\sig^n\mbox{ if } (t,\x) \in (t_{n-1},t_n]\times\Delta_\sig.$$ 
It is clear that 
$\bV_{\Ee_m,\bdt_m}$ converges uniformly towards $\bV$ as $m \to \infty$ thanks to the regularity of $\bV$.
Then 
\[
\iint_{Q_T}\grad_{\Tt_m,\bdt_m} r(\bc) \cdot \bV_{\Ee_m,\bdt_m}\d\x\d t \underset{m\to \infty}\longrightarrow 
\iint_{Q_T} \bU\cdot\bV \d\x\d t.
\]
But, using the geometric relation $d m_{\Delta_\sig} = d_\sig m_\sig$ and the definition of $\bV_\sig^n$, one has 
\begin{multline*}
\iint_{Q_T}\grad_{\Tt,\bdt} r(\bc) \cdot \bV_{\Ee,\bdt}\d\x\d t 
= 
\sum_{n=1}^N \dt_n \sum_{\sig \in \Ee} d m_{\Delta_\sig} \frac{r(c_{K\sig}^n)-r(c_K^n)}{d_\sig} \n_{K\sig} \cdot \bV_\sig^n \\
=   - \sum_{n=1}^N \dt_n r(c_K^n) \sum_{\sig \in \Ee_K} \int_\sig \bV(\bgamma,t_n) \cdot \n_{K\sig} \d\bgamma 
=  - \sum_{n=1}^N \dt_n r(c_K^n) \int_K \div(\bV(\x,t_n)) \d\x. 
\end{multline*}
We deduce from the convergence of $\pi_{\Tt_m,\bdt_m}r(\bc)$ towards $r(c)$ that 
\[
\iint_{Q_T}\grad_{\Tt_m,\bdt_m} r(\bc) \cdot \bV_{\Ee_m,\bdt_m}\d\x\d t  \underset{m\to \infty}\longrightarrow  - \iint_{Q_T} r(c) \div(\bV) \d\x\d t 
=  \iint_{Q_T}\grad r(c) \cdot \bV \d\x\d t, 
\]
so that $\bU = \grad r(c)$. 
\end{proof}

We have two kind of face values at hand : $\Bigl(\Cc^{n}_\sig\Bigr)_{\sig \in \Ee_{\textup{int}}, 1 \leq n \leq N}$ and $\left(\wt\Cc^{n}_\sig \right)_{\sig \in \Ee_{\textup{int}}, 1 \leq n \leq N}$ defined respectively by \eqref{eq:Csig} and \eqref{eq:wtCsig}. 
Based on this, we can reconstruct 
two approximate concentration profiles $c_{\Ee,\bdt}$ and $\wt c_{\Ee,\bdt}$ that are piecewise constant on the diamond cells by setting 
\be\label{eq:c_Ee}
c_{\Ee,\bdt}(t,\x) = \begin{cases}
\Cc_\sig^{n} & \text{if}\;(t,\x) \in (t_{n-1},t_n] \times \Delta_\sig, \quad \sig \in \Ee_{{\rm int}}, \\
c_K^n & \text{if}\;\x \in \Delta_\sig, \quad \sig \in \Ee_{{\rm ext}}\cap \Ee_{K},
\end{cases}
\ee
and 
\be\label{eq:wt_c_Ee}
\wt c_{\Ee,\bdt}(t,\x) = \begin{cases}
\wt \Cc_\sig^{n} & \text{if}\;(t,\x) \in (t_{n-1},t_n] \times \Delta_\sig, \quad \sig \in \Ee_{{\rm int}}, \\
c_K^n & \text{if}\;\x \in \Delta_\sig, \quad \sig \in \Ee_{{\rm ext}}\cap \Ee_{K}.
\end{cases}
\ee

\begin{lemma}\label{lem:c_Ee}
For the centered scheme and the Sedan scheme,  there holds 
\begin{align}\label{eq:conv_c_Ee}
c_{\Ee_m,\bdt_m} \underset{m\to\infty}\longrightarrow c &\quad \text{in}\; L^p(Q_T)\; \text{for all $p\in[1,\infty)$},\\
\wt c_{\Ee_m,\bdt_m}\underset{m\to\infty}\longrightarrow c &\quad \text{in}\; L^p(Q_T)\; \text{for all $p\in[1,\infty)$},
\label{eq:conv_wt_c_Ee}
\end{align}
where $c$ is as in Proposition~\ref{prop:compact}.
\end{lemma}
\begin{proof}
We only prove~\eqref{eq:conv_c_Ee} since the proof of~\eqref{eq:conv_wt_c_Ee} is similar.
Here again, we get rid of  $m$ for clarity. 
Since $c_{\Tt,\bdt}$ converges almost everywhere to $c$ and remains bounded between $0$ and $1$, it converges in $L^p(Q_T)$.
$c_{\Ee,\bdt}$ is also uniformly bounded, hence it suffices to show that $\|c_{\Ee,\bdt}-c_{\Tt,\bdt}\|_{L^1(Q_T)}$ tends to $0$. 
Denoting by $\Delta_{K\sig}$ the half-diamond cell which is defined as the interior of the convex hull of $\{\x_K,\sig\}$ for $K\in\Tt$ 
and $\sig \in \Ee_K$, one has 
\begin{align*}
\|c_{\Ee,\bdt}-c_{\Tt,\bdt}\|_{L^1(Q_T)} \leq & \sum_{n=1}^N \dt_n \sum_{K\in\Tt} \sum_{\sig \in \Ee_K} m_{\Delta_{K\sig}}|c_K^n - \Cc_\sig^n| \\
\leq & \frac{h_\Tt}{d}  \sum_{n=1}^N \dt_n \sum_{K\in\Tt} \sum_{\sig \in \Ee_K} m_\sig|c_K^n - \Cc_\sig^n|
\end{align*}
where we have used the geometric relation $m(\Delta_{K\sig}) = \frac1d m_\sig {\rm dist}(\x_K,\sig) \leq \frac{h_\Tt}d m_\sig$.
For the internal faces, Lemma~\ref{lem:avg} (use~\eqref{eq:wt_avg} instead for $\wt\Cc_\sig^n$) implies that 
\[
|c_K^n - \Cc_\sig^n| + |c_{L}^n - \Cc_\sig^n| = |c_K^n - c_{L}^n|, \qquad \forall \sig = K|L.
\]
Therefore, we obtain that 
\begin{align*}
\|c_{\Ee,\bdt}-c_{\Tt,\bdt}\|_{L^1(Q_T)} \leq& \frac{h_\Tt}d\sum_{n=1}^N \dt_n \sum_{\sig \in \Ee_{\rm int}} m_\sig D_\sig \bc^n\\
\leq &\frac{h_\Tt}d \left(\sum_{n=1}^N \dt_n \sum_{\sig \in \Ee_{\rm int}} m_\sig d_\sig\right)^{1/2}
\left(\sum_{n=1}^N \dt_n \sum_{\sig \in \Ee_{\rm int}} \tau_\sig |D_\sig \bc^n|^2\right)^{1/2}.
\end{align*}
Since $|r(a) - r(b)| > |a-b|$ for all $a,b \in (0,1)$, we deduce from Lemma~\ref{lem:L2H1_xi} that 
\[
\|c_{\Ee,\bdt}-c_{\Tt,\bdt}\|_{L^1(Q_T)} \leq C h_\Tt.
\]
\end{proof}

\subsection{Convergence towards a weak solution}\label{ssec:identify}

\begin{prop}\label{prop:convPhi}
Let $c$ be as in Proposition~\ref{prop:compact} and let $\Phi \in L^\infty(Q_T) \cap L^\infty((0,T);H^1(\O))$ 
be the solution to
the Poisson equation~\eqref{Poisson equation} with boundary conditions~\eqref{eq:BC_Phi}. Then, for the centered scheme and the Sedan scheme, there holds 
\be\label{eq:convPhi1}
\pi_{\Tt_m, \bdt_m} \bPhi_m \underset{m\to\infty} \longrightarrow \Phi\; \text{in the $L^\infty(Q_T)$ weak-$\star$ sense}, 
\ee
and
\be\label{eq:convPhi2}
\grad_{\Tt_m, \bdt_m} \bPhi_m \underset{m\to\infty} \longrightarrow \grad \Phi \quad \text{in the $L^\infty((0,T);L^2(\O)^d)$ weak-$\star$ sense}.
\ee
\end{prop}
\begin{proof}
The existence of some $\Phi \in L^\infty(Q_T)$ such that~\eqref{eq:convPhi1} holds is a straightforward consequence of Lemma~\ref{lem:LinfPhi}, 
whereas the existence of some $\bU \in L^\infty((0,T);L^2(\O)^d)$ such that 
$\grad_{\Tt_m, \bdt_m} \bPhi$ tends to $\bU$ as $m$ tends to $\infty$ follows from Lemma~\ref{lem:LinfH1Phi} 
together with \eqref{eq:norm_L2H1}. The proof of $\bU=\grad \Phi$ is similar to the proof of Proposition~\ref{prop:compact}.

We show now that $\Phi$ satisfies the Poisson equation~\eqref{Poisson equation}.
Let $\psi \in C^\infty_c([0,T]\times\left\{\O\cup\Gamma^N\right\})$, then define $\psi_K^n=\psi(\x_K,t_n)$ and $\psi_\sig^n = \psi(\x_\sig,t_n)$ for 
$1\leq n\leq N$, $K\in\Tt$ and $\sig\in\Ee_{\rm ext}$. Following~\cite{DE_FVCA8} (see~\cite{CVV99} for a practical example), 
one can reconstruct a second approximate gradient operator $\wh \grad_{\Tt}: \R^\Tt \to L^\infty(\O)^d$ such that 
\[
\int_{\O}\grad_\Tt \bu \cdot \wh\grad_\Tt \bv \d\x = \sum_{\sig \in \Ee} \tau_\sig D_{K\sig} \bu D_{K\sig} \bv, \qquad \forall \bu, \bv \in \R^\Tt, 
\]
and which is strongly consistent, i.e., 
\be\label{eq:ov_grad}
\wh \grad_\Tt \bpsi^n \underset{h_\Tt\to0}\longrightarrow \grad \psi(\cdot, t_n)\;\;\text{uniformly in}\; \ov\O, \quad \forall n \in \{1,\dots, N\},
\ee
thanks to the smoothness of $\psi$. The scheme~\eqref{eq:scheme_Phi} then reduces to 
\[
\int_\O \grad_\Tt \bPhi^n \cdot \wh\grad_\Tt \bpsi^n \d\x= \int_\O \pi_\Tt(\bc^n  + c^{dp}) \pi_\Tt \bpsi^n \d\x, \quad \forall n \in \{1,\dots, N\}, \;
\forall \bpsi \in \R^{(\Tt+\Ee_{\rm ext})\times N}.
\]
Integrating with respect to time over $(0,T)$ and passing to the limit $h_\Tt, \ov \bdt \to 0$ thanks to Proposition~\ref{prop:compact}, \eqref{eq:convPhi2} and \eqref{eq:ov_grad} then yields 
\[
\iint_{Q_T} \grad \Phi \cdot \grad \psi \d\x\d t = \iint_{Q_T} (c + c^{dp}) \psi \d\x\d t, \qquad \forall \psi \in C^\infty_c([0,T]\times \O\cup\Gamma^N).
\]
In particular, \eqref{eq:weak_Phi} holds for almost every $t\in(0,T)$.

The last point to be checked is the boundary condition for $\Phi$, i.e., that $\Phi = \Phi^D$ on $(0,T) \times\Gamma^D$. 
This can be proved for instance following the lines of \cite[Section 4]{BCH13}. 
\end{proof}

\begin{remark}[Enhanced convergence properties]\label{rem:convPhi}
The convergence described in Proposition~\ref{prop:convPhi} is not optimal. One can rather easily show that the convergence 
of $\pi_{\Tt_m,\bdt_m} \bPhi_m$ towards $\Phi$ for the strong topology of $L^2(Q_T)$ due to the strong convergence of 
$\pi_{\Tt_m, \bdt_m} \bc_m$ to $c$ established in Proposition~\ref{prop:compact}. Moreover, 
one can establish the strong convergence of $\wh \grad_{\Tt_m,\bdt_m} \bPhi_m$ 
towards $\grad \Phi$, where the gradient reconstruction operator $\wh \grad_{\Tt_m,\bdt_m}$ is the extension to the time-space domain 
$Q_T$ of the operator $\wh \grad_{\Tt_m}$ used in the proof of Proposition~\ref{prop:convPhi}.  We refer to~\cite{DE_FVCA8} for 
details on these enhanced convergence properties.
\end{remark}

\begin{prop}\label{prop:conv_c}
Let $c$ be as in Proposition~\ref{prop:compact}, then $c$ satisfies the weak formulation~\eqref{eq:weak_c}.
\end{prop}
\begin{proof}
Let $\varphi \in C^\infty_c([0,T)\times\ov\O)$, then define $\varphi_K^n = \varphi(\x_K,t_n)$ for all $n \in \{0,\dots, N\}$ and $K \in \Tt$.  Multiplying~\eqref{eq:scheme_c}  by $\dt_n\varphi_K^{n-1}$, then summing over $K\in\Tt$ and $n\in\{1,\dots, N\}$ 
and using expression~\eqref{eq:FKsig_Csig} for the fluxes leads to 
\be\label{eq:T123}
T_1 + T_2 + T_3 = 0, 
\ee
where we have set 
\begin{align*}
T_1 = &\sum_{n=1}^N\sum_{K\in\Tt}m_K(c_K^n - c_K^{n-1}) \varphi_K^{n-1},\\
T_2 = &\sum_{n=1}^N\dt_n\sum_{\sig \in \Ee}\tau_\sig\Cc_\sig^n D_{K\sig}h(\bc^n)D_{K\sig}\bvarphi^{n-1},\\
T_3 = &\sum_{n=1}^N\dt_n\sum_{\sig \in \Ee}\tau_\sig\Cc_\sig^n D_{K\sig}\bPhi^nD_{K\sig}\bvarphi^{n-1}.
\end{align*}
The term $T_1$ can be rewritten as 
\[
T_1 = \sum_{n=1}^N\dt_n\sum_{K\in\Tt}m_K c_K^n \frac{\varphi_K^{n-1}-\varphi_K^n}{\dt_n} - \sum_{K\in\Tt}m_K c_K^0 \varphi_K^0,
\]
so that it follows from the convergence of $\pi_{\Tt,\bdt} \bc$ towards $c$ and of $\pi_\Tt \bc^0$ towards $c^0$ together with the regularity of $\varphi$ that 
\be\label{eq:T1}
T_1 \underset{m\to \infty} \longrightarrow -\iint_{Q_T} c \p_t \varphi \d\x\d t - \int_\O c^0 \varphi(0,\cdot)\d\x. 
\ee
On the other hand, the term $T_3$ can be rewritten as 
\[
T_3 = \iint_{Q_T} c_{\Ee,\bdt} \grad_{\Tt,\bdt} \bPhi\cdot \wh \grad_{\Tt,\bdt} \bvarphi \d\x\d t,
\]
where  $\wh \grad_{\Tt,\bdt}$ is the strongly consistent gradient reconstruction operator introduced in the proof of 
Proposition~\ref{prop:convPhi} and in Remark~\ref{rem:convPhi}. In particular, due to the smoothness of $\varphi$, 
$\wh \grad_{\Tt,\bdt} \bvarphi$ converges uniformly towards $\grad \varphi$.
Therefore, it follows from Lemma~\ref{lem:c_Ee} and Proposition~\ref{prop:convPhi} that 
\be\label{eq:T3}
T_3 \underset{m \to \infty} \longrightarrow \iint_{Q_T} c \grad \Phi\cdot \grad \varphi \d\x\d t.
\ee
Define the term 
\[
\wt T_2= 
 \sum_{n=1}^N\dt_n\sum_{\sig \in \Ee}\tau_\sig \wt \Cc_\sig^nD_{K\sig}h(\bc^n)D_{K\sig}\bvarphi^{n-1}
  = \iint_{Q_T}\grad_{\Tt,\dt} r(\bc)\cdot \wh\grad_{\Tt,\dt}\bvarphi \d\x\d t,  
\]
then it follows from Proposition~\ref{prop:compact} that 
\[
\wt T_2 \underset{m\to \infty} \longrightarrow \iint_{Q_T}  \grad r(c)\cdot \grad \varphi \d\x\d t.
\]
Therefore, it only remains to show that 
$|T_2 - \wt T_2|$ tends to $0$ to conclude the proof of Proposition~\ref{prop:conv_c}.
Thanks to the triangle and Cauchy-Schwarz inequalities, one has 
\begin{align*}
|T_2 - \wt T_2| \leq&\;  \sum_{n=1}^N\dt_n\sum_{\sig \in \Ee}\tau_\sig\left|\Cc_\sig^n-\wt\Cc_\sig^n\right| D_{\sig}h(\bc^n)D_{\sig}\bvarphi^{n-1} \\
\leq & \; \left(\sum_{n=1}^N\dt_n\sum_{\sig \in \Ee}\tau_\sig \Cc_\sig^n |D_\sig h(\bc^n)|^2\right)^{1/2}
\left(\sum_{n=1}^N\dt_n\sum_{\sig \in \Ee}\tau_\sig \frac{(\Cc_\sig^n - \wt\Cc_\sig^n)^2}{\Cc_\sig^n} |D_\sig \bvarphi^{n-1}|^2\right)^{1/2}.
\end{align*}
The first term in the right-hand side is uniformly bounded thanks to Remark~\ref{rem:D_T2}, to $0\leq \Cc_\sig^n \leq 1$ and to Lemma~\ref{lem:LinfH1Phi} (we can adapt a part of the proof of \eqref{eq:normL2H10_r}) . 
Thus our problem amounts to show that 
\be\label{eq:Rto0}
\mathcal{R} :=\sum_{n=1}^N\dt_n\sum_{\sig \in \Ee}\tau_\sig \frac{(\Cc_\sig^n - \wt\Cc_\sig^n)^2}{\Cc_\sig^n} |D_\sig \bvarphi^{n-1}|^2
\underset{m\to\infty}\longrightarrow 0.
\ee
Let us reformulate $\mathcal{R}$ as 
\[
\mathcal{R} :=\sum_{n=1}^N\dt_n\sum_{\sig \in \Ee}\tau_\sig |\Cc_\sig^n - \wt\Cc_\sig^n| \left|1 - \frac{\wt\Cc_\sig^n}{\Cc_\sig^n}\right|
 |D_\sig \bvarphi^{n-1}|^2.
\]
Thanks to \eqref{eq:alacon}, the quantity $\left|1 - \frac{\wt\Cc_\sig^n}{\Cc_\sig^n}\right|$ is uniformly bounded, 
whereas the regularity of $\varphi$ implies that 
$D_\sig \bvarphi^{n-1} \leq \|\grad \varphi\|_\infty d_\sig$. 
Putting this in the above expression of $\mathcal R$, we obtain that 
\[
\mathcal{R} \leq C \|c_{\Ee,\bdt} - \wt c_{\Ee,\bdt}\|_{L^1(Q_T)} \underset{m\to \infty}\longrightarrow 0,
\]
thanks to Lemma~\ref{lem:c_Ee}.
\end{proof}
\section{Numerical comparison of the schemes}\label{sec:numerics}

The numerical examples \cite{UnipolarDriftDiffusion} have been implemented in the Julia language
\cite{bezanson2017julia} based on the package \texttt{VoronoiFVM.jl} \cite{VoronoiFVM}
which realizes the implicit Euler Voronoi finite volume method for nonlinear diffusion-convection-reaction
equations on simplicial grids. The resulting nonlinear systems of equations are solved using Newton's method
with parameter embedding. An advantage of the implementation in Julia is the availability of \texttt{ForwardDiff.jl}
\cite{RevelsLubinPapamarkou2016},
an automatic differentiation package. This package allows the assembly of analytical Jacobians based on
a generic implementation of nonlinear parameter functions without the need to write source code
for derivatives.

\subsection{1D time evolution and convergence test}

\begin{figure}[h]
  \begin{center}
    \includegraphics[width=0.5\linewidth]{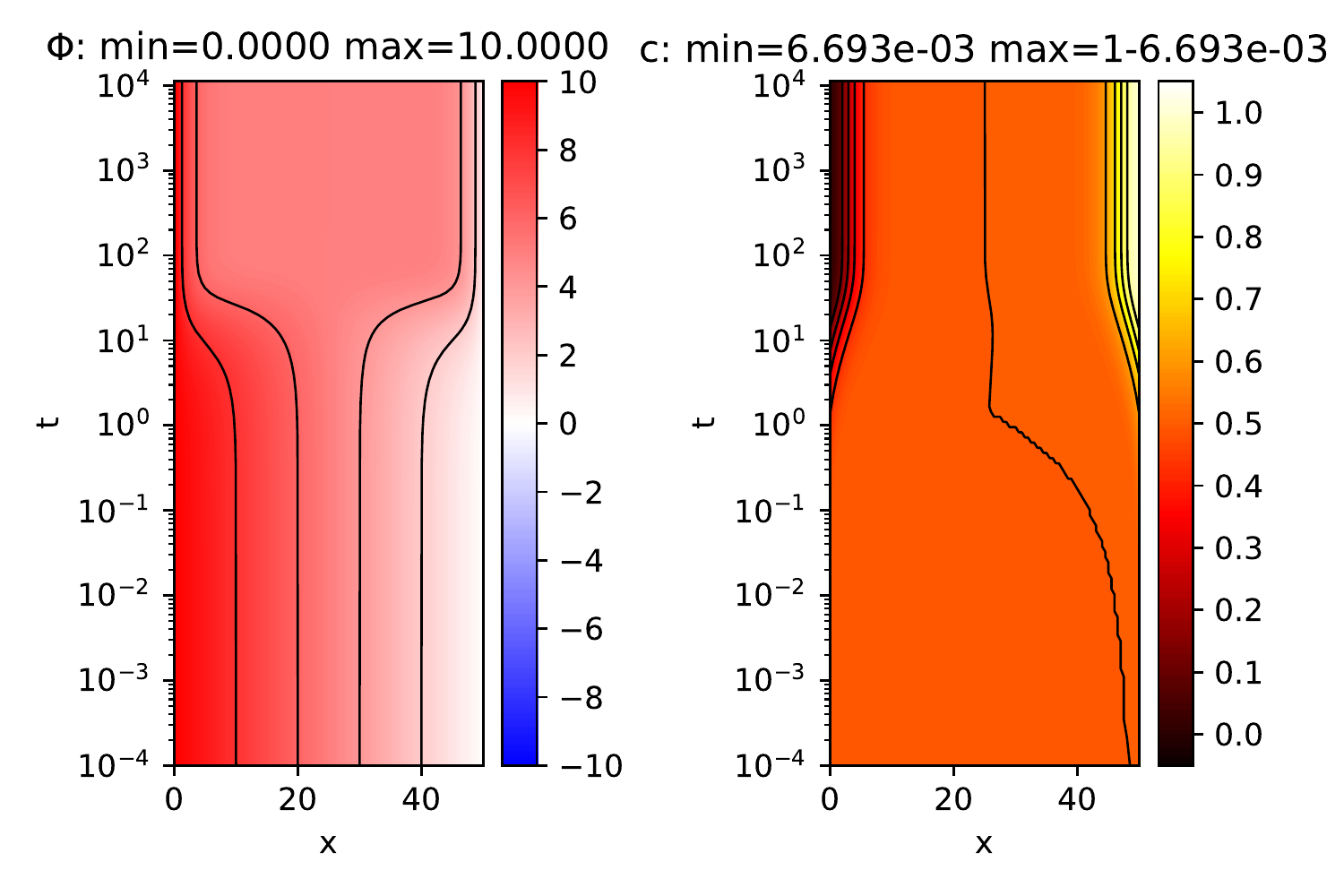}
    \includegraphics[width=0.45\linewidth]{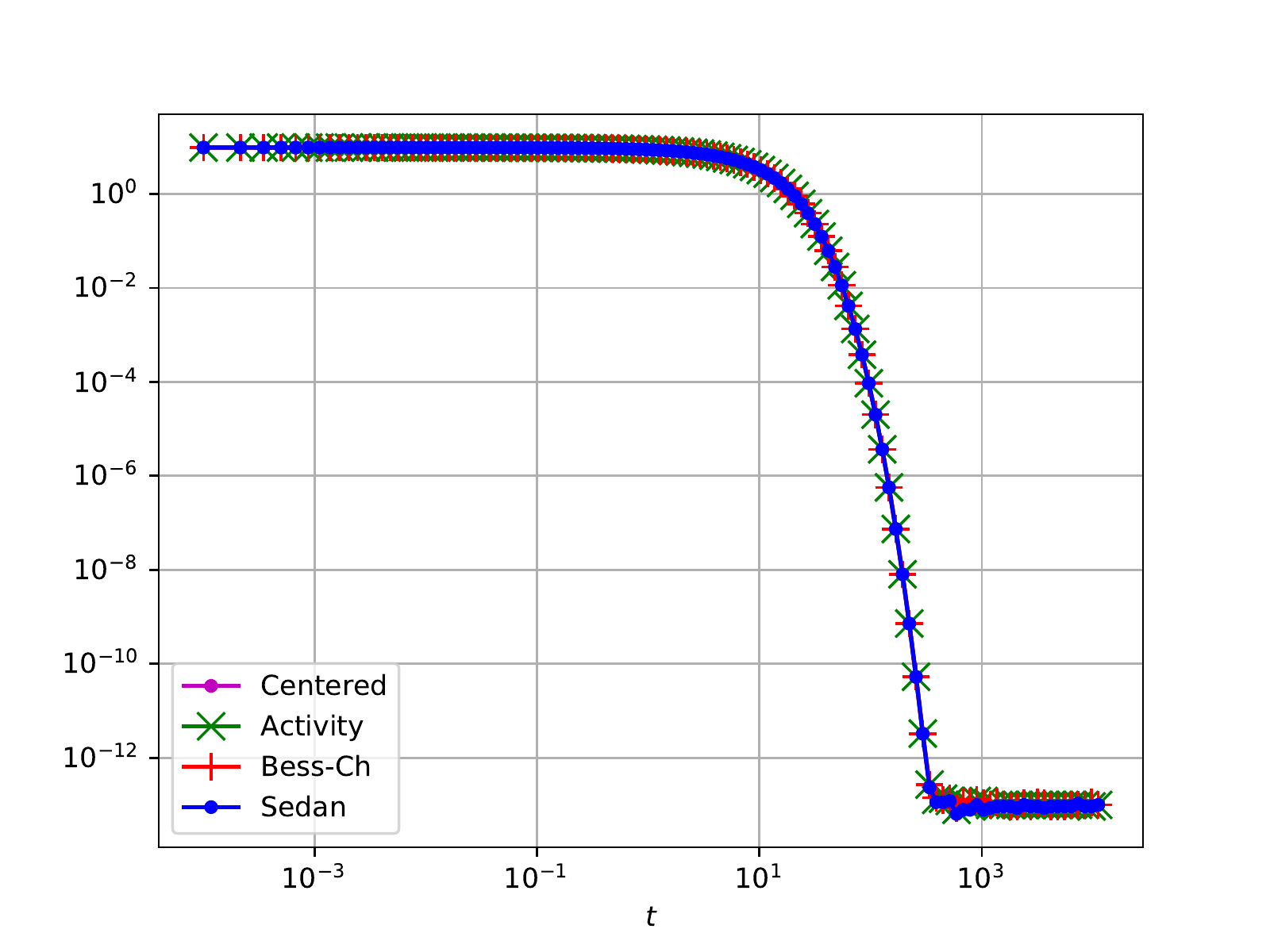}
    \caption{Left: time evolution of solution on domain $\Omega=(0,50)$ from constant initial value $c=\frac12$
      with Dirichlet boundary conditions $\Phi(0)=10$, $\Phi(50)=0$, $c^{dp}=-\frac12$ and homogeneous Neumann boundary conditions
      for $c$. Right: Evolution of the relative free energy  according to \eqref{eq:E}.}
    \label{fig:evoli}
  \end{center}
\end{figure}

\begin{figure}[h]
  \begin{center}
    \includegraphics[width=0.5\linewidth]{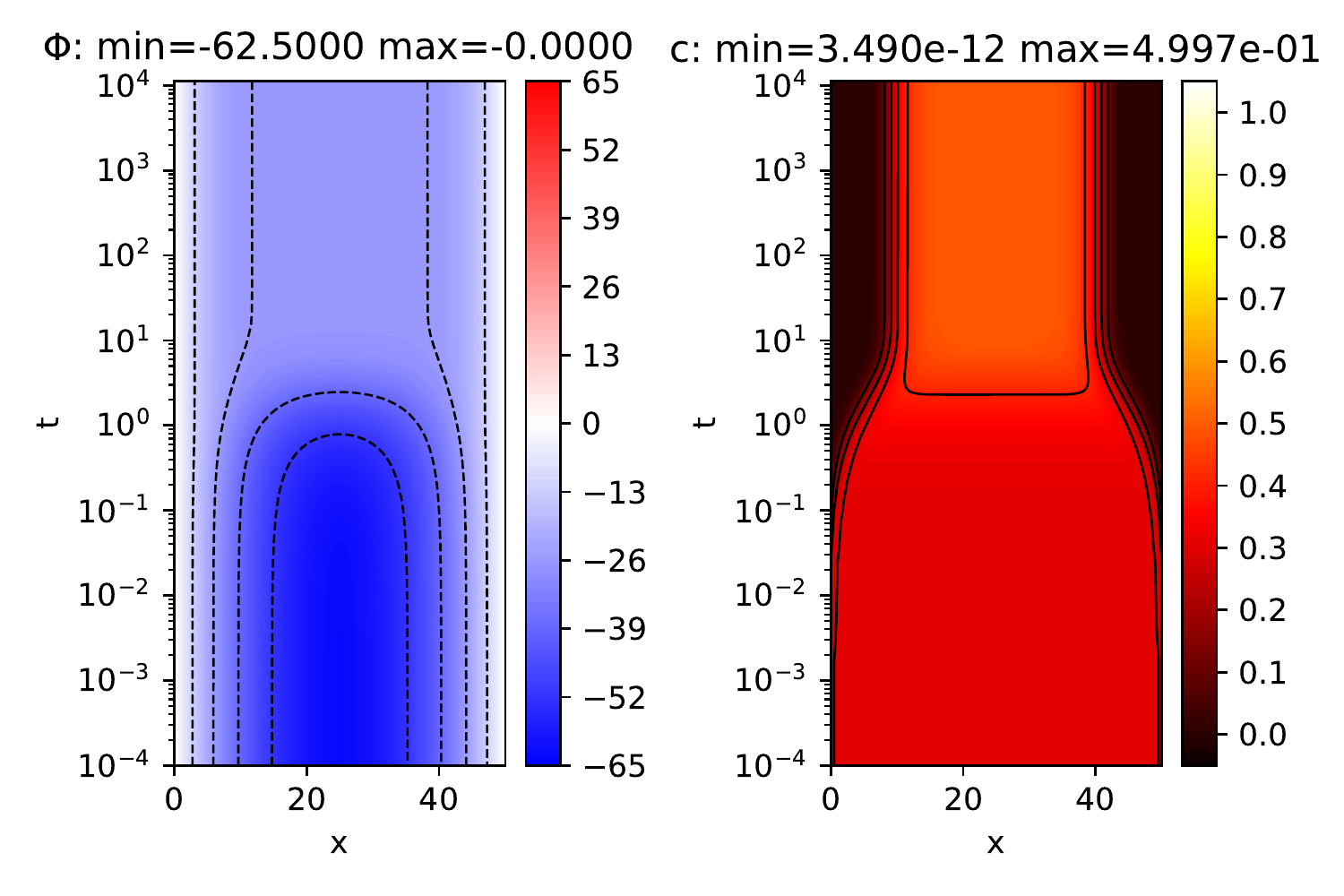}
    \includegraphics[width=0.45\linewidth]{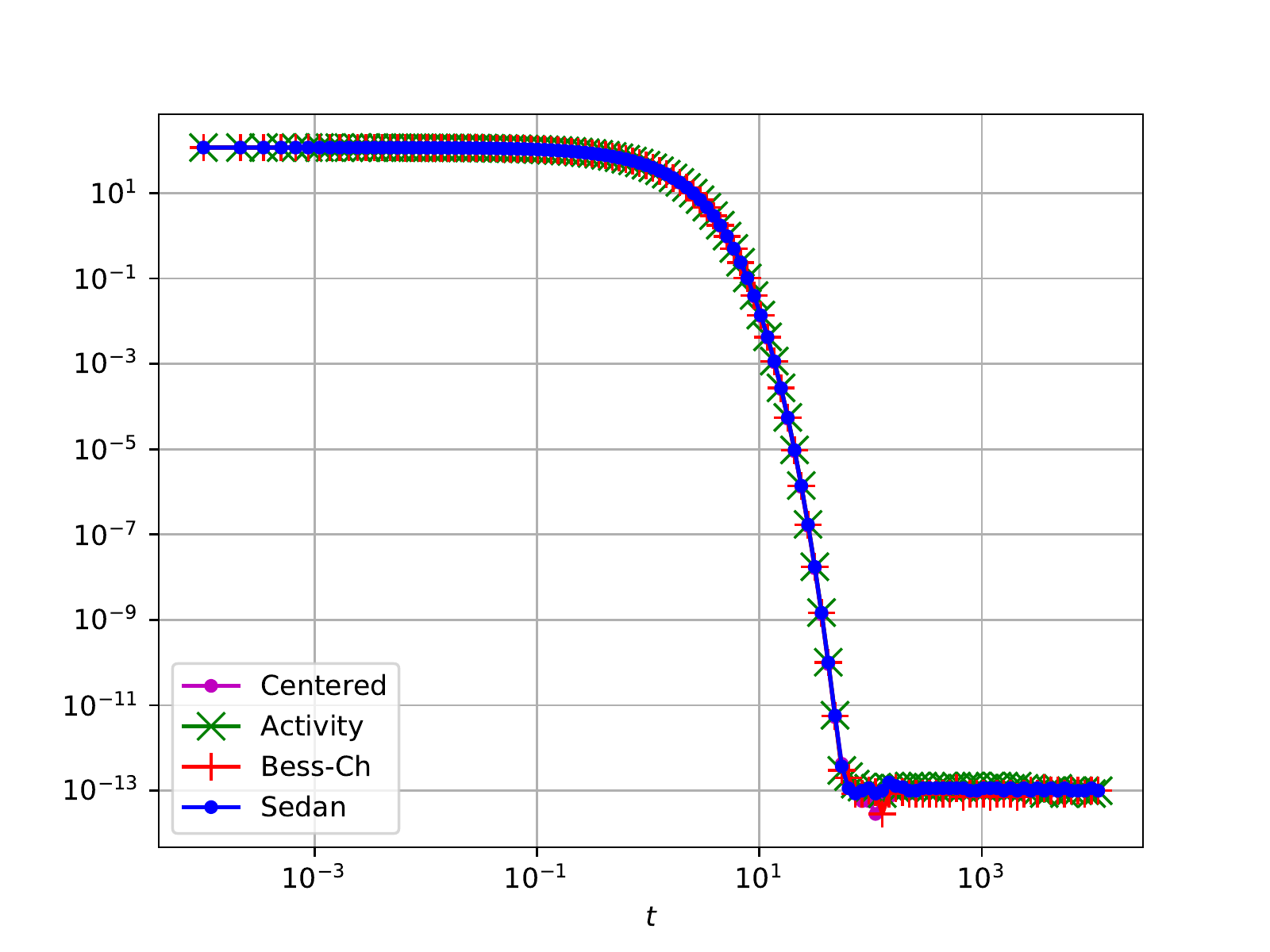}
    \caption{Left: time evolution of solution on domain $\Omega=(0,50)$ from constant initial value $c=0.3$
      with Dirichlet boundary conditions $\Phi(0)=0$, $\Phi(50)=0$, $c^{dp}=-\frac12$ and homogeneous Neumann boundary conditions
      for $c$. Right: Evolution of the relative free energy  according to \eqref{eq:E}.}
    \label{fig:evolii}
  \end{center}
\end{figure}

\begin{figure}[h]
  \begin{center}
    \includegraphics[width=0.5\linewidth]{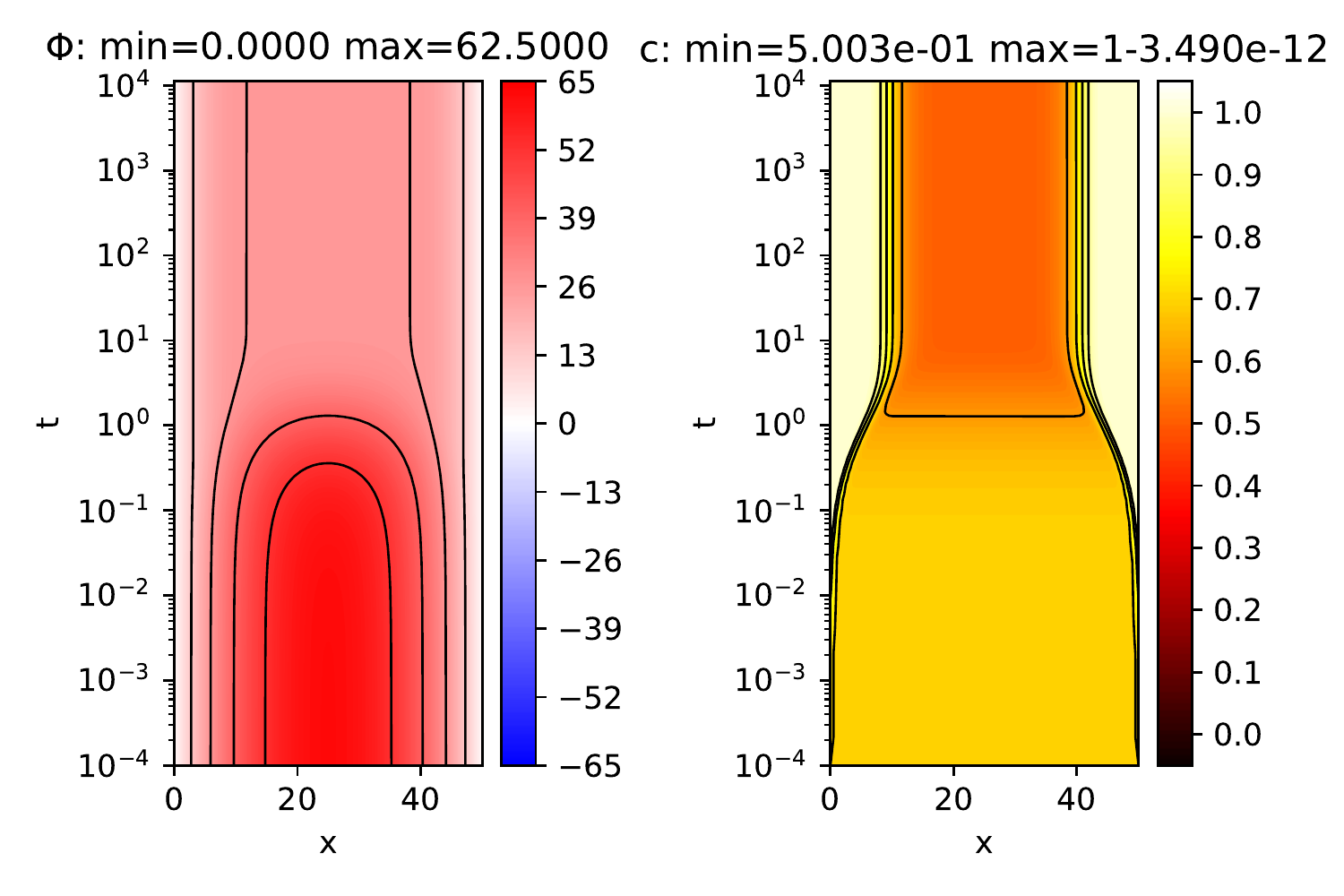}
    \includegraphics[width=0.45\linewidth]{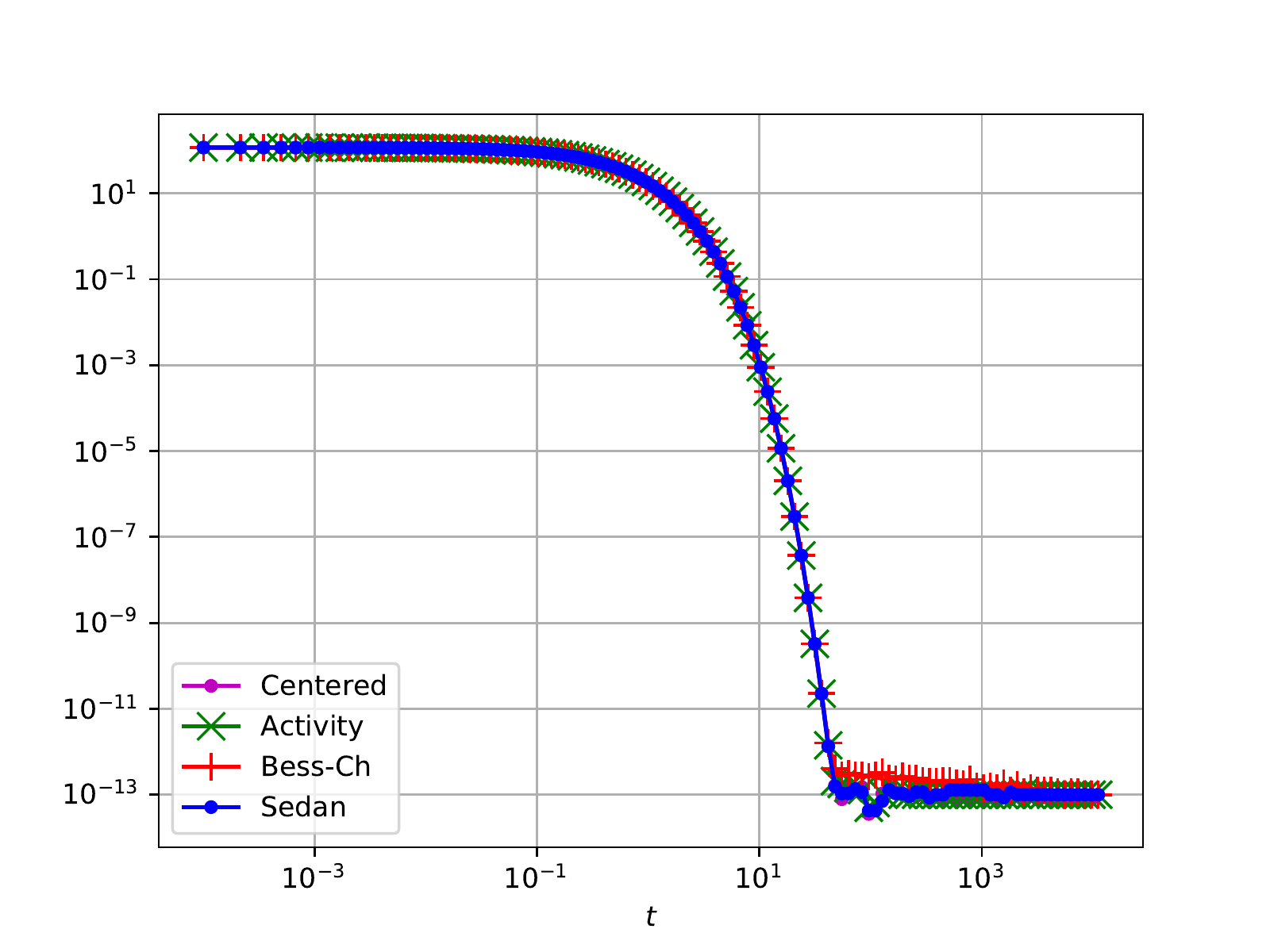}
    \caption{Left: time evolution of solution on domain $\Omega=(0,50)$ from constant initial value $c=0.7$
      with Dirichlet boundary conditions $\Phi(0)=0$, $\Phi(50)=0$, $c^{dp}=-\frac12$ and homogeneous Neumann boundary conditions
      for $c$. Right: Evolution of the relative free energy  according to \eqref{eq:E}.}
    \label{fig:evoliii}
  \end{center}
\end{figure}

\begin{figure}[h]
  \begin{center}
    \includegraphics[width=0.49\linewidth]{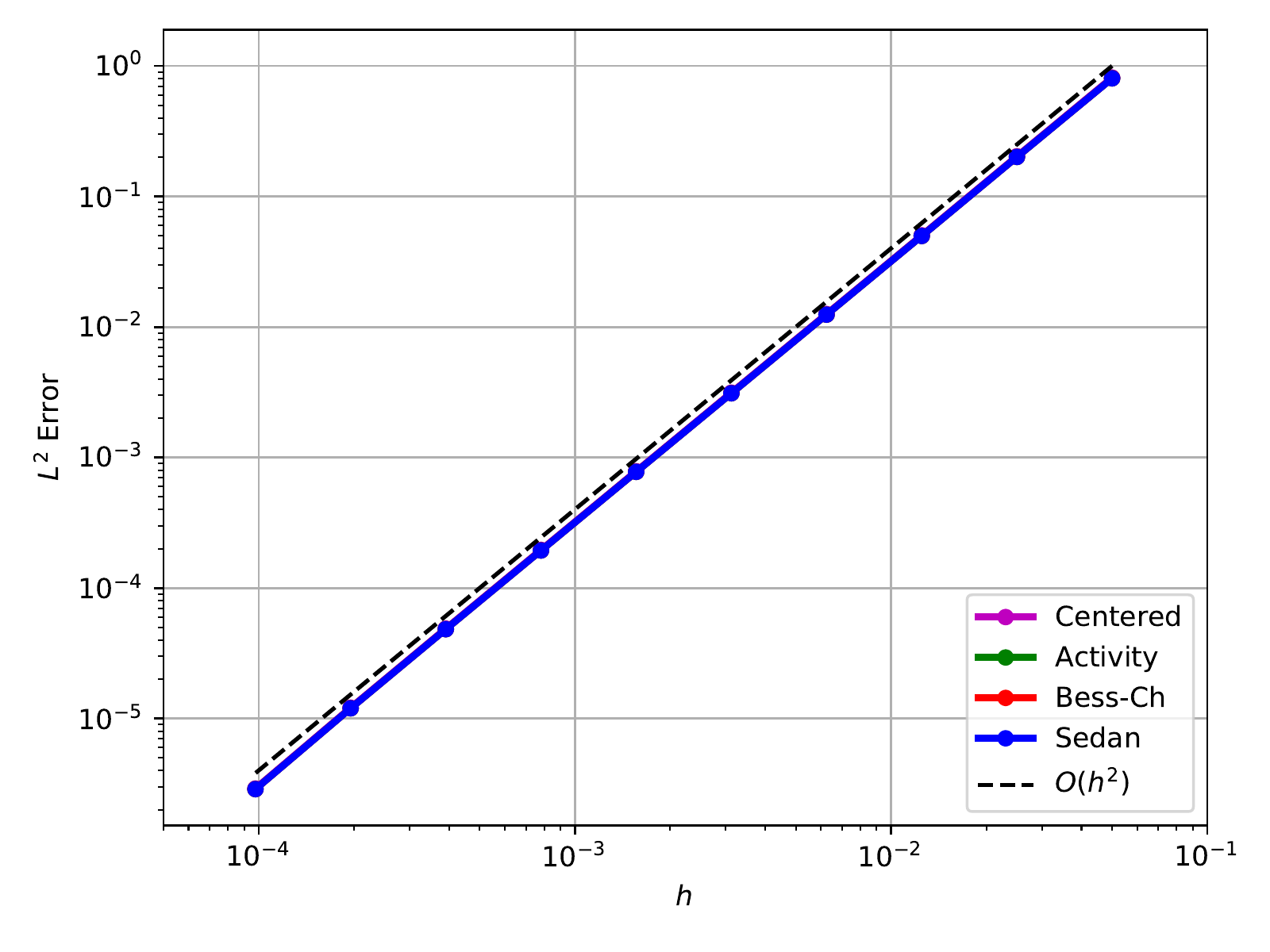}
    \includegraphics[width=0.49\linewidth]{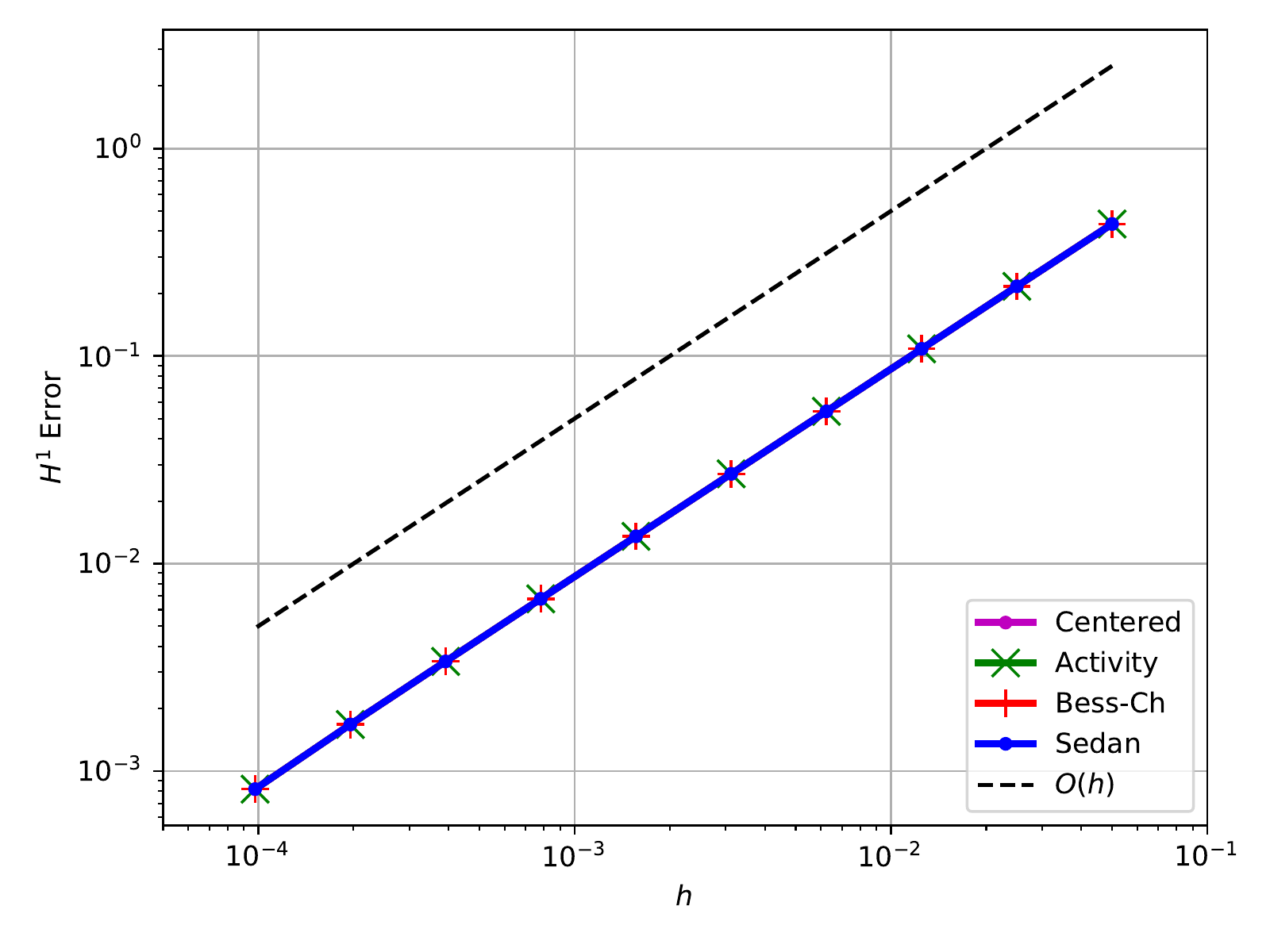}
    \caption{Convergence behavior of the different schemes for the case depicted
      in Fig. \ref{fig:evoli}: comparison of solutions at $t=10$.
      Left: $L^2$-error, right: $H^1$ error. Correspondence to the equation in the paper:
      ``Centered'': \eqref{eq:centered_flux}, ``Sedan'': \eqref{eq:SEDAN_flux}, ``Activity'': \eqref{eq:activity_flux}, ``Bess-Ch'': \eqref{eq:Marianne_flux}.
    }
    \label{fig:schemes-tran}
  \end{center}
\end{figure}

The first group of examples  considers the problem as described by
\eqref{eq:cons_loc}-\eqref{Poisson equation} in a one-dimensional domain
with Dirichlet boundary conditions for $\Phi$ and homogeneous Neumann boundary conditions
for $c$. We regard the time evolution from
a zero potential $\Phi$  and constant concentrations $c_0$. In all examples
we assume a constant doping concentration $c^{dp}=-\frac12$.  Calculations
have been performed with subdivision of the domain $\Omega=(0,L)$ into 100 control
volumes. Time steps have been chosen in a geometric progression $t_i=t_1*\delta^i$
with $\delta=1.15$ and $t_1=10^{-4}$.

In the first example (Fig. \ref{fig:evoli}), $c_0=0.5$, and the initial amount of charge carriers exactly matches the
amount of doping. With the start of time evolution, at $x=0$ a potential of $10$ is
applied leading to a redistribution  of the charge carrier concentration which for  large $t$ approaches a steady
state with two space charge regions at the boundaris with opposite charge and an electroneutral region with $c=0.5$ in the center of the domain. We remark, that the $c$ stays in the range $(0,1)$, and that the energy \eqref{eq:E} decreases
during time evolution for all four schemes discussed in this paper. We also remark that for zero applied
potential, the constant values $\Phi=0$ and $c=0.5$ would comprise a solution for all $t>0$.

Fig. \ref{fig:evolii} considers the case $c_0=0.3$. The available amount of charge carriers is not
able to compensate the amount of doping. At the end of the time evolution, the charge carriers
are concentrated in the center of the domain, establishing an electroneutral region. At both
boundaries, depletion boundary layers create equally charged space charge regions
due to the lack of charge carriers able to compensate
the doping.

Fig. \ref{fig:evoliii}  considers the case  $c_0=0.7$ which in  sense is
symmetric to the previous one. There is again an electroneutral region
in the  center, and  this time,  ``superfluous'' charge carriers are
forced to enrichment boundary layers.

Fig. \ref{fig:schemes-tran} provides a comparison of the convergence behavior for
the test case discussed in Fig. \ref{fig:evoli}. We compare the solutions at a moment
of time where we observe a rather large descent of the relative free energy based
on a reference solution obtained on a fine grid of 40960 nodes.
We observe first order convergence in the $H^1$ norm and second order convergence in
the $L^2$ norm. No significant difference between the results for the various schemes.

\subsection{1D stationary convergence test}

\begin{figure}[h]
  \begin{center}
    \includegraphics[width=0.5\linewidth]{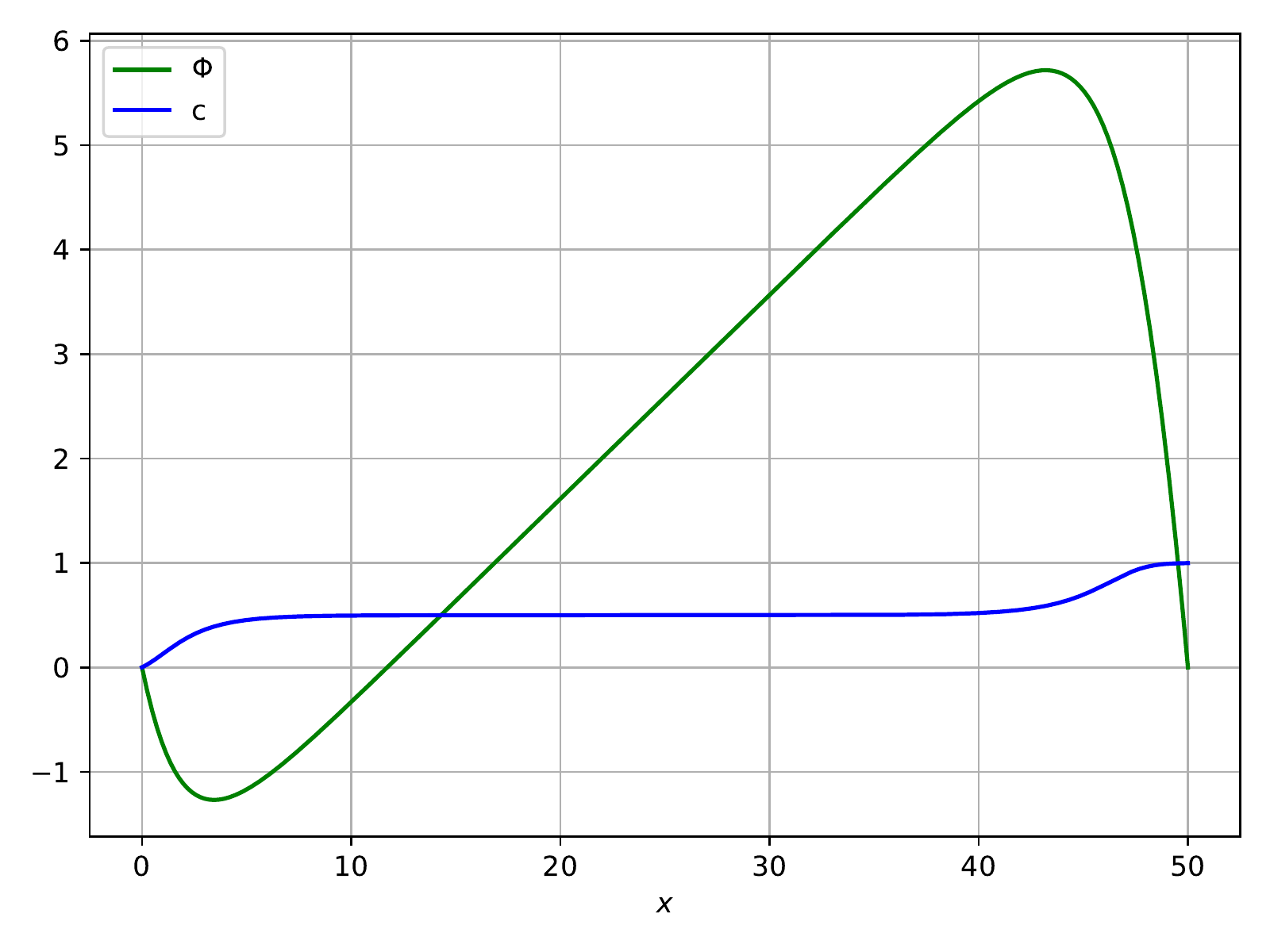}
    \caption{Stationary solution with
      Dirichlet boundary conditions for $c$ and $\Phi$.}
    \label{fig:refsol}
  \end{center}
\end{figure}

\begin{figure}[h]
  \begin{center}
    \includegraphics[width=0.49\linewidth]{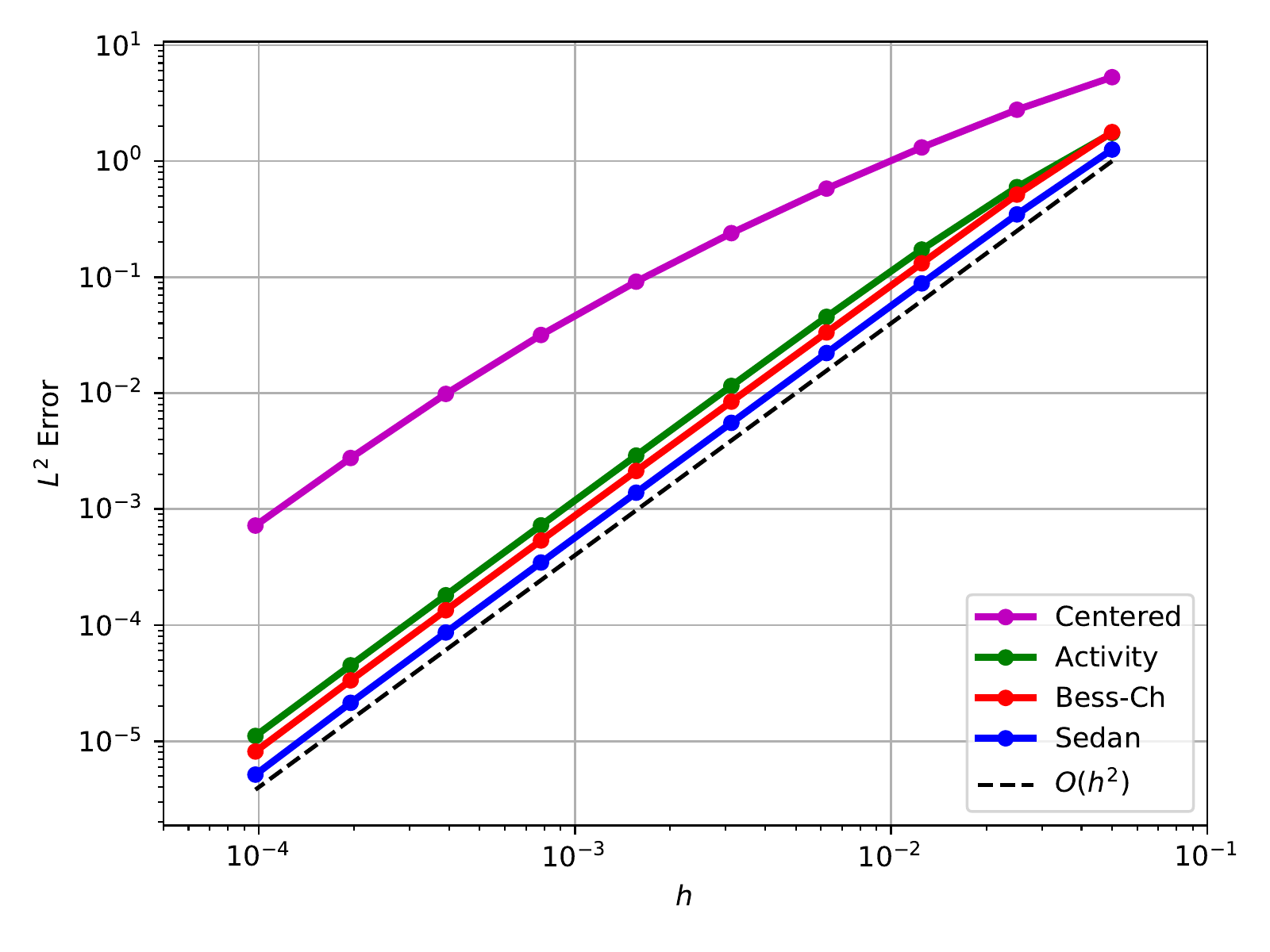}
    \includegraphics[width=0.49\linewidth]{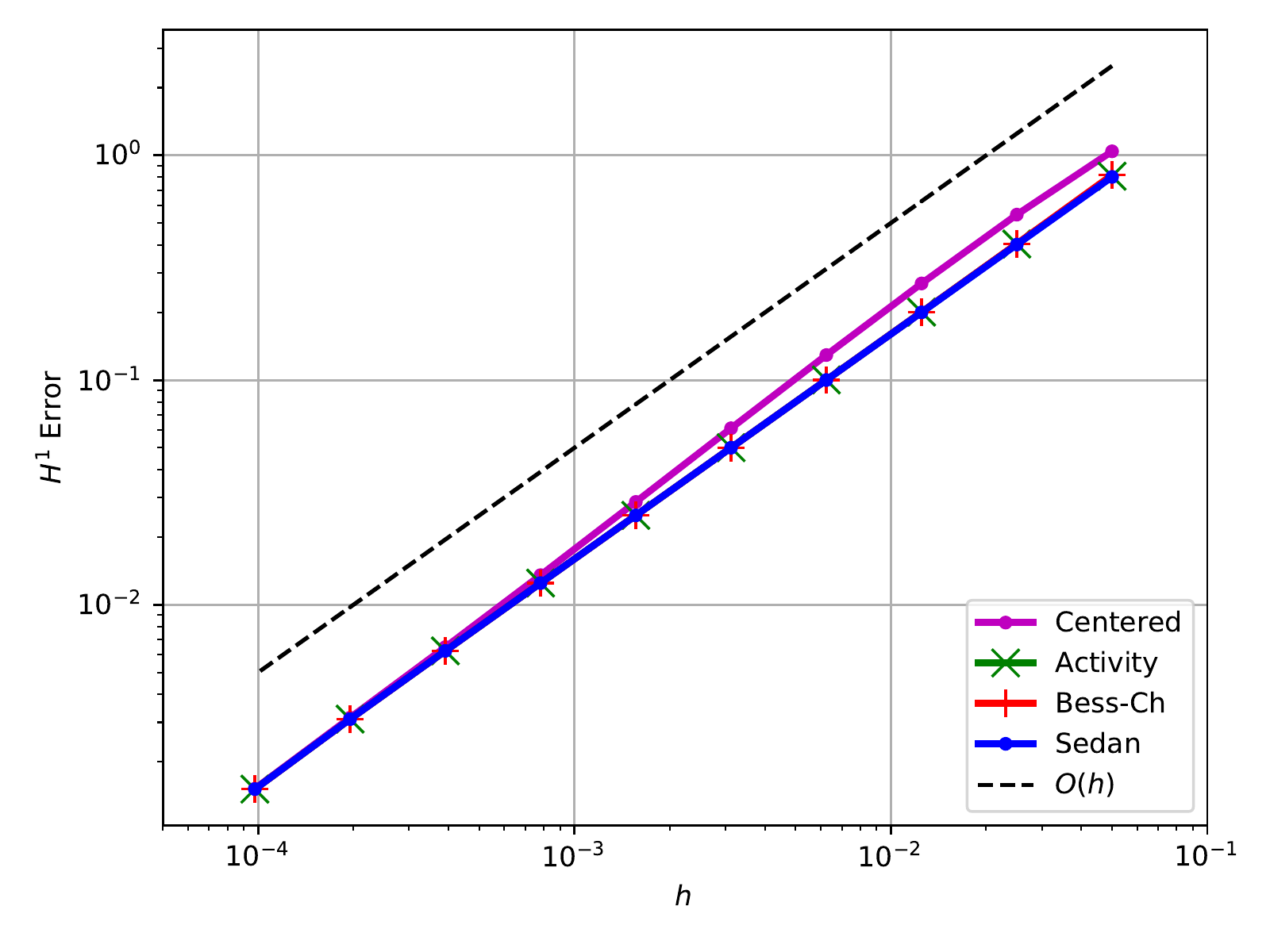}
    \caption{Convergence behavior of the different schemes.
      Left: $L^2$-error, right: $H^1$ error. Correspondence to the equation in the paper:
      ``Centered'': \eqref{eq:centered_flux}, ``Sedan'': \eqref{eq:SEDAN_flux}, ``Activity'': \eqref{eq:activity_flux}, ``Bess-Ch'': \eqref{eq:Marianne_flux}.
    }
    \label{fig:schemes}
  \end{center}
\end{figure}

In order to reveal the behavior of the various schemes under more extreme conditions,
this convergence test outside of thermodynamic equilibrium includes regions of the solution with concentrations extremely close to 0 and 1, respectively, enforced by inhomogeneous Dirichlet boundary conditions for the concentration,
thus leaving the realm of the analysis in this paper.
Once again, we assume $\Omega=(0,L)$ with $L=50$, $c^{dp}=-\frac12$. We set boundary values $\Phi(0)=\Phi(L)=0$
for the electrostatic potential, and $c(0)=10^{-3}, c(L)=1-10^{-3}$. We calculate a reference
solution using the scheme \eqref{eq:SEDAN_flux} on a fine grid of 40960 nodes with grid spacing $h\simeq1.22\cdot 10^{-3}$, see
Fig. \ref{fig:refsol}. We use this solution as a surrogate for an analytical solution in a numerical investigation
of the convergence rates of the different schemes. The result is shown in Fig. \ref{fig:schemes}.
We observe, that both in the $H^1$ and the $L^2$ norms, the schemes based on the modification
of the Scharfetter-Gummel idea behave significantly better than the centered scheme. 
This is probably due to the Dirichlet boundary condition close to $0$ where the  
function $c\mapsto h(c)$ appearing explicitly in the centered scheme is singular. 
Judging from the $L^2$ error plot in Fig. \ref{fig:schemes} (left), the  scheme \eqref{eq:SEDAN_flux} converges
better than  all the others. Asymptotically, all schemes show the same standard behavior: we observe
second order convergence in the $L^2$ norm and first order convergence in the $H^1$-norm.

\subsection{2D Unipolar Field Effect Transistor}

\begin{figure}
  \begin{center}
    \includegraphics[width=0.5\textwidth]{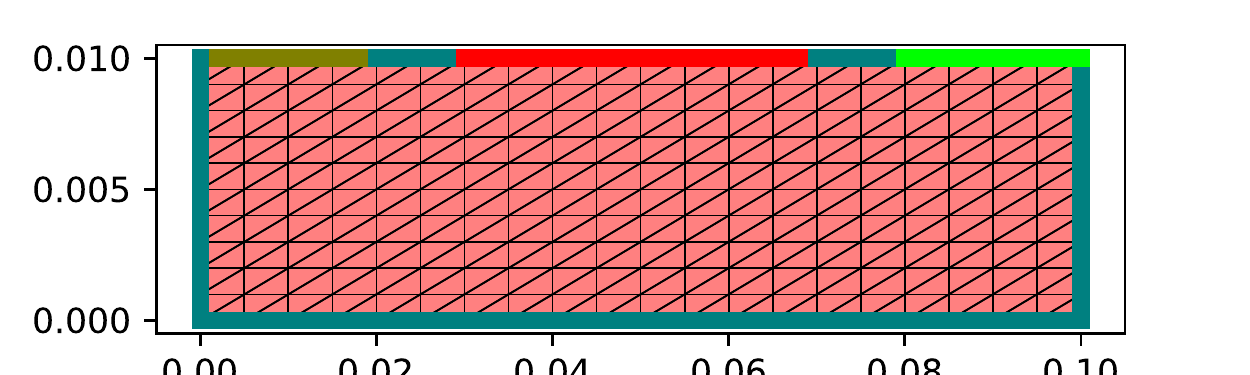}\hfill
    \includegraphics[width=0.5\textwidth]{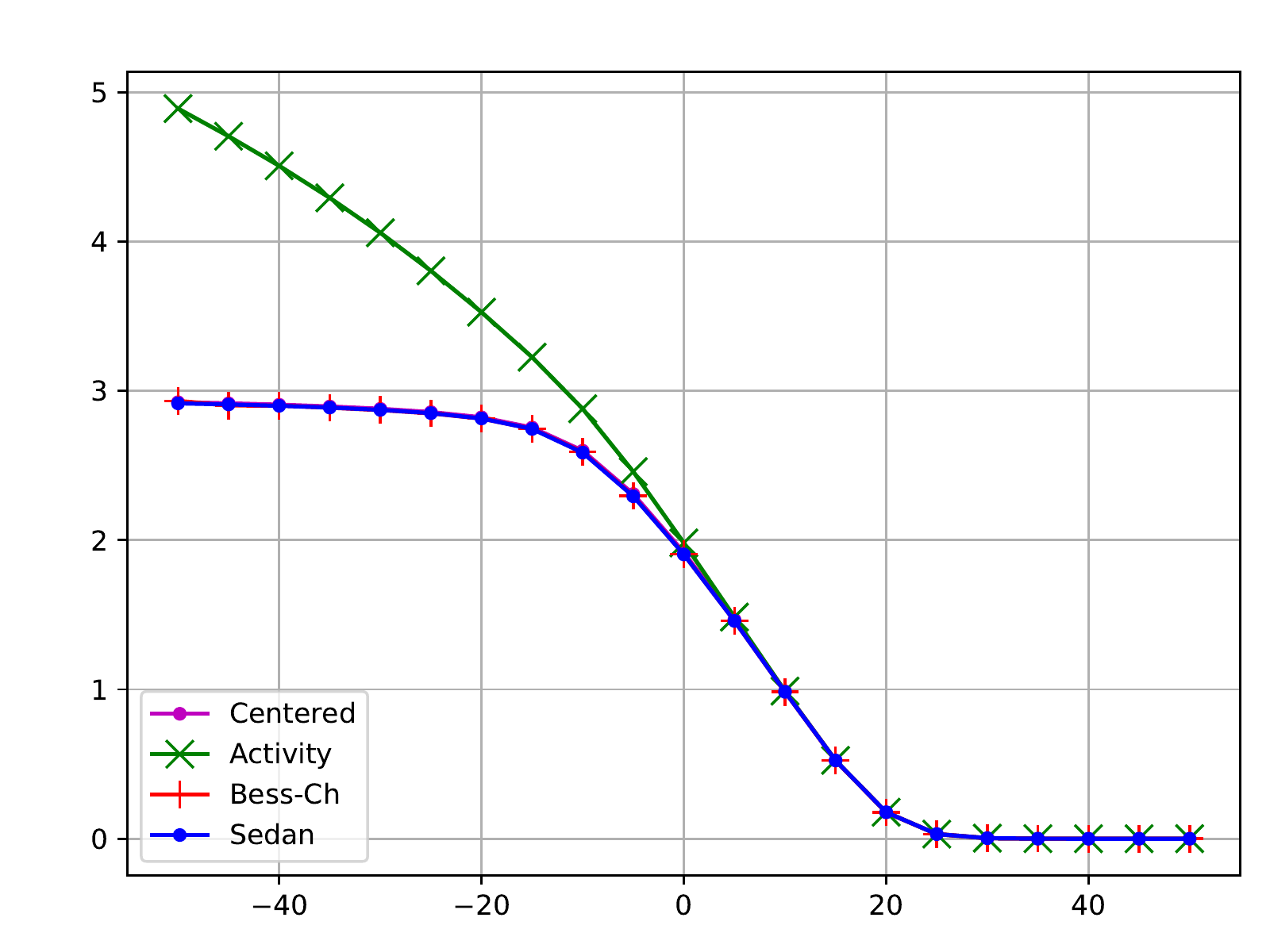}\\
  \end{center}
  \caption{Discretization grid of refinement level $n_{ref}=1$ (left) and  corresponding I-V curves for
    different discretization schemes (right).}
  \label{fig:gridiv}
\end{figure}

\begin{figure}

  \begin{center}
    \includegraphics[width=\textwidth]{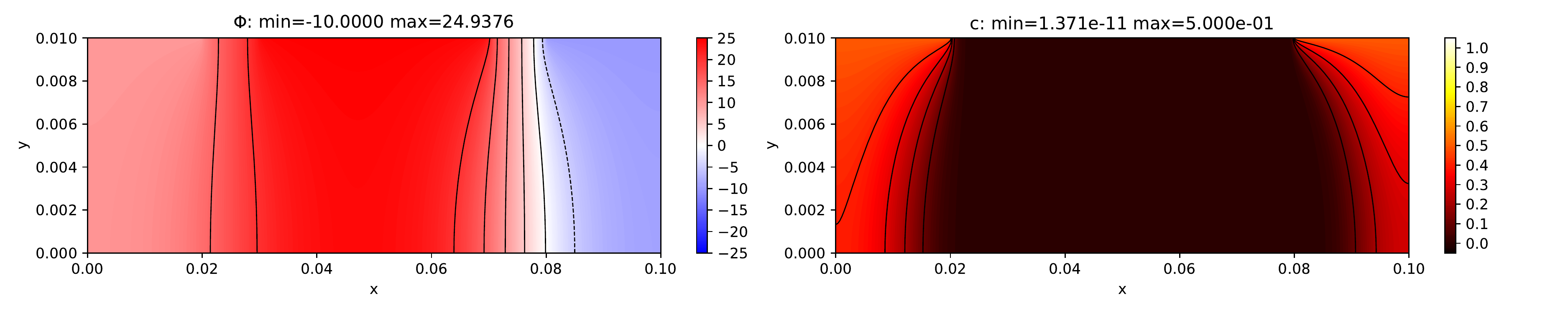}\hfill
  \end{center}
  \caption{Electrostatic potential (left) and concentration (right) for closed gate ($U_{gate}=50$).}
  \label{fig:closed}
\end{figure}

\begin{figure}

  \begin{center}
    \includegraphics[width=\textwidth]{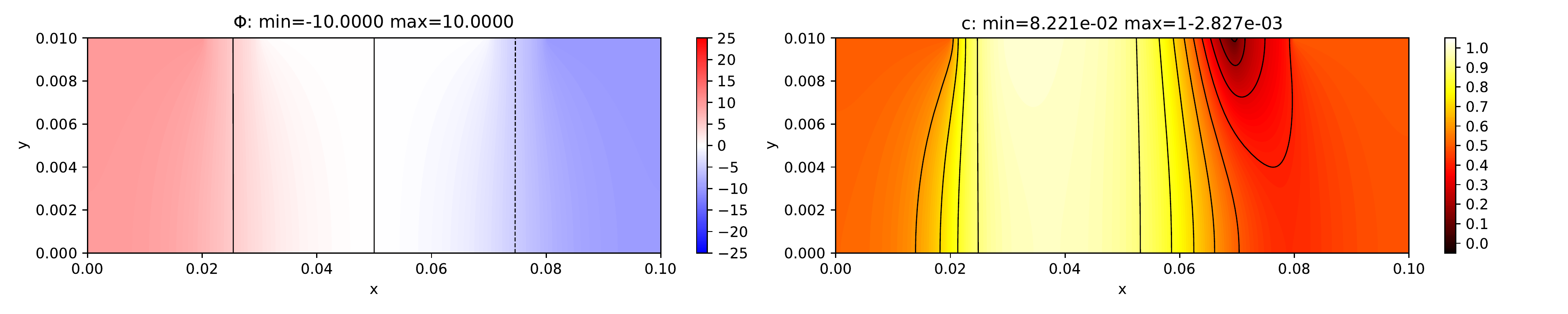}\hfill
  \end{center}
  \caption{Electrostatic potential (left) and concentration (right) for $U_{gate}=0$).}
  \label{fig:g0}
\end{figure}

\begin{figure}

  \begin{center}
    \includegraphics[width=\textwidth]{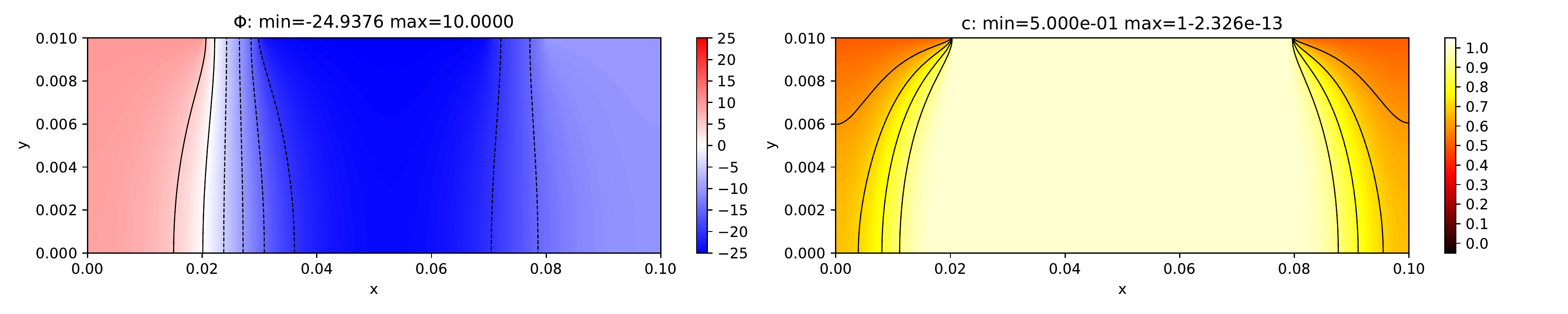}\hfill
  \end{center}
  \caption{Electrostatic potential (left) and concentration (right) for open gate ($U_{gate}=-50$), with
  concentration in the channel reaching the saturation value 1.}
  \label{fig:open}
\end{figure}

\begin{figure}
  \begin{center}
    \includegraphics[width=0.5\textwidth]{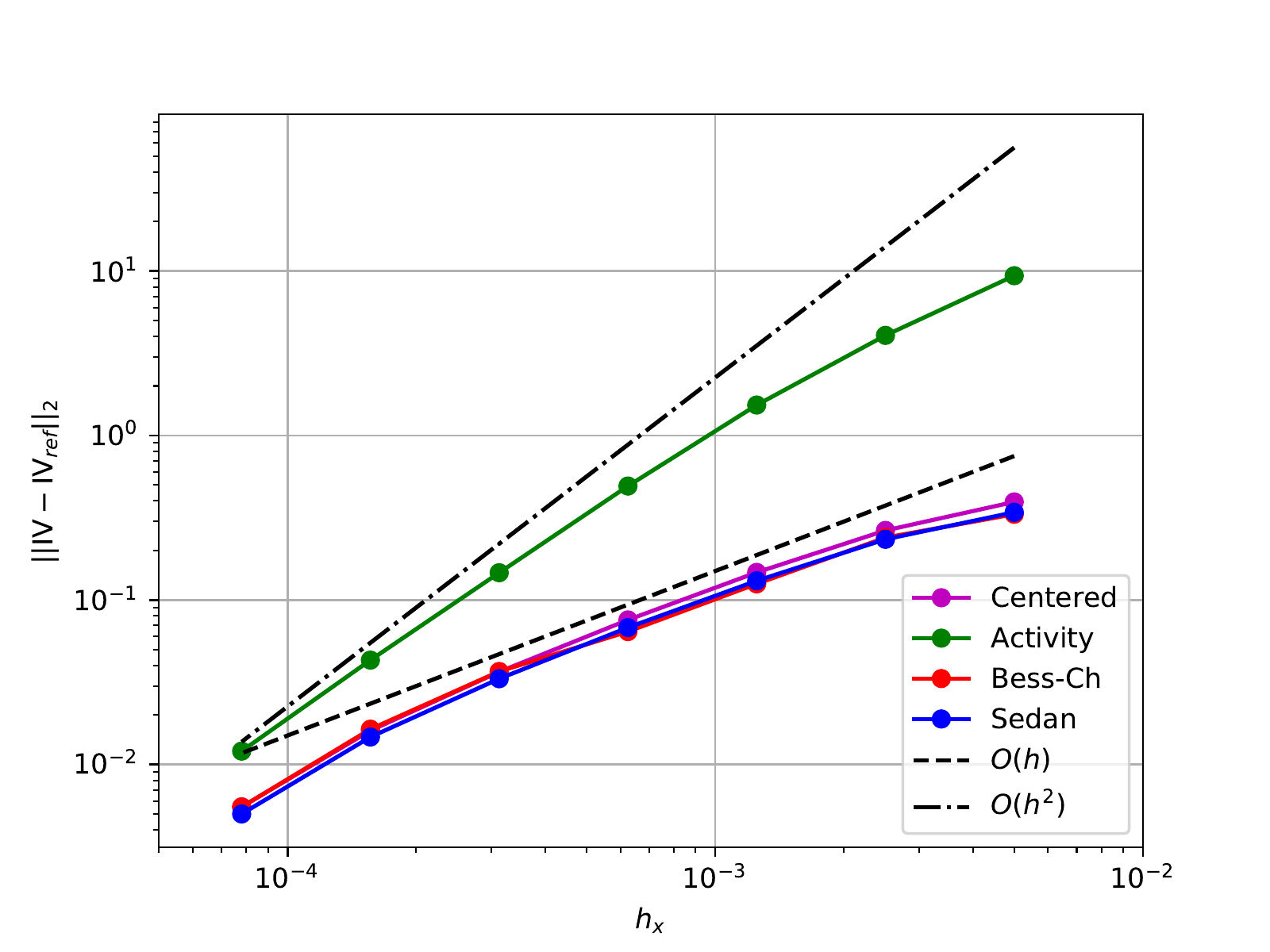}\hfill
  \end{center}
  \caption{ Convergence of the I-V curves calculated using the different discretization schemes.}
  \label{fig:erriv}
\end{figure}

As a second example, we consider an unipolar field effect transistor.
The domain is $\Omega=(0,L)\times(0,H)$ with $L=10^{-2}$, $H=10^{-3}$. We let $c^{dp}=-\frac12$, and set the following boundary conditions
at the contacts:
\begin{align*}
  \begin{pmatrix}
    \Phi\\c 
  \end{pmatrix} &=
  \begin{pmatrix}
    -5 \\ \frac12
  \end{pmatrix} \text{at} \;\Gamma_{source}= (0,0.2\cdot L)\times H\\
  \begin{pmatrix}
    \Phi\\c 
  \end{pmatrix} &=
  \begin{pmatrix}
    5 \\ \frac12
  \end{pmatrix} \text{at} \;\Gamma_{drain}= (0.8\cdot L,L)\times H\\
  \begin{pmatrix}
    \nabla\Phi\cdot \mathbf{n} \\ \mathbf J\cdot\mathbf{n}
  \end{pmatrix} &=
  \begin{pmatrix}
     - \frac{1}{d}(\Phi-U_{gate}) \\ 0
  \end{pmatrix} \text{at} \;\Gamma_{gate}= (0.3\cdot L, 0.7\cdot L)\times H\\
  \begin{pmatrix}
    \nabla\Phi\cdot \mathbf{n} \\ \mathbf J\cdot\mathbf{n}
  \end{pmatrix} &=
  \begin{pmatrix}
     0\\ 0
  \end{pmatrix} \text{at} \;\partial\Omega\setminus (\Gamma_{gate}\cup\Gamma_{source}\cup\Gamma_{drain}).
\end{align*}

Here, $\Phi_{gate}\in (-50,50)$ is the gate voltage, and $d=0.1\cdot H$ is the gate thickness.
We introduce a slightly anisotropic rectangular grid $n_x\times n_y$ grid with $n_x=10\times 2^{n_{ref}}$
and $n_y=5\times 2^{n_{ref}}$, where $n_{ref}$ is the refinement level. Each cell in the rectangular
grid is subdivided into two triangles, see Fig. \ref{fig:gridiv} (left). From the resulting triangle mesh, the Voronoi tesselation is
obtained.

With fixed source and drain voltages, we vary the gate voltage  $U_{gate}$ from 50 to -50. At
$U_{gate}=50$, the positive applied potential pushes away the positively charged carriers from
the channel -- the region under the gate contact, see Fig. \ref{fig:closed}. The resulting lack of charge carriers results in a near zero
current. With decreasing gate voltage, more and more charge carriers are allowed into the channel,
leading to an increase in the current. When the gate voltage decreases further, charge carriers
are attracted to the gate contact and fill up the channel. Due to the degeneration, their concentration cannot exceed 1. As
a result, we observe a saturation of the current close to some maximum value for gate voltages
approaching -50, see Fig. \ref{fig:closed}.

All  schemes under  consideration with the exception of  \eqref{eq:activity_flux}  represent this  saturation
behaviour   quite   well  already   at   rather   coarse  grids,   see
Fig. \ref{fig:gridiv} (right). This appears to be in line with earlier
investigations of  the scheme based  on activity averaging \cite{FKF17}  which hint
that  its  asymptotic  behavior  for  large  electric  fields  is  not
satisfactory.

In order get an idea about the convergence in this case, we
produce a reference solution on a grid with 821121 nodes using the
scheme \eqref{eq:SEDAN_flux} and compare the calculated I-V curves. The
behavior of the error in the I-V curves is shown in Fig. \ref{fig:erriv}.
While all four schemes exhibit convergence of order at least $O(h)$, the activity based scheme \eqref{eq:activity_flux}
converges with a constant approximately one order of magnitude larger
than the others.

\appendix

\section{$L^\infty$ bound on the TPFA FV approximate Poisson equation}\label{app:LinfPhi}

It is well know that the solution to the Poisson equation 
\be\label{eq:Poisson}
\begin{cases}
- \Delta u = f & \text{in}\;\O, \\
u = u^D & \text{on}\; \Gamma^D,\\
\grad u \cdot \n = 0 & \text{on}\;\Gamma^N,
\end{cases}
\ee
is bounded in $L^\infty(\O)$ provided $f \in L^\infty(\O)$ and $u^D \in L^\infty(\p\O)$. The goal 
of this appendix is to get a discrete counterpart of this estimate for TPFA finite volume 
approximations of~\eqref{eq:Poisson}. 
The data $u^D$ and $f$ are discretized into 
\[
u^D_\sig = \frac1{m_\sig}\int_\sig u^D(\bgamma) \d\bgamma, \qquad 
f_K = \frac1{m_K} \int_K f \d\x, \qquad \sig \in \Ee^D, \; K \in \Tt. 
\]
and the classical TPFA finite volume scheme writes:
$$
-\sum_{\sig\in \Ee_{K} }\tau_\sig D_{K\sig} u = m_K f_K,\quad \forall K\in\Tt. 
$$

The associate linear system of equations can be written as
\be\label{eq:Lu=b}
\bbL \bu = \boldsymbol{b}, 
\ee
with $\bu = (u_K,u_\sig)_{K\in\Tt,\sig\in\Ee^D}$ (let us note that we keep the Dirichlet nodes in the set of unknowns), 
$\boldsymbol{b} = \left(f_K,u^D_\sig\right)_{K\in\Tt,\sig\in\Ee^D}$ and 
$\bbL \in \R^{(\Tt+\Ee^D)\times(\Tt+\Ee^D)}$ is the sparse symmetric definite positive matrix defined by
\[
\bbL_{\sig,\sig} = 1, \qquad \bbL_{\sig,\ell} = 0 \;\text{if}\; \ell \neq \sig, \qquad \sig \in \Ee^D, 
\]
\[
\bbL_{K,K\sig} = -\frac{\tau_\sig}{m_K}, \qquad \bbL_{K,K} = \frac1{m_K}\sum_{\sig\in\Ee_K} \tau_\sig, \qquad K\in\Tt. 
\]
In the above definition of $\bbL$, $\ell$ denotes an arbitrary index in $\Tt\cap\Ee^D$, whereas $K\sig$ denotes the mirror index of $K$ 
w.r.t. the faces $\sig \in \Ee_K$, i.e., $K\sig = L$ if $\sig=K|L \in \Ee_K\cap\Ee_{\rm int}$ and $K\sig=\sig$ if $\sig\in\Ee_K\cap\Ee^D$.

The goal of this section is to derive an $\ell^\infty$ bound on the solution $\bu$ to the linear system~\eqref{eq:Lu=b}
which is uniform w.r.t. the mesh. This is the purpose of the following proposition.

\begin{prop}\label{prop:Linf_Poisson}
There exists $C$ depending only on $\O$ such that 
\[
|u_K| \leq C\max\left\{\|u^D\|_{L^\infty(\p\O)}, \|f\|_{L^\infty(\O)}\right\}, \qquad \forall K\in\Tt.
\]
\end{prop}
\begin{proof}
The proof we propose here is an extension to the context of TPFA Finite Volumes of the proof of Hackbusch~\cite{Hackbusch} 
for Finite Differences. 
An alternative proof of Proposition~\ref{prop:Linf_Poisson} based on Stampacchia's truncation estimates is sketched in~\cite{Gallouet_Tamtam}.

The definitions of $u^D_\sig$ and $f_K$ ensure that 
\[\|\boldsymbol{b}\|_\infty \leq \max\left\{\|u^D\|_{L^\infty(\p\O)}, \|f\|_{L^\infty(\O)}\right\},\]
so that 
\[
\|\bu\|_\infty \leq \|\bbL^{-1}\|_\infty \|\boldsymbol{b}\|_\infty\leq  \|\bbL^{-1}\|_\infty\max\left\{\|u^D\|_{L^\infty(\p\O)}, \|f\|_{L^\infty(\O)}\right\}.
\]
Therefore, it only remains to check that $\|\bbL^{-1}\|_\infty\leq C$ for some $C$ not depending on $\Tt$.

The matrix $\bbL$ is a $M$-matrix (see \cite[Definition 4.8]{Hackbusch}). Therefore, owing to \cite[Theorem 4.24]{Hackbusch}, 
if we can exhibit some vector $\bw \in \R^{\Tt+\Ee^D}$ such that $\bbL \bw \geq {\bf 1}$, then $\|\bbL^{-1}\|_\infty \leq \|\bw\|_\infty$.
Define the function $w:\ov \O \to \R$ by 
\[
w(\x) = 1 + \frac1d\left(\sup_{\boldsymbol y\in\O} |\boldsymbol y|^2 - |\x|^2\right) \geq 1, \qquad \x\in\ov\O, 
\]
and the vector $\bw=\left(w_K,w_\sig\right)$ by $w_K = w(\x_K)$, $K\in\Tt$, and $w_\sig = w(\x_\sig)$, $\sig\in\Ee^D$.

The estimate on the Dirichlet nodes is straightforward:
\[
\left(\bbL \bw\right)_\sig = w_\sig \geq 1, \qquad \forall \sig \in \Ee^D.
\]
Now, let us focus on the inner nodes $K\in\Tt$. Since $\sum_{\ell \in\Tt\cup\Ee^D} \bbL_{K,\ell} = \sum_{\sig\in\Ee_K} \bbL_{K,K\sig}= 0$, one has 
\begin{align*}
\left(\bbL \bw\right)_K =& \frac1d\sum_{\sig \in \Ee_K} \bbL_{K,K\sig} |\x_{K\sig}|^2
 =  \frac1{dm_K} \sum_{\sig \in \Ee_K}\tau_\sig\left(|\x_K|^2 - |\x_{K\sig}|^2\right) \\
  = & \frac1{dm_K} \sum_{\sig \in \Ee_K}\tau_\sig \left(|\x_{K} - \x_{K\sig}|^2 + 2 \x_K \cdot (\x_K - \x_{K\sig}) \right) \\
  = & \frac1{dm_K} \sum_{\sig \in \Ee_K} m_\sig d_\sig + \frac2{dm_K} \x_K \cdot \sum_{\sig\in\Ee_K} \tau_\sig (\x_K - \x_{K\sig}).
\end{align*}
Because of the geometric relation $m_\sig d_\sig = d m_{\Delta_\sig}$, and since $K \subset \bigcup_{\sig\in\Ee_K} \Delta_\sig$, 
there holds
\[
 \frac1{dm_K} \sum_{\sig \in \Ee_K} m_\sig d_\sig = \frac1{m_K}\sum_{\sig\in\Ee_K} m_{\Delta_\sig} \geq 1. 
\]
On the other hand, the second term vanishes since 
\[
\sum_{\sig\in\Ee_K} \tau_\sig (\x_K - \x_{K\sig}) = -\sum_{\sig\in\Ee_K} m_\sig \n_{K\sig} = {\bf 0}.
\]
Therefore, $\left(\bbL \bw\right)_K\geq 1$ for all $K\in\Tt$. As a consequence, 
\[
\|\bbL^{-1}\|_\infty \leq \|\bw\|_\infty = 1 + \frac1d\sup_{\boldsymbol y\in\O} |\boldsymbol y|^2 \leq 1 + \frac{{\rm diam}(\O)^2}{4d}.
\]
The last estimate comes from the fact that one can choose the origin for $\boldsymbol y$ arbitrarily.
\end{proof}
\section{Proof of Lemma~\ref{lem:dissip}}\label{app:dissip}

{\bf Step 1.} Let $\delta\in (0,1)$ and $M\in\R$.
We start with the proof of 
\be\label{lim_Upsilon}
\ds\lim_{c_L\to 1}\Upsilon_{\delta,M}(c_L)=+\infty
\ee
where 
$$
{\Upsilon}_{\delta,M}(c_L)=\inf \left\{\Dd(c_K,c_L, \Phi_K, \Phi_L);  c_K\in (0, 1-\delta], (\Phi_K,\Phi_L) \in [-M,M]^2\right\}.
$$
We recall that 
$$
\begin{aligned} 
{\mathcal D} (c_K,c_L,\Phi_K,\Phi_L)&={\mathcal C} (c_K,c_L,\Phi_K,\Phi_L)\vert h(c_K)+\Phi_K-h(c_L)-\Phi_L\vert^2,\\
&={\mathcal F} (c_K,c_L,\Phi_K,\Phi_L)(h(c_K)+\Phi_K-h(c_L)-\Phi_L)
\end{aligned}
$$
and we notice that the diffusion force blows up:
\be\label{eq:diff-force-blowup}
\ds\lim_{c_L\to 1} \inf\left\{\Bigl|h(c_K)-h(c_{L})+ \Phi_K- \Phi_{L}\Bigl| ; c_K\in (0,1-\delta], (\Phi_K,\Phi_L)\in [-M,M]^2\right\}=+\infty.
\ee
Therefore, we can get \eqref{lim_Upsilon} by proving that  either ${\mathcal C} (c_K,c_L,\Phi_K,\Phi_L)$ or ${\mathcal F} (c_K,c_L,\Phi_K,\Phi_L)$ stays  bounded away from 0, uniformly in $c_K\in (0,1-\delta]$, $(\Phi_K,\Phi_L)\in [-M,M]^2$, for $c_L\geq 1/2$. 

For the \underline{centered flux}, we have that,  for all $(c_K,\Phi_K,\Phi_L)\in (0,1-\delta]\times [-M,M]^2$, $\Cc (c_K,c_L,\Phi_K,\Phi_L)=(c_K+c_L)/2\geq c_L/2$. This yields \eqref{lim_Upsilon}. 
For the three other schemes, we remark that, for any $\alpha\in (0,1-\delta)$, we can rewrite
$$
{\Upsilon}_{\delta,M}(c_L)=\min ({\Upsilon}_{\delta,M}^{\alpha,1}(c_L), {\Upsilon}_{\delta,M}^{\alpha,2}(c_L)),
$$
where 
$$
\begin{aligned}
{\Upsilon}_{\delta,M}^{\alpha,1}(c_L)&=\inf \left\{\Dd(c_K,c_L, \Phi_K, \Phi_L);  c_K\in (0, \alpha), (\Phi_K,\Phi_L) \in [-M,M]^2\right\}\\
{\Upsilon}_{\delta,M}^{\alpha,2}(c_L)&=\inf \left\{\Dd(c_K,c_L, \Phi_K, \Phi_L);  c_K\in [\alpha, 1-\delta], (\Phi_K,\Phi_L) \in [-M,M]^2\right\}
\end{aligned}
$$
But Lemma \ref{lem:avg} ensures that, independently of the choice of the numerical flux,  we have at least $\Cc (c_K,c_L,\Phi_K,\Phi_L)\geq \min(c_K,c_L)/2$, so that 
$\Cc (c_K,c_L,\Phi_K,\Phi_L)\geq \alpha/2$ for all $(c_K,c_L,\Phi_K,\Phi_L)\in [\alpha, 1-\delta]\times [1/2,1)\times [-M,M]^2$ if 
$\alpha {\in (0,1-\delta).}$
Therefore, for all $\alpha {\in (0,1-\delta)}$, 
we have 
$$
\ds\lim_{c_L\to 1} {\Upsilon}_{\delta,M}^{\alpha,2}(c_L)=+\infty.
$$
It remains to prove that for a given $\alpha {\in (0,1-\delta)}$
we also have 
\begin{equation}\label{lim_Upsilon_alpha}
\ds\lim_{c_L\to 1} {\Upsilon}_{\delta,M}^{\alpha,1}(c_L)=+\infty.
\end{equation}
{Because of the monotonicity of $\delta \mapsto \Upsilon_{\delta, M}(c_L)$, we can restrict our attention to the 
case $\delta \leq 1/2$, so that we can seek for $\alpha \in (0,1/2]$.}
\smallskip

For the \underline{Bessemoulin-Chatard flux}, we have 
$$
{\mathcal F}(c_K,c_L,\Phi_K,\Phi_L)=dr(c_K,c_L)
\left\{B\left(\frac{\Phi_L-\Phi_K}{dr(c_K,c_L)}\right)c_K-B\left(\frac{\Phi_K-\Phi_L}{dr(c_K,c_L)}\right)c_{L}\right\},
$$
with $dr(c_K,c_L)\geq 1$. Using the monotonicity of the Bernoulli function and the bounds on $\Phi_K$ and $\Phi_L$, we get:
\[B(2M)\leqslant B\left(\pm \frac{\Phi_L-\Phi_K}{dr(c_K,c_L)}\right)\leqslant B(-2M)\]
Hence, for $\alpha=\frac{B(2M)}{4B(-2M)}$:
\[{\mathcal F}(c_K,c_L,\Phi_K,\Phi_L)\leqslant dr(c_K,c_L)
\left\{B(-2M)\alpha-B(2M)\frac12\right\}\leqslant -\frac{B(2M)}{4}
\]
Then, thanks to \eqref{eq:diff-force-blowup}, we deduce \eqref{lim_Upsilon_alpha} for $\alpha=\frac{B(2M)}{4B(-2M)}$ and therefore \eqref{lim_Upsilon}.
\smallskip

For the \underline{Sedan flux}, we use similarly the monotonicity of the function $B$ and $\nu$, so that 
\begin{multline*}
{\mathcal F}(c_K,c_L,\Phi_K,\Phi_L)\leq B\left(-2M+\nu(\frac{1}{2})-\nu(\alpha)\right)\alpha -B\left(2M-\nu(\frac{1}{2})+\nu(\alpha)\right)\frac12\\
\forall c_K\in (0,\alpha), c_L\in (\frac12, 1), (\Phi_K,\Phi_L)\in [-M,M]^2.
\end{multline*}
But the right-hand-side of the last inequality has a negative limit when $\alpha$ tends to $0$, which means that for a given $\alpha$ 
{small enough} it remains bounded away from 0, so that we deduce \eqref{lim_Upsilon_alpha} and therefore \eqref{lim_Upsilon}.

\smallskip

For the \underline{activity based flux}, we also use the monotonicity of $a$ and $\beta$, which yields
\begin{multline*}
{\mathcal F}(c_K,c_L,\Phi_K,\Phi_L)\leq \frac14\left( B(-2M) a(\alpha)-B(2M) a(\frac12)\right)\\
\forall c_K\in (0,\alpha), c_L\in (\frac12, 1), (\Phi_K,\Phi_L)\in [-M,M]^2.
\end{multline*}
The right-hand-side has a negative limit when $\alpha$ tends to 0. Thus it remains bounded away from 0 for a given $\alpha<1/2$ and we deduce \eqref{lim_Upsilon_alpha} and therefore \eqref{lim_Upsilon}.

{\bf Step 2.} We now focus on the proof of
\be\label{lim_Psi}
\ds\lim_{c_L\to 0}\Psi_{\delta,M}(c_L)=+\infty
\ee
where 
$$
{\Psi}_{\delta,M}(c_L)=\inf \left\{\Dd(c_K,c_L, \Phi_K, \Phi_L);  c_K\in [\delta, 1), (\Phi_K,\Phi_L) \in [-M,M]^2\right\}.
$$

We use similar arguments than in Step 1.  First, the diffusion force still blows up :
\be\label{eq:diff-force-blowup2}
\ds\lim_{c_L\to 0} \inf\left\{\Bigl|h(c_K)-h(c_{L})+ \Phi_K- \Phi_{L}\Bigl| ; c_K\in [\delta, 1), (\Phi_K,\Phi_L)\in [-M,M]^2\right\}=+\infty.
\ee
For the \underline{centered flux}, we have : $\Cc (c_K,c_L,\Phi_K,\Phi_L)=(c_K+c_L)/2\geq \delta/2$ hence \eqref{lim_Psi}.
For the other fluxes, we rewrite again
$$
{\Psi}_{\delta,M}(c_L)=\min ({\Psi}_{\delta,M}^{\alpha,1}(c_L), {\Psi}_{\delta,M}^{\alpha,2}(c_L)),
$$
where 
$$
\begin{aligned}
{\Psi}_{\delta,M}^{\alpha,1}(c_L)&=\inf \left\{\Dd(c_K,c_L, \Phi_K, \Phi_L);  c_K\in [\delta, \alpha], (\Phi_K,\Phi_L) \in [-M,M]^2\right\};\\
{\Psi}_{\delta,M}^{\alpha,2}(c_L)&=\inf \left\{\Dd(c_K,c_L, \Phi_K, \Phi_L);  c_K\in (\alpha, 1), (\Phi_K,\Phi_L) \in [-M,M]^2\right\}.
\end{aligned}
$$

Using the symmetry of the flux ${\mathcal F}(c_K,c_L,\Phi_K,\Phi_L)=-{\mathcal F}(c_L,c_K,\Phi_L,\Phi_K)$  and following the proof of \eqref{lim_Upsilon_alpha}, we get that for $\alpha=1/2$,
\[
\lim_{c_L\to 0} {\Psi}_{\delta,M}^{\alpha,2}(c_L)=+\infty
\]
We now have to prove that, for $\alpha=1/2$, 
\be\label{lim_Psi_alpha}
\lim_{c_L\to 0} {\Psi}_{\delta,M}^{\alpha,1}(c_L)=+\infty
\ee
To this end, we will show bounds on the flux. The set $[\delta, \alpha]\times [-M,M]^2$ is compact, and the flux functions are continuous. It is sufficient to show a positive lower bound {for} the limit at any $(c^*,\Phi^*,\Phi_*)\in [\delta, \alpha]\times [-M,M]^2$:
\[l^*=\quad\lim_{\mathclap{(c_K,c_L,\Phi_K,\Phi_L)\to (c^*,0,\Phi^*,\Phi_*)}}\quad{\mathcal F}(c_K,c_L,\Phi_K,\Phi_L) {>0}.\]

For the Sedan scheme, we have:
\[
l^*= B\left(\Phi_*-\Phi^* - \nu(c^*)) \right) c^*\geq \delta B(2M).
\]
For the Bessemoulin-Chatard scheme, we have:
$
\ds\lim_{(c_K,c_L)\to (c^*,0)} dr(c_K,c_L)={1},
$
hence:
\[
l^*=B(\Phi_*-\Phi^*)c^*\geq \delta B(2M).
\]
For the activity based scheme we have:
\[
l^*= \frac{\beta(c^*)+1}{2}
B(\Phi_*-\Phi^*)a(c^*)\geq \frac{a(\delta)}{2}B(2M).
\]
As these limits are bounded away from zero we have \eqref{lim_Psi_alpha} hence \eqref{lim_Psi}. It concludes the proof of Lemma~\ref{lem:dissip}.

\section{Comparison of face concentration functionals}\label{app:alacon}

For each scheme, we have defined a face concentration functional $\Cc : (0,1)\times(0,1)\times\R\times\R\to \R$. 
We introduce a second face concentration functional $\wt\Cc: (0,1)\times (0,1)\to\R$ , defined by 
$$
\wt\Cc (c_K,c_L)=\ds\frac{r(c_K)-r(c_L)}{h(c_K)-h(c_L)} \mbox{ if } c_K\neq c_L\mbox{ and } c_K \mbox{ otherwise}.
$$
As $r'(c)=ch'(c)$, it is clear that 
\be\label{eq:wt_avg}
\min(c_K,c_L)\leq \wt\Cc (c_K,c_L)\leq \max(c_K,c_L)\mbox{  for all }(c_K,c_L)\in (0,1)\times (0,1). 
\ee
Lemma \ref{lem:cctilde} states a comparison between $\Cc$ and $\wt\Cc$ for the centered and the Sedan schemes.
\begin{lemma}\label{lem:cctilde}
For the centered scheme and the Sedan scheme, there exists $G>0$, depending only on $M$, such that for all $(c_K,c_L,\Phi_K,\Phi_L)\in (0,1)\times(0,1)\times[-M,M]\times[-M,M]$,
\be\label{est:cctilde}
\ds\frac{\wt\Cc(c_K,c_L)}{\Cc(c_K,c_L,\Phi_K,\Phi_L)}\leq G.
\ee  
\end{lemma}

\begin{proof}
The case of the centered scheme defined by \eqref{eq:centered_flux} is the easiest one, since 
$$
\Cc(c_K,c_L,\Phi_K,\Phi_L)=\frac{c_K + c_{L}}2\geq \frac12 \max(c_K,c_{L}),
$$
so that \eqref{est:cctilde} holds with $G=2$ thanks to~\eqref{eq:wt_avg}.

Let us now focus on the Sedan scheme defined by \eqref{eq:SEDAN_flux}. We can introduce the function ${\mathcal G} : (0,1)\times(0,1)\times\R\times\R\to \R$ defined by 
$$
{\mathcal G}(c_K,c_L,\Phi_K,\Phi_L)=\ds\frac{\wt\Cc(c_K,c_L)}{\Cc(c_K,c_L,\Phi_K,\Phi_L)}.
$$
It is a continuous function which satisfies the symmetry property $ \Gg(c_K,c_{L},\Phi_K,\Phi_{L}) = \Gg(c_{L},c_{K},\Phi_{L},\Phi_{K})$ and the consistency $ \Gg(c_K,c_{K},\Phi_K,\Phi_{L})=1$.
Bearing in mind 
the expression~\eqref{eq:B_avg} of $\Cc$, with $x=\log(c_K)-\log(c_L)$, $y= \Phi_L+\nu(c_L)-\Phi_K-\nu(c_K)$ and $x-y=h(c_K)+\Phi_K-h(c_L)-\Phi_L$, one can rewrite 
\[
\Gg(c_K,c_{L},\Phi_K,\Phi_L) =\frac{(h(c_K)+\Phi_K-h(c_L)-\Phi_L)(r(c_K)-r(c_L))}{\big(B(y)c_K - B(-y) c_L\big)(h(c_K) - h(c_L))}
\]
Because of the symmetry and the consistency properties,
we can assume without loss of generality that $c_K> c_L$ .
Using the average properties~\eqref{eq:avg} and~\eqref{eq:wt_avg}, one obtains that 
\be\label{eq:angle}
\Gg(c_K,c_{L},\Phi_K,\Phi_{L}) \leq \frac{c_K}{c_L} \leq \frac1{c_L}, 
\ee
so that we only have to check that $\Gg(c_K,c_{L},\Phi_K,\Phi_{L})$ remains uniformly bounded as $c_L$ tends to $0$ in order to prove \eqref{est:cctilde}. Therefore, considering that $(\Phi_K,\Phi_L)$ is given, we study the limit of $\Gg$ when $(c_K,c_L)$ tends to (1,0), (0,0) and $(c^\ast, 0)$ with $c^\ast\in (0,1)$. 

We first consider the limit $(c_K,c_L)\to (1,0)$. We have the following equivalences when $(c_K,c_L)\to (1,0)$:
$$
\begin{gathered}
h(c_K)-h(c_L)\sim -\log(1-c_K)-\log(c_L)\\
h(c_K)+\Phi_K-h(c_L)-\Phi_L\sim -\log(1-c_K)-\log(c_L)\\
r(c_K)-r(c_L)\sim -\log(1-c_K)\\
B(y)c_K - B(-y) c_L\sim -c_K\log(1-c_K)
\end{gathered}
$$
This yields:
\be\label{lim1}
\ds\lim_{(c_K,c_L)\to (1,0)}\Gg(c_K,c_{L},\Phi_K,\Phi_{L})=1 \quad \forall (\Phi_K,\Phi_L)\in\R\times\R.
\ee
With similar arguments, we compute the limit $(c_K,c_L)\to (c^\ast,0)$ with $c^\ast\in(0,1)$. We get:
\be\label{lim2}
\ds\lim_{(c_K,c_L)\to (c^\ast,0)}\Gg(c_K,c_{L},\Phi_K,\Phi_{L})=\frac{r(c^\ast)}{B(y^\ast)c^\ast} \quad \forall (\Phi_K,\Phi_L)\in\R\times\R,
\ee
with $y^\ast=\Phi_L-\Phi_K+\log(1-c^{\ast})$.

In the neighborhood of $(0,0)$, the behaviour is more complex, as the limit of $\log(c_K/c_L)$ is not defined and $\Gg$ does not have a limit. However, thanks to \ref{eq:angle},  $\Gg(c_K,c_{L},\Phi_K,\Phi_{L})$ stays bounded if $c_K/c_L$ stays bounded while $(c_K,c_L)\to (0,0)$.
It remains to consider the case where $(c_K,c_L)\to (0,0)$ while $c_K/c_L\to \infty$. In this case, 
$$
\begin{gathered}
\frac{(h(c_K)+\Phi_K-h(c_L)-\Phi_L)}{h(c_K) - h(c_L)}\to 1,\\
r(c_K)-r(c_L)\sim -c_L+c_K\\
B(y)c_K - B(-y) c_L\sim B(\Phi_L-\Phi_K)c_K-B(\Phi_K-\Phi_L)c_L,
\end{gathered}
$$
and 
$$
\ds\lim_{\stackrel{(c_K,c_L)\to (0,0)}{c_K/c_L\to \infty}}\Gg(c_K,c_{L},\Phi_K,\Phi_{L})=\frac{1}{B(\Phi_L-\Phi_K)}.
$$
We conclude that $\Gg(c_K,c_{L},\Phi_K,\Phi_{L})$ stays bounded  when $(c_K,c_L)$ is in the neighborhood of $(0,0)$ for all $(\Phi_K,\Phi_L)\in \R^2$.
Combined with \eqref{eq:angle}, \eqref{lim1} and \eqref{lim2}, it concludes the proof of Lemma \ref{lem:cctilde}.
\end{proof}

\begin{remark}
For the Bessemoulin-Chatard scheme, the bound \eqref{est:cctilde} does not hold. Let us consider that $\Phi_K=\Phi_L=\Phi$, then with the notations $x=\log(c_K/c_L)$, and $y=\log(\frac{1-c_L}{1-c_K})$, we have: 
\[
\frac{\tilde{\Cc}(c_K,c_L)}{\Cc(c_K,c_L,\Phi,\Phi)}=\frac{xy}{(c_K-c_L)(x+y)}.
\]
For $(c_K,c_L)\to(1,0)$, $x$ and $y$ tends to $+\infty$, hence the blow up of the ratio. 
\end{remark}

\end{document}